\documentclass{amsart}

\usepackage{bm}
\usepackage{graphicx}
\usepackage{enumerate}
\usepackage{latexsym}   
\usepackage{amssymb}    
\usepackage{amsmath}    
\usepackage{amsthm}     
\usepackage{hyperref}
\usepackage{color}
\usepackage[cmtip,all,matrix,arrow,tips,curve]{xy}

\newtheorem{theorem}{Theorem}[section]
\newtheorem{proposition}[theorem]{Proposition}

\newtheorem{lemma}[theorem]{Lemma}

\theoremstyle{definition}
\newtheorem{definition}[theorem]{Definition}

\theoremstyle{remark}
\newtheorem{example}[theorem]{Example}
\newtheorem{remark}[theorem]{Remark}

\newcommand{\GL}{\operatorname{GL}}
\newcommand{\SL}{\operatorname{SL}}
\newcommand{\Sp}{\operatorname{Sp}}
\newcommand{\GSp}{\operatorname{GSp}}
\newcommand{\Sl}{\operatorname{SL}}

\newcommand{\G}{\mathbb{G}}
\newcommand{\R}{\mathbb{R}}
\newcommand{\Q}{\mathbb{Q}}
\newcommand{\F}{\mathbb{F}}
\newcommand{\ri}{\mathcal O}
 
\newcommand{\Z}{\mathbb Z}
\newcommand{\A}{\mathbb A}
\newcommand{\C}{\mathbb C}
\newcommand{\bT}{\mathbf T}
\newcommand{\bV}{\mathbf V}
\newcommand{\cT}{\mathcal T}
\newcommand{\bG}{\mathbf G}
\newcommand{\bZ}{\mathbf Z}

\newcommand{\vol}{\operatorname{vol}}
\newcommand{\ord}{\operatorname{ord}}
\newcommand{\tr}{\operatorname{Tr}}
\newcommand{\Tr}{\operatorname{Tr}}
\newcommand{\Ad}{\operatorname{Ad}}
\newcommand{\Lie}{\operatorname{Lie}}
\newcommand{\ad}{\operatorname{ad}}

\newcommand{\fg}{\mathfrak {g}}

\newcommand{\ft}{\mathfrak {t}}
\newcommand{\fc}{\mathfrak {c}}
\newcommand{\fsl}{\mathfrak {sl}}
\newcommand{\fgl}{\mathfrak {gl}}
\newcommand{\fsp}{\mathfrak {sp}}
\newcommand{\fso}{\mathfrak {so}}

\newcommand{\spec}{\operatorname{Spec}}
\newcommand{\rank}{\operatorname{Rank}}
\newcommand{\diag}{\operatorname{diag}}
\newcommand{\geom}{{\operatorname{geom}}}
\newcommand{\charp}{{\operatorname{char.p.}}}
\newcommand{\res}{\operatorname{Res}}

\newcommand{\rk}{\operatorname{rank}}
\newcommand{\der}{{\operatorname{der}}}
\newcommand{\Jac}{{\operatorname{Jac}}}
\newcommand{\Gal}{\operatorname{Gal}}
\newcommand{\fr}{\operatorname{Fr}}
\newcommand{\can}{{\operatorname{can}}}
\newcommand{\fin}{{\operatorname{fin}}}
\newcommand{\spl}{{\operatorname{spl}}}
\newcommand{\rss}{{\operatorname{rss}}}

\newcommand{\sep}{{\operatorname{sep}}}

\newcommand{\SO}{\mathrm {SO}}
\newcommand{\be}{\mathbf {e}}
\newcommand{\bff}{\mathbf {f}}
\newcommand{\bh}{\mathbf {h}}

\title{Orbital integrals and normalizations of measures}
\author{Julia Gordon}
\address{Department of Mathematics \\
University of British Columbia\\
121-1984 Mathematics Rd., Vancouver BC V6T 1Z2, Canada\\
gor@math.ubc.ca}

\begin{document}

\maketitle

\begin{abstract}
This note provides an informal introduction, with examples, to some technical aspects of the re-normalization of 
measures on orbital integrals used in the work of Langlands, Frenkel-Langlands-Ng\^o, and Altug on Beyond Endoscopy. 
In particular, we survey different relevant measures on algebraic tori and  
explain the connection with the Tamagawa numbers. We work out the example of $\GL_2$ in complete detail. 
The Appendix by Matthew Koster illustrates, for the Lie algebras $\fsl_2$ and $\fso_3$, the relation between the 
so-called geometric measure on the orbits and Kirillov's measure on 
co-adjoint orbits in the linear dual of the Lie algebra. 
\end{abstract}

\section{Introduction}
Haar measure on a locally compact topological group is unique up to a constant. In many situations the normalization matters (as we shall see below). In particular, it seems that some measures are more convenient than others in the approach to Beyond Endoscopy by Frenkel-Langlands-Ng\^o, Arthur, and Altug. The main goal of this note is to track some of the normalizations of the measures that arise in the literature on orbital integrals on reductive groups over non-Archimedean local fields, and provide an introduction to this aspect of Altug's lectures.  
In particular, our first goal is to give an exposition of the formula (3.31) in \cite{langlands-frenkel-ngo}, which is the first technical step towards Poisson summation. 
The second goal is to summarize the relationship between measures on $p$-adic manifolds, point-counts over the residue field,  and local $L$-functions. These relations are scattered over the literature, and the aim here is to collect the references in one place, and provide some examples.

Fundamentally, there are two approaches to choosing a normalization of a Haar measure on the set of $F$-points of an algebraic group 
for a local field $F$: one can consider a measure associated with a specified differential form; or one can choose a specific compact subgroup and  prescribe the volume of that subgroup. As we shall see, both approaches have certain advantages, and converting between these two kinds of normalizations can be surprisingly tricky. All the objects we consider  here will be affine algebraic varieties, and we will only consider algebraic differential forms.  To avoid confusion, we will try to consistently denote varieties by bold letters, while various sets of rational points will be denotes by letters in the usual font.

In this context, we start, in \S \ref{sec:vf}, with a quick survey of A.Weil's definition of the measure on the set $\bV(F)$ of $F$-points of an affine variety $\bV$ 
associated with an algebraic  volume form on $\bV$. We discuss the relation between this measure and counting rational points of $\bV$ over the residue field, and the results 
on the various measures on $\bT(F)$ for an algebraic torus $\bT$ that follow. Next in \S \ref{sec:oi}, we compare 
two natural measures on  the orbits of the adjoint action of an algebraic group $\bG(F)$ on itself. This comparison is the main reason for writing this note. More specifically, we introduce Steinberg map and derive the relationship between two measures on the orbits: 
the so-called \emph{geometric} measure,  obtained by considering the stable orbits as fibres of Steinberg map,  and the measure obtained as a quotient of two natural measures coming from volume forms. In \S \ref{sec:canon}, we combine the outcome with the results 
of \S\ref{sec:vf} to obtain the relationship between the geometric measure and the so-called \emph{canonical} measure (which is the one most frequently used to define orbital integrals). We do the $\GL_2$ example in detail. 
Finally, in \S \ref{sec:global}, we assemble the local results into a global calculation, first, in the context of the analytic class number formula, and then in the case of Eichler-Selberg Trace Formula.   

\section*{Acknowledgment and disclaimer.} These notes would not have been possible without many conversations with W. Casselman over the years; in particular, among many other insights, I thank him for pointing out the key point of \S\ref{subsub:general}. I learned most of the material presented in these notes while working on a seemingly unrelated project with 
Jeff Achter, S. Ali Altug, and Luis Garcia.  I am very grateful to these people. 
One of the things that surprised me during that work was the difficulty of doing calculations with Haar measure and tracking its normalizations in the literature. The goal of these notes is to illustrate practical ways of doing such calculations, and provide references (and emphasize the normalizations of the measures in these sources) to the best of my ability at the moment. I am not aiming at presenting general (or rigorous) proofs here. 

My sincere gratitude also goes to the organizers of the Program ``On the Langlands Program: Endoscopy and Beyond'' at NUS for inviting me to give these lectures and for their patience and encouragement during the preparation of these notes; and to the referee for many helpful suggestions.

\section{Volume forms and point-counting}\label{sec:vf}
Everywhere in this note, $F$ stands for a non-Archimedean local field (of characteristic zero or positive characteristic), with the ring of integers $\ri_F$ and residue field $k_F$ of cardinality $q$. We denote its \emph{uniformizing element} (a \emph{uniformizer} for short)
by $\varpi$; by definition, $\varpi$ is a generator of the maximal ideal of $\ri_F$. We denote the normalized valuation on $F$ by $\ord$; thus, $\ord(\varpi)=1$.

\subsection{The measure on the affine line} We start with choosing, once and for all, an additive Haar measure on the affine line.
For a non-Archimedean local field $F$, we normalize the additive Haar measure on $F$ so that $\vol(\ri_F)=1$. 
For an affine line over $F$, given a choice of the coordinate $x$,  there is an invariant differential form $dx$; we declare that the associated measure $|dx|$ on $\A^1(F)$ also gives volume $1$ to the ring of integers (this choice is analogous 
to setting up the `unit interval' on the $x$-axis over the reals and declaring that the interval $[0,1]$ has `volume' (\emph{i.e.}, length) $1$ with respect to the measure $|dx|$).
Now that this choice is made, any non-vanishing top degree differential form $\omega$ (defined over $F$ or any finite extension of $F$) 
on a $d$-dimensional $F$-variety $\bV$ determines a 
measure $|\omega|$ on the set $\bV(F)$ of its $F$-points, where $|\cdot|$ stands for the absolute value on $F$ 
(respectively, its unique extension to the field of definition of $\omega$). Thus, our definition of the measure associated with a volume form is such that for the additive group 
$\G_a$ both approaches to normalizing the measure give the same natural measure: the measure that gives $\G_a(\ri_F)$ volume $1$. 
We shall see that such a measure is closely related to counting points over the residue field.

\begin{remark} Note that our choice of the normalization of the measure on the affine line differs from that of \cite{langlands-frenkel-ngo}: if the local field under consideration arises as a completion of a global field at a finite place, we normalize the Haar measure on it so that the ring of integers has volume $1$, whereas in \cite{langlands-frenkel-ngo}, the authors fix a  choice of  a character of the global field and normalize the measure on every completion so that it is self-dual with respect to that character. This choice is important for the Poisson summation formula; however, as the authors point out, this makes all the measures locally non-canonical. Since our exposition is local, we chose to omit this complication.
However, this means that given any variety $\bV$ defined over a global field $K$, 
our calculations of measures on $\bV(K_v)$ at every place $v$ differ from those of \cite{langlands-frenkel-ngo} by 
$N{\mathfrak d}_v^{\dim(\bV)/2} =|\Delta_{K/\Q}|_v^{\dim(\bV)/2}$, where $\mathfrak d$ is the different, $N:K\to\Q$ is the norm map,  and $\Delta_{K/\Q}$ is the discriminant of $K$.\footnote{more precisely, we quote (approximately) from \cite{langlands-frenkel-ngo}: `because of this there are no canonical local calculations. The ideal ${\mathfrak d}_v$ is however equal to $\ri_v$  almost everywhere. So there are canonical local formulas almost everywhere.'}
If $K=\Q$, this issue disappears.
\end{remark}

There is a natural notion of integration on the set of $p$-adic points of a variety with respect to a volume form, see \cite{weil:adeles}. A key feature of this theory is that if ${\mathbf X}$ is a \emph{smooth} scheme over $\ri_F$, and $\omega$ is a top degree non-vanishing differential form on ${\mathbf X}$, defined over $\ri_F$, then  the volume of ${\mathbf X}(\ri_F)$ with respect to the measure 
$|\omega|$ is given by the number of points on the closed fibre of ${\mathbf X}$: 
\begin{equation}\label{eq:weil}
\vol_{|\omega|}({\mathbf X}(\ri_F))=\frac{\#{\mathbf X}(k_F)}{q^{\dim({\mathbf X})}}.
\end{equation}
The relationship between volumes and point-counts for more general sets (e.g. not requiring smoothness) was further explored by Serre \cite[Chapter III]{serre:chebotarev}, Oesterl\'e \cite{Oesterle}, and in the greatest generality,\footnote{a far-reaching generalization of these ideas is the theory of motivic integration started by Batyrev \cite{batyrev:calabi-yau}, Kontsevich, Denef-Loeser \cite{DL.arithm}, and Cluckers-Loeser \cite{cluckers-loeser}.} W. Veys \cite{Veys:measure}. 

Here we will only need to consider the case of reductive algebraic groups. We start with algebraic tori, where the volumes already carry interesting arithmetic information.

\subsection{Tori}\label{sub:tori} 
Let $\bT$ be an algebraic torus defined over $F$.  Let $F^\sep$ be the separable closure of $F$. 
As discussed above, to define a measure on 
$T:=\bT(F)$ one can start with a differential form or with a compact subgroup. To define either, we first need a choice of coordinates on $\bT$.  One natural choice, albeit not defined over $F$ unless $\bT$ is $F$-split, comes from any basis of the character lattice of $\bT$.     
Let $\chi_1,\dots ,\chi_r$, where $r$ is the rank of $\bT$, be any set of generators of $X^\ast(\bT)$ over $\Z$ (the characters a priori are defined over $F^\sep$).  We note that this choice is equivalent to a choice of an isomorphism $\bT\simeq \G_m^r$ over $F^\sep$. 
Then we can define a volume form (defined over $F^\sep$ but \emph{not over $F$}, unless $\bT$ is split over $F$): 
\begin{equation}\label{eq:omegaT}
\omega_T=\frac{d\chi_1}{\chi_1} \wedge\ldots \wedge \frac{d{\chi_r}}{\chi_r}.
\end{equation}

The group of $F$-points $\bT(F)$ has a unique maximal compact subgroup in the $p$-adic topology; 
we denote it by $T^c$, following the notation of \cite{shyr77}. Ono gave the description of this subgroup in terms of 
characters:  
$$T^c 
=\{t\in \bT(F): |\chi(t)|=1 \text{ for } \chi\in X^\ast(\bT)_F\},$$
where $X^\ast(\bT)_F$ is the sublattice of $X^\ast(\bT)$ consisting of the characters defined over $F$. 


\noindent {\bf Question 1:}\label{q1}
What is the volume of $T^c$ with respect to the measure $|\omega_T|$? 

The question is well-defined because any other $\Z$-basis of $X^\ast(\bT)$ would differ from $\{\chi_i\}$ by a $\Z$-matrix of determinant $\pm 1$, and hence the resulting volume form would give rise to the same measure. 

The complete answer to this question is quite involved and requires machinery beyond the scope of this note. 
Here we show some basic examples illustrating the easy part and the difficulty, and provide further references in \S\ref{subsub:general}. 

\begin{example}\label{ex:split} $\bT$ is $F$-split. 
For $\G_m$, we have: the invariant form as above is $dx/x$, where $x=\chi_1:\G_m\to \G_m$ is the identity character and the natural coordinate, and 
$\G_m^c =\ri_F^\times$. 
The volume calculation gives: 
$$\begin{aligned} 
&\vol_{\left|\frac{dx}x\right|}(\ri_F^\times)= \int_{\ri_F^\times }\frac{1}{|x|} |dx|  = \int_{\ri_F^\times} |dx| = \vol_{|dx|}(\ri_F) - 
\vol_{|dx|}(\varpi \ri_F)\\
&= 1-\frac1q =\frac{\#k_F^\times}{q},
\end{aligned}$$
as predicted by (\ref{eq:weil}).
Then for an $F$-split torus of rank $r$, 
$$\vol_{|\omega_T|}(T^c)=\left(1-\frac1q\right)^r.$$
\end{example}

The next easiest case is Weil restriction of scalars.
\begin{example}\label{ex:quadr} Let $E/F$ be a quadratic extension, and $\bT:=\res_{E/F}\G_m$. 
Assume that the residue characteristic $p\neq 2$. 

We write $E=F(\sqrt{\epsilon})$, where $\epsilon$ is any non-square in $F$, and we can choose it to be in $\ri_F$ without  loss of generality.
By definition, $\bT$ has two characters defined over $E$, call them $z_1$ and $z_2$ and think of them as $E$-coordinates on $\bT(E)$.   
Then the volume form  $\omega_T$ is simply 
$\omega_T=\frac{dz_1}{z_1}\wedge\frac{dz_2}{z_2}$. Note that it is defined over $E$ but not over $F$. 
We can try to rewrite it in $F$-coordinates: 
we write $z_1= x+\sqrt{\epsilon} y, z_2= x-\sqrt{\epsilon} y$, and get: 
\begin{equation}\label{eq:pullback_res}
\begin{aligned}
&
\omega_T=\frac{d(x+\sqrt{\epsilon}y)}{x+\sqrt{\epsilon}y}\wedge\frac{d(x-\sqrt{\epsilon}y)}{x-\sqrt{\epsilon}y}\\
&=\frac{(dx+\sqrt{\epsilon}dy)\wedge(dx -\sqrt{\epsilon}dy)}{x^2-\epsilon y^2} = \frac{-2 \sqrt{\epsilon}}{N_{E/F}(x+ \sqrt{\epsilon} y)} dx\wedge dy,
\end{aligned}
\end{equation}
where $N_{E/F}$ is the norm map. 

We note that $\bT(F)= E^\times$ as a set, and the norm map is the generator of the group of $F$-characters 
of $\bT$.  
Thus the subgroup $T^c$  of $\bT(F)$ is 
$$T^c=\{x+\sqrt{\epsilon} y \in E^\times : |x^2-{\epsilon}y^2|_F=1 \}.$$
Its volume with respect to the measure $|\omega_T|$ is: 
\begin{equation}\label{eq:ex1}
\vol_{|\omega_T|} (T^c) = |2\sqrt{\epsilon}| \int_{\{(x,y)\in F^2: |x^2-\epsilon y^2|=1\}} dx dy. 
\end{equation} 

Thus we have reduced the computation of the volume of $T^c$ with respect to the volume form $\omega_T$ to the computation of the volume of the subset of $\A^2(F)$, which we denote by $C$, defined by 
$$C:=\{(x,y)\in F^2: |x^2-\epsilon y^2|=1\}$$
with respect to the usual measure on the plane $|dx\wedge dy|$ (note that $C$ is open in $\A^2(F)$ in the $p$-adic topology).  
The computation of this volume  illustrates the  way to use  point-counting over the residue field, and for this reason we  do it in detail. 
There are two cases: 
$\epsilon$  is a unit (i.e., $E/F$ is unramified), and  $\epsilon$ is not a unit.

First, consider the unramified case. Note that $\ord(x^2-\epsilon y^2)=\min(\ord(x^2), \ord(y^2))$ since $\epsilon$ is a non-square unit. 
Then our set can be decomposed as: 
$$C=\{(x, y): x\in \ri_F^\times, y\in \ri_F \}\sqcup \{(x,y): x\in \ri_F\setminus \ri_F^\times, y\in \ri_F^\times\},$$ and its volume with respect to the affine plane measure $|dx\wedge dy|$ is, therefore: 
$$\vol_{|dx\wedge dy|}(C)= \frac{(q-1)q}{q^2}+ \frac{1}q\frac{q-1}q = \frac{q-1}q\left(1+\frac 1q\right).$$

There is an alternative (and more insightful) way to do this calculation: 
first, as above, observe that necessarily the set $C$ is contained in $\ri_F^2$. Then 
consider the reduction mod $\varpi$ map $(x, y)\mapsto (\bar x, \bar y)$ from $\ri_F^2$ to $k_F^2$, where $\varpi$ is the uniformizer of the valuation of $F$. Each fibre of this map is a translate of the set 
$(\varpi)\times(\varpi)\subset \ri_F^2$, thus each fibre has volume $q^{-2}$ with respect to the measure that we have denoted by $|dx\wedge dy|$. Therefore we just need to compute the number of these fibres to complete the calculation. There are two ways to do it: one is to proceed by hand, which in this case is easy enough. 
Another is to appeal to a generalization of Hensel's Lemma:  the affine $\ri_F$-scheme defined by $x^2-\epsilon y^2\neq 0$ is smooth; this implies that
the reduction map from the set of its $\ri_F$-points is surjective onto the set of $k_F$-points of its special fibre (see, e.g.,  \cite[\S 3]{serre:chebotarev} and 
\cite[III.4.5, Corollaire 3, p.271]{Bourbaki:commalg}). The set $C$ can be written as: 
$$C=\{(x,y)\in \ri_F^2: \overline{(x^2-{\epsilon} y^2)}\neq 0\}.$$
Since the reduction map is surjective in this case, we just need to find the number of points 
$(\bar x, \bar y)\in k_F^2$ such that $\bar x^2 -\bar \epsilon \bar y^2 \neq 0$. The set of points satisfying this condition is in bijection with $\F_{q^2}^\times$, where $\F_{q^2}$ is the quadratic extension of our residue field $k=\F_q$,  and we get the same result as above for the volume of $C$. 

If the extension is ramified, the calculation changes. 
In this case  $\ord(\epsilon)=1$, hence for $x^2-\epsilon y^2$ to be a unit, $x$ has to be a unit and there is no condition on $y$ other than it has to be an integer. Again consider the reduction modulo the uniformizer map. Its image in this case is $\F_q^\times \times \F_q$ (again, this can be checked by hand in this specific case), and thus the volume of the set $C$ is 
$\frac{(q-1)q}{q^2}$.

We summarize (cf. \cite{Langlands:2013aa}): 
\begin{equation}\label{eq:res_quadr}
\vol_{|\omega_T|}(T^c)=\begin{cases}&(\frac{q-1}q)^2 \quad {\bT}\text { split }\\
&\frac{q-1}q \frac {q+1}q \quad {\bT}\text { non-split unramifed  }\\
& \frac1{\sqrt{q}}\frac{q-1}q \quad {\bT}\text { non-split ramified. }\\
\end{cases}
\end{equation} 
Note the factor $|2|$ from (\ref{eq:ex1}) disappears (i.e., $|2|_v = 1$)  since we are assuming that $p\neq 2$. The factor 
$\frac1{\sqrt{q}}$ in the ramified case 
comes from the factor $|\sqrt{\epsilon}|$ in (\ref{eq:ex1}) (in this case,  $\epsilon=\varpi$ up to a unit since $p\neq 2$). 

For $p=2$, everything in the calculation is slightly different (and a lot longer) but the answers are similar. 
Not to interrupt the flow of the exposition, we postpone the discussion of $p=2$  till \S\ref{sub:two} below.  

\subsubsection{Norm-1 torus of a quadratic extension}\label{susbsub:norm1} 
Finally, to illustrate general difficulties of this volume computation, we consider the example of the norm-1 torus of a quadratic extension. 

There is an exact sequence of algebraic $F$-tori: 
\begin{equation}\label{eq:exact_seq_n1}
1\to \res_{E/F}^{(1)}\G_m \to \res_{E/F} \G_m \to \G_m \to 1,
\end{equation}
 where the last map is the norm map; its kernel is an algebraic torus $\bT_1:=\res_{E/F}^{(1)}\G_m $ over $F$, called the \emph{norm-1} torus. 
 
It is tempting to try to use this exact sequence to compute the volume of $T_1^c$ with respect to $\omega_{T_1}$, but that is not the right way to proceed. The standard way to do this calculation is to consider an isogeny between $\bT_1\times \G_m$ and $\bT$ and use the results of Ono on the behaviour of various invariants attached to tori under isogenies; see also \cite{shyr77}. However, this would take us too far afield; instead we proceed with an elementary calculation. 
Before we do this calculation for a $p$-adic field, consider for a moment the  situation when $F=\R$ and $E=\C$, in order to get some geometric intuition. 

\begin{example}\label{ex:heuristic} 
Let  $\bT := \res_{\C/\R}\G_m(\R)=\C^\times$; then the norm-1 torus is the unit circle $S^1$. 

The same calculation as in Example \ref{ex:quadr}  shows that the volume form $\omega_T$ gives the measure 
$|\omega_T|=|\frac{2i dx\wedge dy}{x^2+y^2}| =\frac{2}{x^2+y^2}|dx\wedge dy|$ on $\C^\times$. 

Similarly, if we go by the definition of the volume form $\omega_{S^1}$ on $S^1$, we obtain the following.
The generator of the character group $X^\ast(S^1)$ (over $\C$) is simply the identity character $z\mapsto z=x+iy$. 
We get that $\omega_{S^1}=\frac{dz}{z}$ by definition, but intuitively it is not  yet clear what is the measure defined by this form.
We write $dz=dx+idy$, and note that $\frac1z =\bar z=x-iy$ when $z\in S^1$. 
We obtain: 
\begin{equation}\label{eq:circle}
\omega_{S^1}=(x-iy)(dx+idy)= (x-iy)dx+(y+ix)dy.
\end{equation}
It is still not obvious what measure this form gives; it would be convenient to rewrite it 
using the local coordinate of some chart on the circle. 
Here we can use the fact that we are working over $\R$ and take $\theta$ to be the arc length; then $x=\cos \theta$, $y=\sin \theta$, 
$0\le \theta <2\pi$ 
is the familiar (transcendental) parametrization. 
We get:
\begin{equation}\label{eq:S1}
\begin{aligned}
&\omega_{S^1} =(x-iy)dx+(y+ix)dy\\
&=(\cos \theta -i\sin \theta) d(\cos \theta)+(\sin \theta +i\cos\theta) d(\sin \theta) \\
&= -\cos \theta \sin \theta d\theta + \sin \theta \cos \theta d\theta +i(\cos^2 \theta d\theta +\sin^2 \theta d\theta) =id\theta.  
\end{aligned}
\end{equation}
Since $|i|=1$, we see that the measure $|\omega_{S^1}|$ coincides with the arc length.

The exact sequence (\ref{eq:exact_seq_n1}) gives a relation between this measure and the measures on $S^1$ and $\G_m$, which is the same as rewriting the measure $dx\wedge dy$ in polar coordinates. 
Indeed, we have (from calculus)  $dx\wedge dy = r dr \wedge d\theta$, so 
$\frac{dx\wedge dy}{x^2+y^2} =\frac{dr}r \wedge d\theta$. 
We obtain the relation between $\omega_{S^1}$ and the form $\omega_T$ on $\bT=\res_{\C/\R}\G_m$: 
\begin{equation}
\omega_T= 2 \omega_{S^1}\wedge \omega_{\G_m}.
\end{equation}
\end{example}
The appearance of the factor  $2$ in this relation, combined with the fact the norm map to $\G_m$ is not surjective on $\R$-points 
and is $2:1$ illustrates that the relation between the measures  on $\bT$ and $\bT_1$ is not straightforward (if one cares for a power of $2$). Armed with this caution, we move on to the $p$-adic fields.

\begin{example}\label{ex:norm1}
Let  $\bT_1 := \res_{E/F}^{(1)}\G_m$ be the norm-1 torus of a quadratic extension as above, but with $E=F(\sqrt{\epsilon})$ an extension of non-Archimedean local fields as in Example \ref{ex:quadr}. As before, we assume $p\neq 2$ (the case $p=2$ is treated below in \S\ref{sub:two}).  

As above, we would like to understand the form $\omega_{\bT_1}$. We observe that the relation (\ref{eq:circle}) can be easily adapted to this case (essentially, replacing $i$ with $\sqrt{\epsilon}$). What we need is an algebraic parametrization of the conic $x^2-\epsilon y^2=1$. 
Such a parametrization is given in projective coordinates $(x:y:z)$ by: 
\begin{equation}\label{eq:param_conic}
x=\epsilon t^2 + 1 ,\quad y=2 t, \quad z=1-\epsilon t^2, \quad t\in F.
\end{equation}
Then a calculation similar to (\ref{eq:S1}) shows that in the affine chart $z\neq 0$, 
$$
\omega_{\bT_1}=-\frac{2\sqrt{\epsilon}}{1-{\epsilon}t^2} dt.
$$
We note here that we could have used the same rational parametrization for the unit circle in the example above; then at this point we would have obtained the same answer: if we plug in $\epsilon =-1$,  we 
get that
the ``volume'' of the circle with respect to $|\omega_{S^1}|$ is
$2\int_{\R}\frac1{t^2+1}\, dt = 2\pi$, as expected. 

Continuing with the $p$-adic calculation, we can discard $|2|$ since $p\neq 2$.
Next, we observe that $T_1^c=\bT_1(F)$ (our torus is \emph{anisotropic}; it has no non-trivial $F$-characters, and hence the condition defining $T_1^c$ is vacuous).  
Therefore, the volume of $T_1^c$ with respect to $\omega_{\bT_1}$ is 
\begin{equation}\label{eq:n1t}
\vol_{|\omega_{\bT_1}|}(T_1^c) = \int_{F}\left|\frac {\sqrt{\epsilon}}{{1-\epsilon}t^2}\right|\,  |dt|. 
\end{equation}

Now we need to consider two cases. \\
{\bf Case 1.} The extension is unramified, i.e., $\epsilon$ is a non-square unit. 
Then 
(\ref{eq:n1t}) becomes (using the fact that 
the volume of the $p$-adic annulus  $\{t: \ord(t)=n\}$ with respect to the measure $|dt|$ equals $q^{-n}-q^{-(n+1)}$) : 
\begin{equation}\label{eq:vol_unram} 
\begin{aligned}
&\vol_{|\omega_{\bT_1}|}(T_1^c) =& \int_{F}\frac1{|1-\epsilon t^2|}\, |dt| = 
\int_{\ri_F}  |dt| + \sum_{n=1}^\infty \frac1{q^{2n}}(q^{n}-q^{(n-1)}) \\
&{} & =1+\left(1-\frac1q\right)\sum_{n=1}^{\infty}\frac1{q^n}= 1+\frac 1q.
 \end{aligned}
 \end{equation}

{\bf Case 2.} The extension is ramified, i.e., $\ord(\epsilon)=1$.
Then $|1-\epsilon t^2|= 1$ if $t\in \ri_F$ and  $|1-\epsilon t^2|= q^{2n-1}$ if $\ord(t)=n<0$.  
Thus, the integral computing the volume of $T_1^c$ again breaks down as a sum: 
\begin{equation}\label{eq:one-more-two}
\begin{aligned}
&\vol_{|\omega_{\bT_1}|}(T_1^c) = & \frac1{\sqrt{q}}\int_{F}\frac1{|1-\epsilon t^2|}\, |dt| = 
\frac1{\sqrt{q}} \left(\int_{\ri_F}  |dt| + \sum_{n=1}^\infty \frac1{q^{2n-1}}(q^{n}-q^{(n-1)})\right) \\
&{} & =\frac1{\sqrt{q}}\left(1+q\left(1-\frac1q\right)\sum_{n=1}^{\infty}\frac1{q^n}\right)= 
\frac2{\sqrt{q}}.
 \end{aligned}
 \end{equation}
Let us compare the results of this calculation with the approach to volumes via point-counting. 
If we take the equation $x^2-\epsilon y^2=1$ and reduce it modulo the uniformizer, we get an equation of a conic over $\F_q$.
In the unramified case, this conic is in bijection with ${\mathbb P}^1(\F_q)$ via (\ref{eq:param_conic}); thus we expect the volume to equal 
$\frac{q+1}q$, which agrees with (\ref{eq:vol_unram}).  

In the ramified case, when $\epsilon$ is not a unit, the reduction of the same equation modulo the uniformizer gives 
a disjoint union of two lines over the finite field: it is the subvariety of the affine plane defined by $x^2=1$.
The point-count over $\F_q$ gives us $2q$, thus the volume we obtain is 
$\frac{2q}{q}=2$, which agrees with (\ref{eq:one-more-two}) once we make the correction for the fact that our volume form had a factor of $\sqrt{\epsilon}$ and thus was not defined over $F$ (this again illustrates why in \cite{weil:adeles} the disciminant factor appears in the definition of the volume form).  
We note that when $p\neq 2$, the affine scheme $\spec F[x,y]/(x^2-\epsilon y^2-1)$ is smooth over $\ri_F$ (in both the ramified and unramified cases, which can be checked by the Jacobi criterion, \cite[\S 2.2]{bosch-lutkebohmert-raynaud:NeronModels}), and this justifies the fact that  the point-count on the reduction $\mod \varpi$ does give us the correct answer. 

\end{example}

\subsubsection{The N\'eron model}\label{subsub:Neron}
How do the above calculations generalize to an arbitrary algebraic torus? 
The issue is that for a torus that is not $F$-split, it is not a priori obvious how to choose `coordinates' defined over $\ri_F$; more precisely, one first needs to define an \emph{integral model} for $\bT$, i.e., a scheme over $\ri_F$ such that its generic fibre is $\bT$. In order to use the formula 
(\ref{eq:weil}), this model would also need to be a \emph{smooth} scheme over $\ri_F$. 
There is a canonical way to define such a smooth integral model for $\bT$, namely, the \emph{weak N\'eron model}, \cite[Chapter 4]{bosch-lutkebohmert-raynaud:NeronModels}, which we shall denote by $\mathcal T$.
In general it is a scheme not of finite type (it can have infinitely many connected components). 
The $\ri_F$-points of its identity component $T^0:={\mathcal T}^0(\ri_F) $ provide another canonical compact subgroup of $\bT(F)$. This subgroup is traditionally used in the literature to normalize the Haar measures on tori (and plays a role in normalization of measures on general reductive groups, as we shall see below), but it plays no explicit role in this note, hence we do not discuss any details of its definition. 

Moreover, once we have the $\ri_F$-model for $\bT$, we can use the local coordinates associated with this model to define a volume form on $\bT(F)$. Unlike the volume form defined above by using the characters, this form actually has coefficients in $F$; it is called the \emph{canonical} volume form in the literature, following the article \cite{gross:motive}; we call it $\omega^{\can}$, but we shall not use any explicit information about it in this note.

In general, $T^0$ is a subgroup of finite index in $T^c$ (this index is an interesting arithmetic invariant, see \cite{bitan} for a detailed study), 
and the relationship between the form $\omega_T$ defined above and 
 $\omega^{\can}$ is discussed in \cite{gan-gross:haar}.
 The subgroup $T^c$ is the set of $\ri_F$-points of the so-called \emph{standard} integral model of $\bT$, which is not smooth in general; roughly speaking, the coordinates on the standard model come from the characters $\chi_i$ as above  in the examples (see \cite[\S 1.1]{bitan}). 

If $\bT$ splits over an unramified extension of $F$, the situation is simple (see \cite{bitan} and references therein for details): 
\begin{theorem} Suppose that $\bT$ splits over an unramified extension of $F$. Then 
\begin{enumerate}
\item $T^0= T^c$, and the special fibre ${\mathcal T}^0_{\kappa}$ of ${\mathcal T}^0$ is an algebraic torus over $\F_q$. 
\item $\vol_{|\omega_T|}(T^c)= \vol_{|\omega^{\can}|}(T^c)= \frac{\# {\mathcal T}^0_{\kappa}(\F_q)}{q^{\dim(\bT)}}$.  
\end{enumerate}
\end{theorem}

We give one illustrative example without any details, and summarize the known general results below in \S \ref{subsub:general}. 
\begin{example}  Consider again  Example \ref{ex:norm1}, where $\bT$ is the norm-1 torus of a quadratic 
extension. It is anisotropic over $F$, and consequently its N\'eron model is a scheme of finite type over $F$. 
If $E/F$ is unramified, the `standard model' (see \cite{bitan}) coincides with the N\'eron model, and is simply defined by the equation 
$x^2-\epsilon y^2=1$ over $\ri_F$. 
It is connected and its special fibre is an irreducible conic over $\F_q$, which has $q+1$ 
rational points over $\F_q$, as discussed above. 

If $E/F$ is ramified, e.g. $E=F[\sqrt{\varpi}]$, then $|T^c/{T}^0|=2$ (the special fibre of $\mathcal T$ has two connected components, each isomorphic to an affine line as we saw above -- note that  it is not an algebraic torus!) 
And in this case we have 
$$\vol_{|\omega^{\can}|}({T}^0)= \frac{\# {\mathcal T}^0_{\kappa}(\F_q)}{q^{\dim(\bT)}}=\frac{q}q=1,$$
and also $\omega_T=|\sqrt{\varpi}|_E \omega^{\can} = \frac1{\sqrt{q}}\omega^{\can}$. 
\end{example} 

\end{example} 
\begin{example}\label{ex:restr} 
For an arbitrary finite extension $E/F$ and 
$\bT=\res_{E/F}\G_m$, we have  $\bT(F) = E^\times$, and $T^c =\ri_E^\times$. 
If $\alpha_1, \dots, \alpha_r$ are elements of $\ri_E$ that form a basis for $\ri_E$ over $\ri_F$, 
then $\chi_i(x):=\tr_{E/F}(\alpha_i x)$ form a basis of $X^\ast(\bT)$ over $\Z$. 
Then by definition of the discriminant (as the norm of the different ${\mathfrak d}_{E/F}$), 
the measure $|\omega_T|=|\wedge d\chi_i|$ equals
$|\det(\tr_{E/F}(\alpha_i))|_F \omega^{\can}$,  i.e. the conversion factor is
the square root of the $F$-absolute value of the discriminant of $E$:
$$\omega_T=\sqrt{|\Delta_{E/F}|}\omega^\can.$$
This calculation is generalized to an arbitrary reductive group (not just an arbitrary torus) in \cite{gan-gross:haar}. 
\end{example}

\subsubsection{References to the general results: a non-self-contained answer to Question 1} \label{subsub:general} 
\begin{enumerate} 
\item In general, the index $[T^c:T^0]=|(X_\ast)_I^{tor}|$ can be computed by looking at the inertia co-invariants on the co-character lattice of $\bT$, see \cite[(3.1)]{bitan} which follows  \cite[\S 7]{kottwitz:isocrystals-2}. 
\item The relation
$ \vol_{|\omega^{\can}|}(T^0) = \frac{\# {\mathcal T}^0_{\kappa}(\F_q)}{q^{\dim(\bT)}}$ holds for \emph{any} algebraic torus, 
\cite[Proposition 2.14]{bitan}.
\item As noted above, the form $\omega_T$ is not generally defined over $F$. It turns out (see \cite{gan-gross:haar}) that it only needs to be corrected by a factor that is a square root of an element $F$ to get to a form defined over $F$. We saw this already in the case when $\bT$ is of the form 
$\res_{E/F}\G_m$; the general case follows from this calculation and a theorem of Ono that relates an arbitrary torus with a torus of the form $\res_{E/F}\G_m$, see proof of Corollary 7.3 in \cite{gan-gross:haar}. 
Specifically, Corollary 7.3 in \cite{gan-gross:haar} states (using our notation) that if $F$ has characteristic 
zero or if  $\bT$ splits over a Galois extension of $F$ of degree relatively prime to the the characteristic of $F$, then: 
\begin{equation}
\left |\frac{\omega_T}{\sqrt{D_M}}\right|  = | \omega^\can|, 
\end{equation}
where $D_M$ is a \emph{refined Artin conductor} of the motive $M$ associated with $\bT$, defined in \cite[(4.5)]{gan-gross:haar}. We discuss this motive briefly in the next subsection, but do not discuss the definition of the Artin conductor. 

\item Combining these results gives an answer to Question 1 above.  
 
\end{enumerate}

\subsubsection{The local $L$-functions} There is yet another way to express the number of $\F_q$-points of 
${\mathcal T}^0_\kappa$, and hence the volume of $T^0$, entirely in terms of the representation of the Galois group on the character lattice, using Artin L-factor.
 
Indeed, an algebraic torus $\bT$ over $F$ is uniquely determined by the action of the Galois group of its splitting field $E$ on $X^\ast(\bT)$. Let $\Gamma=\Gal(E/F)$ and let $I$ be the inertia subgroup and let $\fr$ be the Frobenius automorphism of $k_E$ over $k_F$. Recall that we have the exact sequence of groups 
$$1\to I \to \Gamma \to \langle \fr \rangle \to 1,$$
and the cyclic group $\langle \fr \rangle$ is isomorphic to the Galois group of the residue field $\Gal(k_E/k_F)$. 
Thus we get a natural action of $\Gal(k_E/k_F)$ on the set of inertia invariants $X^\ast(\bT)^I$. 
Let us denote this integral representation by 
$$\sigma_T: \Gal(k_E/k_F) \to {\operatorname{Aut}}_\Z(X^\ast(\bT)^I)\simeq \GL_{d_I}(\Z),$$ 
where $d_I=\rk(X^\ast(\bT)^I)$. 
The Artin $L$-factor associated with this representation is, by definition,
\begin{equation}\label{eq:L}
L(s, \sigma_T)=\det\left(I_{d_I}- \frac{\sigma_T(\fr)}{q^s}\right)^{-1} \quad \text { for } s\in \C,
\end{equation}
where $I_{d_I}$ is the identity matrix of size $d_I$.
Then the following relation holds (we are quoting it from \cite[Propostion 2.14]{bitan}): 
\begin{theorem}\label{thm:torus}
$\omega^{\can}(T^0)=\frac{\# {\mathcal T}^0_{\kappa}(\F_q)}{q^{\dim(\bT)}} = L(1, \sigma_T)^{-1}$.
\end{theorem}
Note that if $E/F$ is unramified, the inertia is trivial, so $d_I=\rk(\bT)$.

Thus, to summarize, we have defined a natural invariant form $\omega_T$ on $\bT$ and described the maximal compact subgroup $T^c$  of $\bT(F)$ in terms of the characters of $\bT$.  
If $\bT$ splits over an unramified extension, the volume of $T^c$ with respect to this differential form equals 
\begin{equation}\label{eq:unram}
\vol_{|\omega_T|}(T^c)=
\frac{\#{\mathcal T}^0_\kappa (\F_q)}{q^{\dim(\bT)}} = L(1, \sigma_T)^{-1},
\end{equation}
where the second equality holds for any $\bT$, not necessarily unramified. 
In general, the volume of $T^c$ with respect to this differential form contains two more factors -- the index of 
$T^0$ in $T^c$ and the ratio between the differential forms $\omega_T$ and $\omega^{\can}$.

\subsection{Reductive groups} 
Similarly to the case of tori discussed above, for a general reductive group $\bG$ over a local field $F$, 
the choice of a normalization of Haar measure is linked with a choice of a `canonical' differential form or a `canonical' compact subgroup of $\bG(F)$ (unlike an algebraic torus, the set of $F$-points of a general reductive group can have more than one conjugacy class of maximal compact subgroups, and this choice matters for the normalization of measure).
Luckily, for the questions studied in \cite{langlands-frenkel-ngo} the choice of the normalization of measure on $\bG(F)$ does not matter -- it only contributes some global constant. 

However, for completeness, we record that a `canonical' choice of a compact subgroup $G^0$ and an associated volume form $\omega_G$ is 
described by B. Gross in \cite{gross:motive}, using Bruhat-Tits theory. 
The group $G^0$ is the set of $\ri_F$-points of a smooth scheme $\underline{\bG}$ over $\ri_F$ whose generic fibre is $\bG$. 
Hence, by Weil's general argument (since $\underline{\bG}$ is smooth over $\ri_F$),  the volume of $G^0$ is obtained by counting points on the special fibre of $\underline \bG$ (see  (\cite[Proposition 4.7]{gross:motive}): 
\begin{equation}
\vol_{|\omega_G|}(G^0)=\frac{\#\underline{\bG}_\kappa(\F_q)}{q^{\dim\bG}}. 
\end{equation}
Moreover, Gross defines an Artin-Tate motive $M$ associated with $\bG$ such that the volume of the canonical compact subroup with respect to the canonical form is given by the value of the Artin $L$-function associated with this motive at $1$. (If $\bG$ is an algebraic torus, the associated motive is precisely the representation of the Galois group on its character lattice as above, and Gross' result amounts precisely to the statement of Theorem \ref{thm:torus} above).


\section{Orbital integrals: the geometric measure}\label{sec:oi}
\subsection{The two normalizations}\label{subseq:the_problem}
Let $\bG$ be a connected reductive algebraic group over $F$.
We denote the sets of regular semisimple elements in $G:=\bG(F)$ by  $G^\rss$  
(respectively, $\fg^\rss$ for the Lie algebra $\fg$ of $\bG$).\footnotetext{The simplest way to characterize the set of regular semisimple elements is to use any faithful representation of 
$G$ to think of its elements as matrices; 
then an element 
$\gamma\in \bG(F)$  (respectively, $X\in \fg$) is regular and semisimple if and only if its eigenvalues (in an algebraic closure of $F$) are distinct; we will also give a precise definition below in \S \ref{sub:weyl_discr}.}

Let $\gamma\in G:=\bG(F)^\rss$ 
(in this note we are only interested in this setting). 
The adjoint orbit (or simply, `orbit', or sometimes, `rational orbit' when we want to emphasize that it is the group of $F$-points of 
$\bG$ that is acting on it) of $\gamma$ is the set 
$$\ri(\gamma) := \{g\gamma g^{-1} \mid g\in \bG(F)\}.$$
The centralizer of $\gamma$ is by definition the group $C_{G}(\gamma) = \{  g\in \bG(F)\mid g\gamma g^{-1} =\gamma \}$.
We will also briefly refer to the notion of a \emph{stable} orbit of $\gamma$. It is a finite union of rational orbits;  
as a first approximation, it can be thought of as the set 
$$\ri(\gamma)^\mathrm{stable} := \{g\gamma g^{-1} \mid g\in \bG(F^\sep)\}\cap \bG(F).$$
However, this is not the correct definition in general; see \cite{kottwitz82}. 
We will not need a precise definition in this note.  
If $\bG=\GL_n$, then for $\gamma\in G^\rss$, the stable orbit and rational orbit are the same.


When $\gamma \in G^\rss$, the identity component of the centralizer of $\gamma$ is a maximal torus $T\subset G$; it can be thought of as a set of $F$-points $T=\bT(F)$ of an algebraic torus $\bT$ defined over $F$.
This leads to 
two natural approaches to normalizing the measure on the orbit of $\gamma$: 
\begin{enumerate}
\item Normalize the measures on $G$ and $T$ according to one of the methods discussed above and consider the quotient measure.
\item Describe the space of all (stable) orbits, and derive a measure on each orbit as a quotient measure with respect to the measure on the space of orbits. \footnote{In fact, there is a third natural approach if we are working with the orbital integrals on the Lie algebra rather than the group: namely, to identify $\fg$ with $\fg^\ast$ and consider the differential form on the orbit itself which comes from Kirillov's symplectic form on co-adjoint orbits, see \cite{kottwitz:clay}. We will not discuss this approach here as it is not related to the main subject of the note. However, the example of $\fsl_2$ where one can clearly see the relation of this measure to what one would expect from calculus is provided in the Appendix by Matthew Koster.}
\end{enumerate} 
 
Since the $G$-invariant measure on each orbit is unique up to a constant multiple, the two orbital integrals defined with respect to these measures will of course differ by a constant; however, this constant can, and does, depend on the orbit. 
The goal of this section is to give a detailed explanation for the formula that relates the two orbital integrals; this is equation (3.31) in \cite{langlands-frenkel-ngo}. 
More specifically,  we start with a review of the construction of Steinberg map 
$\fc: \bG \to \A_G$, in \S \ref{sub:steinberg} below. 
The set $\A_G(F)$ has an open dense subset whose points parametrize the stable orbits of regular semisimple elements in $G$ -- each fibre of the map $\fc$ over a point of this subset is such a stable orbit. 
The relation (3.31) in \cite{langlands-frenkel-ngo} we aim to explain is:  
\begin{equation}\label{eq:FLN}
\int_{\fc^{-1}(a)} f(g) d|\omega_a| = |\Delta(t)| L(1, \sigma_{T\backslash G}) O^{\mathrm{stable}}(t, f),
\end{equation} 
where $a=\fc(t)$. 
We start by defining all the ingredients of this formula (as we shall see, this formula is not really about orbital integrals; it is simply a statement about the relationship between two invariant measures on an orbit). 
We also simultaneously treat the orbital integrals on Lie algebras. 

\subsection{The space $\A_G$; Chevalley and Steinberg maps}\label{sub:steinberg}
We start with the Lie algebra, where the situation is simpler.
We recall that $\bG$ acts on $\fg$ via adjoint action, denoted by $\Ad:\bG \to \GL(\fg)$ (for the classical groups and their Lie algebras, 
$\Ad(g)$ is simply matrix conjugation by an element $g\in \bG(F)$). 
When we talk about orbits in $\fg$, it is the orbits under the adjoint action. 
For $X\in \fg$, its \emph{centralizer} $C_G(X)$ is, by definition, its stabilizer (in $\bG(F)$) under the adjoint action. 
If $X\in \fg^\rss$, then $C_G(X)$ is a maximal torus in $\bG(F)$. 

\subsubsection{Reductive Lie algebra; algebraically closed field}\label{subsub:Liealg}\footnote{This section is entirely based on 
\cite[\S 14]{kottwitz:clay}.}
Let $\fg$ be the Lie algebra of $\bG$. For the moment let us work over an algebraically closed field $k$ of characteristic $0$ (in fact, assuming sufficiently large characteristic is sufficient here but we will not pursue this direction). 
Let $\ft=\mathrm{Lie}(\bT)$ be a maximal Cartan subalgebra. 
Then the ring of polynomial functions on $\ft$ is the symmetric algebra $S=S(\ft^\ast)$. The Weyl group acts on $S$, and the ring of invariants $S^W$ is the ring of regular functions on the quotient $\ft/W$. 
In other words, $\ft/W$ is a variety over $k$, isomorphic to $\spec S^W$. 
Furthermore, in fact $S^W$ is itself a polynomial ring, and so $\ft/W$ is isomorphic to the affine space $\A^r$, where $r=\rank(\bT)$. 
\begin{example} Let $\bG=\GL_n$, and let $\ft\subset \fg=\mathfrak{gl}_n$ be the Cartan subalgebra consisting  of diagonal matrices.  Then $W=S_n$, $S=k[x_1, \dots, x_n]$, and $S^W$ is the algebra of symmetric polynomials. 
As we know, it is generated by the elementary symmetric polynomials. Thus, the map $\ft\to\A^r$ is given by: 
$t=\diag(t_1, \dots, t_r)\mapsto (\sigma_1(\bar t), \dots, \sigma_r(\bar t))$, where $\bar t =(t_1, \dots, t_r)$ and 
$\sigma_k(\bar t)=\sum_{\{i_1, \dots, i_k\}\subset \{1, \dots, r\}} t_{i_1}\dots t_{i_k}$ is the $k$-th elementary symmetric polynomial. 
Note that in particular, for $n=2$, we get the map $t\mapsto (\tr(t), \det(t))$. 
\end{example}

Let $k[\fg]$ be the $F$-algebra of polynomial functions on $\fg$, and let $k[\fg]^G$ be the subalgebra of 
the polynomials invariant under the adjoint action of $G$. 
We quote from \cite[\S 14.2]{kottwitz:clay}: Chevalley's restriction theorem can be stated as: 
$$k[\fg]^G\cong S^W,$$
where the isomorphism is given by restricting the polynomial functions from $\fg$ to $\ft$. 
Dually to the inclusion $k[\fg]^G\hookrightarrow k[\fg]$, we get the surjection (which we will refer to as  \emph{Chevalley map})
$$\fc_{\fg}: \fg \to \A_G=\ft/W,$$
which maps $X\in \fg$ to the unique $W$-orbit in $\ft$ consisting of elements conjugate to the semisimple part of $X$.  
An important observation (which is not used in these notes but is very relevant for the subject) is that the nilpotent cone in $\fg$ is $\fc_{\fg}^{-1}(0)$. 

In general, the role of `elementary symmetric polynomials' is played by the traces of irreducible representations of $\fg$ determined by the \emph{fundamental weights}. Namely, let $\{\mu_i\}_{i=1}^r$ be the fundamental weights determined by a choice of simple roots for $\fg$
(i.e., the weights of $\fg$ defined by $\langle \mu_i, \alpha_j\rangle=\delta_{ij}$, where $\Delta=\{\alpha_j\}_{j=1}^r$ is a base of the root system of $\fg$). 
Let $\rho_i$ be the representation of $\fg$ of highest weight $\mu_i$. 
Then $S^W$, as an algebra,  is generated by $\Tr(\rho_i)$ (see,  e.g., \cite[23.1]{humphreys}). 

For the type $A_n$, one recovers the elementary symmetric polynomials from this construction. 
Namely, for $\fsl_n$, it happens that  the exterior powers of the standard representation are irreducible, and they give all the fundamental representations: 
$\rho_i=\wedge^i \rho_1$, for $i=1, \dots, {n-1}$, where $\rho_1$ is the standard representation, which has highest weight $\mu_1$. 
Consequently, since the coefficients of the characteristic polynomial of a matrix are (up to sign) the traces of its exterior powers, we obtain: 

\begin{example}
For $\fsl_n$, 
Chevalley map can be realized explicitly as 
$X \mapsto (a_i)$, where $a_i$ are the coefficients of the characteristic polynomial of $X$. 
\end{example}

\subsubsection{Reductive Lie algebra, non-algebraically closed field}
When the field $F$ is not algebraically closed, the space $\A_G$ can be defined as $\spec F[\fg]^G$, avoiding the need to choose a maximal torus;  it turns out that the morphism $\fc$ is still defined over $F$ (see \cite[\S 14.3]{kottwitz:clay}). 
However, in this note we are only considering the case of $\bG$ split over $F$, and it is convenient for us to continue using an explicit definition of the map.
Namely, if $\bG$ is split over $F$, we can choose the split maximal torus $T^\spl$ in $G$, and define $\A_G=\ft^\spl/W$ exactly as above. The definition of the map $\fc_{\fg}$ stays the same. 
Consider explicitly what happens in the $\bG=\GL_2$ example. 

\begin{example}\label{ex:gl2}
 As above, Chevalley map is the map $\fc_{\mathfrak{gl}_2}: \mathfrak{gl}_2 \to \A^2$,
$X\mapsto (\tr(X), \det(X))$. 
All the split Cartan subalgebras are conjugate in $\fg$.
The image under Chevalley map of any split Cartan subalgebra  $\ft$ of  $\fg$ is the set 
$$(a_1, a_2)\in \A^2: a_1^2-4a_2 \text{ is a square in } F.$$ 
We observe that the $F$-conjugacy classes of Cartan subalgebras in ${\mathfrak{gl}_2}$ are in bijection with quadratic extensions of $F$: as discussed above in Example \ref{ex:restr}, for each quadratic extension $E$ of $F$ we get the torus $R_{E/F}\G_m$ in $\GL_2$. Its Lie algebra maps under Chevalley map onto the set 
$$(a_1, a_2)\in \A^2(F): {a_1^2-4a_2} \text{  is a square in } E.$$ 
We note that the image of the set of semisimple elements of $\fg$ is the complement of the origin in $\A^2(F)$, and the image of the set of regular semisimple elements is the complement of the locus $a_1^2-4a_2=0$. 
\end{example} 

This situation is general: all Cartan subalgebras become conjugate to $\ft$ over the algebraic closure of $F$; Chevalley map is defined over $F$, and on $F$-points, the images of $(\Lie S)(F)$ under Chevalley map cover a Zariski open subset of $\A_G(F)$ 
as $S$ runs over a set of representatives of the $F$-conjugacy classes of tori.  

Now we return to the group itself; here the situation gets more complicated because of the central isogenies.

\subsubsection{Semi-simple simply connected split group}
Assume that $\bG$ is split over $F$, and let $\bT$ be an $F$-split maximal torus of $\bG$.
We shall see that $\bT/W\simeq \ft/W = \A_G$ in this case. 
To do this,  we construct a basis for the coordinate ring of $\bT/W$ (see \cite[\S 3.3]{langlands-frenkel-ngo}). 
Let $\alpha_1, \dots \alpha_r$ be a set of simple roots for $\bG$ relative to $\bT$ (since $\bG$ is assumed to be semi-simple, the root lattice spans the same vector space as the character lattice $X^\ast(\bT)$, so there are $r$ simple roots). 
Let $\mu_i$ be the fundamental weights, as above, defined by $\mu_i(\alpha_j^\vee)=\delta_{ij}$ for $1\le i,j\le r$. 
We recall that for a semi-simple algebraic group, \emph{simply connected} means that the character 
lattice $X^\ast(\bT)$ coincides with the weight lattice, i.e. $\mu_i$ with $i=1, \dots r$ constitute  a $\Z$-basis of $X^\ast(\bT)$. 

Let $\rho_i$ be the algebraic representation of $G$ of the highest weight $\mu_i$  for $i=1, \dots r$, 
and let $a_i(t)=\tr \rho_i(t)$. 
These functions are algebraically independent over $F$ and $$\bT/W\simeq \spec F[a_1, \dots, a_r].$$  

As above, we get the map $\fc:\bG\to \A_G$, defined by $g\mapsto (\tr\rho_i(g))$. 
This map for the group is called Steinberg map.

\begin{example} As a baby example, take $\bG=\Sl_2$, with $\bT$ the torus of diagonal matrices, and let $\rho_1$ be its standard representation on $F^2$.  For $x\in \F^\times$, let $t(x)\in \bT(F)$ be the one-parameter subgroup of diagonal matrices,  $t(x)=\diag(x, x^{-1})$. 
Then the weights of $\rho_1$ are $\mu_1:=\left(\diag(x, x^{-1})\mapsto x\right)$ and 
$-\mu_1 = \left(\diag(x, x^{-1})\mapsto x^{-1}\right) $ (which form a single Weyl orbit). 
We have $a:=\tr(\rho_1)(t(x)) =x+x^{-1}$, and this is the coordinate on the affine line $\A^1=\A_{\Sl_2}$. 

More generally, for $\bG=\Sl_n$, we have $r=n-1$, and with the standard choice of simple roots $\alpha_i(\diag(x_1, \dots x_n))=x_i x_{i+1}^{-1}$, the above construction yields $\rho_1$ - the standard representation of $\Sl_n(F)$ 
on $F^n$, and $\rho_i=\wedge^i \rho_1$ (see \cite[\S 15.2]{fulton-harris:RepresentationTheory} for a detailed treatment over $\C$, which in fact works for algebraic representations over $F$). 
We recover the same `characteristic polynomial' map: the trace of the $i$-th alternating power of the standard representation applied to a diagonal matrix is precisely the $i$-th coefficient of its characteristic polynomial
(which is, up to sign, the degree $i$ elementary symmetric polynomial of the eigenvalues). 


(Note, however, that this is a coincidence that holds just for groups of type $A_n$: the isomorphism $\rho_i\simeq \wedge^i \rho_1$ does not hold for other types; 
we discuss this issue below in \S \ref{sub:naive_meas}). 
\end{example}

{\bf Caution:} Note that unlike the typical situation when one has an algebraic homomorphism of Lie algebras which then is `integrated' to obtain a homomorphism of simply connected Lie groups, Chevalley map on $\fg$ is \emph{not} the differential of Steinberg map (e.g. for $\Sl_2$, the map on $\fsl_2$ is $X\mapsto \det(X)$, while on $\Sl_2$ the map is $g\mapsto \tr(g)$).

\subsubsection{Split reductive group with simply connected derived subgroup}
Let $\bG$ be a split, reductive group of rank $r$, with simply connected derived group $\bG^\der$ (whose Lie algebra we will denote by $\fg^\der$). Let $\bZ$ be the connected component of the  centre of $\bG$. 
By our assumption that $\bG$ is split, $\bZ$ is a split torus. 
Let $T\supset Z$ be a split maximal torus in $G$, $T^\der = T \cap G^\der$ (note that $T^\der$ is \emph{not} the derived group of $T$), and let $W$ be the Weyl group of $G$ relative to $T$.  Let $\A_{G^\der}=T^\der/W$ 
be the Steinberg quotient for the semisimple group $\bG^\der$.  Let us denote $\rank(\bZ)$ by $r_Z$.
(Naturally, the most common situation is $r_Z=1$. )
We have $\A_{G^\der}\simeq \A^{r-r_Z}$.  
We have the exact sequence of algebraic groups [(3.1) in \cite{langlands-frenkel-ngo}]: 
\begin{equation}\label{eq:exseq}
\xymatrix{
1 \ar[r]  & {\mathbf A}  \ar[r] & \bZ\times \bG^\der 
 \ar[r]&  \bG \ar[r] & 1, \quad {\mathbf A}=\bZ\cap \bG^\der. 
}
\end{equation}
 For example, for $\bG=\GL_2$, the group ${\mathbf A}$ is the algebraic group $\mu_2$ of square roots of $1$; it is defined by the equation $x^2=1$. \footnote{It is important to think of ${\mathbf A}$ as a group scheme. As the authors point out, this group scheme presents an `annoying difficulty' in characteristic $2$ (by not being \'etale).} 

Na\"ively, then, one would like to define Steinberg-Hitchin base as $\A_{G^\der} \times \bZ$, 
and  establish a correspondence between the stable conjugacy classes in $G$ and the points of the base, as it was done for semi-simple simply connected groups.  
The obstacle is that we cannot really define a good map from $\bG$ to $\A_{G^\der} \times \bZ$ over $F$ by means of the exact sequence (\ref{eq:exseq}): first,  the decomposition $g=g'z$ with $g'\in G^\der$ and $z\in Z$ is defined only up to replacing $g'$ and $z$ with $ag'$, $az$ ($a\in {\mathbf A}(F)$), and second,  the map $(g', z)\to g'z$ is in general not surjective on $F$-points: for example for $\GL_2$, its image only consists of elements whose determinant is a square in $F$. 

The way to deal with this issue is described in \cite[(3.15)]{langlands-frenkel-ngo}: the set of $F$-points of the Steiberg-Hitchin base 
$\mathfrak A_G$ is defined as the union over cocycles 
$\eta\in H^1(F,A)$ of the spaces $(\mathfrak B_\eta(F)\times \bZ_\eta(F))/{\mathbf A}(F)$, where 
$\mathfrak B_\eta$ and $\bZ_\eta$ are torsors of, respectively, $\A_{G^\der}$ and $\bZ$, defined by the cocycle $\eta$. 

Finally, note that for the Lie algebra there is no issue because the Lie algebra actually splits as a direct sum $\fg=\fg^\der\oplus \mathfrak z$, and this is why we could treat all \emph{reductive} Lie algebras above on equal footing.

The situation is more complicated if $\bG^\der$ is not simply connected, as Steinberg quotient in this case will no longer be an affine space. We will not address this case (as well as the non-split case) in this note. 

\subsection{Weyl discriminant} We recall the definition and the basic properties of the Weyl discriminant (for the Lie algebra, the main source is  \cite[\S\S 7, 14]{kottwitz:clay}). 

\subsubsection{Weyl disriminant on the Lie algebra}\label{sub:weyl_discr}
\begin{definition}
Let $\fg$ be a reductive Lie algebra, let $X\in \fg$ be a regular semisimple element, and let $T=C_G(X)$ be its centralizer with the Lie algebra $\ft=\Lie (T)$. Then 
$$
D(X)= \det(\ad(X)\vert_{\fg/\ft})$$
is called the Weyl discriminant of $X$.
\end{definition}
The discriminant is, in fact, a polynomial function on $\fg$ (and thus extends to all of $\fg$ from the dense subset of regular semisimple elements): $D(X)$ is the lowest non-vanishing coefficient of the characteristic polynomial of $\ad(X)$ (see \cite[7.5]{kottwitz:clay}).
This interpretation allows us to give an intrinsic characterization of the set of regular semisimple elements: in fact,  
$X\in \fg$ is regular semisimple if and only if  $D(X)\neq 0$; thus it can be taken as a definition of \emph{regular semisimple}. 

We also recall the expression for $D(X)$ in terms of roots: 
\begin{equation}\label{eq:Droots}
D(X)=\prod_{\alpha\in \Phi} \alpha(X)=
(-1)^{\frac{\dim\fg-\rk\fg}{2}}\left(\prod_{\alpha\in \Phi^+} \alpha(X)\right)^2,
\end{equation}
where $\Phi$ is the set of all roots and $\Phi^+ $ is any set of positive roots.

\begin{example}\label{ex:Weyl_disc} 
We compute the explicit expressions for the Weyl discriminant in terms of the eigenvalues of $X$, in the cases 
$\fg=\fsl_n$ and $\fg=\fsp_{2n}$,  for use in future examples.

Ler $X\in \fg$ have eigenvalues $\lambda_i\in \bar F$. 
 
For $\fg=\fsl_n$, the roots are $\alpha_{ij}(X)= \lambda_i-\lambda_j$, $1\le i, j \le n \text{ and  } i\neq j$. 
Then the Weyl discriminant of $X$ coincides with the polynomial discriminant of the characteristic polynomial of $X$: 
$$D(X)=\prod_{{1\le i, j \le n} \atop{ i\neq j}}(\lambda_i-\lambda_j).$$
(We observe that the eigenvalues satisfy the relation $\sum_{i=1}^n\lambda_i=\tr(X)=0$).

For $\fg=\fsp_n$, the explicit expression for the roots depends on the choice of the coordinates for the standard representation (though of course the answer does not). 
We define $\Sp_{2n}$ and $\fsp_{2n}$ explicitly as:
$$ \Sp_{2n}(F)=\{g\in \GL_{2n} (F): g^{t} J g =J\},  \quad   \fsp_{2n}(F)=\{X\in \fgl_{2n} (F): X^{t} J + JX =0\},$$
where $J=\left[\begin{smallmatrix}0 & I_n \\ -I_n & 0 \end{smallmatrix}\right]$ and $I_n$ stands for the $n\times n$-identity matrix.
Then the eigenvalues of any element $X\in \fsp_{2n}$ satisfy 
$\lambda_{n+i}=-\lambda_i$, $1\le i \le n$, and 
the set of values of the roots at $X$ is (cf. \cite[\S 12.1]{humphreys}): 
$\{\pm (\lambda_i\pm \lambda_j), 1\le i, j \le n, i\neq j\}\cup \{\pm 2\lambda_i, 1\le i\le n \}$. 
Then we get: 
$$D(X)=
(-1)^{\frac{n(n+1)}2}2^n\prod_{{1\le i, j \le n} \atop{ i\neq j}}(\lambda_i^2-\lambda_j^2)\prod_{i=1}^n \lambda_i.$$
\end{example}

We also observe that in a reductive Lie algebra, the Weyl discriminant of any element is computed entirely via the derived subalgebra $\fg^\der$, by definition (since $\fg/\ft=\fg^\der/\ft^\der$). 

\subsubsection{Weyl discriminant on the group}
On the group, the definition is obtained essentially by reducing to the Lie algebra: 
\begin{definition}
Let $\gamma \in \bG(F)$ be a regular semisimple element, and let $T=C_G(X)$ be its 
centralizer with the Lie algebra $\ft=\Lie (T)$. Then the Weyl discriminant of $\gamma$ is  
$$
D(\gamma)= \det(1-\Ad(\gamma)\vert_{\fg/\ft}).$$
\end{definition}
Similarly to the Lie algebra case, the Weyl discriminant has an expression in terms of the (multiplicative) roots: 
\begin{equation}\label{eq:disc_group}
D(\gamma)=\prod_{\alpha\in \Phi} (1-\alpha(\gamma)) =
(-1)^{\frac{\dim\fg-\rk\fg}{2}}\rho^2(\gamma)\left(\prod_{\alpha\in \Phi^+}(1-\alpha(\gamma))\right)^2,
\end{equation} 
where $\rho$ is half the sum of positive roots, so $2\rho$ is the sum of positive roots (in the above formula, 
$\rho^2(\gamma)$ is the value of the character $2\rho$ at $\gamma$).
Note that the second part of the formula 
expressing the Weyl discriminant as a product over \emph{positive} roots now has an extra factor that did not arise in the Lie algebra case (the examples below illustrate this).  
 
We again show the calculation for the general linear and symplectic groups. 
Note that the final expressions are a lot simpler when restricted to $\bG^\der$. 

\begin{example}\label{ex: disc_group}
In all examples, we give  an explicit expression for the Weyl discriminant of a regular semisimple element 
$\gamma\in G(F)$ with eigenvalues $\{\lambda_i\}\subset \bar F$. We observe that these expressions do not depend on the field (so one could even consider $F=\C$). 
\begin{enumerate}
\item $G=\GL_2$:  
$D(\gamma)=(1-\frac{\lambda_1}{\lambda_2})(1-\frac{\lambda_2}{\lambda_1}) = -\frac{(\lambda_1-\lambda_2)^2}{\det(\gamma)}$. 
\item $G=\GL_n$: $D(\gamma)= \prod_{{1\le i, j \le n} \atop{ i\neq j}}\left(1-\frac{\lambda_i}{\lambda_j}\right)$. 

\item $G=\Sp_{2n}$: 
$D(\gamma)= 
 \prod_{1  \le i < j \le n} d_{ij} \cdot \prod_{1 \le i \le n}
d_{i}$, 
{where}
$$\begin{aligned}
d_{ij} &=
\left(1-\frac{\lambda_i}{\lambda_j}\right)\left(1-\frac{\lambda_j}{\lambda_i}\right)\left(1-\lambda_i\lambda_j\right)\left(1-\frac1{\lambda_i\lambda_j}\right)\\
d_i &= \left(1-\lambda_i^2\right)\left(1-1/\lambda_i^2\right).
\end{aligned}
$$

\item $G={\mathrm {GSp}}_{2n}$.  By definition, ${\mathrm {GSp}}_{2n}(F)$ is the algebraic group whose functor of points is 
defined as, for any $F$-algebra $R$,  
$${\mathrm {GSp}}_{2n}(R)= \{g\in \GL_{2n}(R): \exists \nu(g)\in R^\times,  g^tJg =\nu(g) J\},$$ 
where $J$ is the same matrix as the one used to define $\Sp_{2n}$.
It fits into the exact sequence of algebraic groups 
$$1\to \Sp_{2n}\to {\mathrm {GSp}}_{2n} \to \G_m \to 1,$$
where the map to $\G_m$ is the map $g\mapsto \nu(g)$, called the \emph{multiplier}.
We have ${\mathrm {GSp}}_{2n}^\der =\Sp_{2n}$, so ${\mathrm {GSp}}_{2n}$ is a good example (other than $\GL_n$) of a reductive but not semi-simple algebraic group whose derived subgroup is simply connected.  

If the element $\gamma$ has multiplier $\nu$, then as above for $G=\Sp_{2n}$, 
$D(\gamma)= 
 \prod_{1  \le i < j \le n} d_{ij} \cdot \prod_{1 \le i \le n}
d_{i}$, 
{but now we have }
$$\begin{aligned}
d_{ij} &=
\left(1-\frac{\lambda_i}{\lambda_j}\right)\left(1-\frac{\lambda_j}{\lambda_i}\right)\left(1-\frac{\lambda_i\lambda_j}{\nu}\right)\left(1-\frac{\nu}{\lambda_i\lambda_j}\right)\\
d_i &= \left(1-\frac{\lambda_i^2}{\nu}\right)\left(1-\frac{\nu}{\lambda_i^2}\right).
\end{aligned}
$$
\end{enumerate} 
\end{example}

\subsection{Orbital integrals: the Lie algebra case}\label{sub:lie} 
We start with a prototype case of a Lie algebra. 
\subsubsection{Definitions: Lie algebra}\label{sub:def_lie} 
Let $\bG$ be a connected reductive group defined over a local field $F$, as above. The orbital integrals  (for regular semisimple elements) on the Lie algebra 
are distributions on the space $C_c^\infty(\fg)$ of the locally constant compactly supported functions on $\fg$, defined as follows.

Let $X\in \fg$ be a regular semisimple element, and let $f\in C_c^\infty(\fg)$. 
Since $X$ is regular semisimple, its centralizer is a torus $T=C_G(X)=\bT(F)$, as discussed above, and thus the adjoint orbit of $X$ can be identified 
with the quotient $\bT(F)\backslash \bG(F)$.  
Both $\bT(F)$ and $\bG(F)$ can be  endowed with any of the natural measures discussed above  in
\S\ref{subseq:the_problem}. Once the measures on $G$ and $T$ are fixed, there is a unique quotient measure on $T\backslash G$, which
we will denote by  $\mu_{T\backslash G}$ (see e.g.,  \cite[\S 2.4]{kottwitz:clay} for the definition of the quotient measure in this context). 
The  orbital integral  with respect to this measure is
\begin{equation}\label{eq:oi}
O_X(f) :=\int_{T\backslash G} f(\Ad(g^{-1})X) d\mu_{T\backslash G}.
\end{equation}
We observe that there are finitely many $F$-conjugacy classes of tori in $G$; thus there are finitely many choices of measures that we need to make on the representatives of these conjugacy classes, and these choices endow the orbit 
of every regular semisimple element with a measure. If the canonical measures (in the sense of \cite{gross:motive}, discussed above in \ref{subsub:Neron}) are chosen on the tori, the resulting orbital integrals are called \emph{canonical}.  
This approach to the normalization of measures on the orbits is the one typically used in the literature. 

On the other hand, one can use Chevalley map defined above to normalize the measures on orbits. 
For a general reductive group $\bG$ and $X\in \fg$ regular semisimple, the fibre  $\fc^{-1}(\fc(X))$
over the point $a:=\fc(X)\in \A_G(F)$ 
is the stable orbit of $X$, which is a finite union of $F$-rational orbits. Thus, every $F$-rational orbit is an open subset of $\fc^{-1}(a)$ for some $a\in \A_G(F)$, and if we define  a measure on this fibre, we get a measure on every $F$-rational orbit contained in it by restriction. 

In the Introduction, we have fixed measures on affine spaces \emph{with a choice of a basis}.  
The Lie algebra $\fg$ is an affine space; it does not come with a canonical choice of a basis, and this choice would not matter much in the discussion below; we can choose an arbitrary $F$-basis $\{e_i\}_{i=1}^n$ of  $\fg$ for our purposes. 
This basis then gives rise to a differential form $\omega_{\fg}=\wedge_{i=1}^n dx_i$ on $\fg$, which gives a measure $|\omega_\fg|$ as in the Introduction. 
The space $\A_G$ is also an affine space under our assumptions (since at the moment we are working with the Lie algebra);  and in our construction it comes with a choice of basis $\{\rho_i\}_{i=1}^r$ as in \S \ref{subsub:Liealg}.  We let $\omega_{\A_G}$ be the differential form associated with this basis. 

Thus we get the quotient measure on each fibre  $\fc^{-1}(\fc(X))$: it is the measure associated with the differential form $\omega_{\fc(X)}^\geom$ such that 
\begin{equation}\label{eq:geom_lie1}
\omega_\fg=\omega^{\geom}_{\fc(X)}\wedge \omega_{\A_G}.
\end{equation}
That is, by definition of $\omega_{\fc(X)}^{\geom}$,  for any $f\in C_c^\infty(\fg)$,
\begin{equation}\label{eq:geom_lie2}
\int_{\fg} f(X)\, d|\omega_{\fg}|=\int_{\A_G}\int_{\fc^{-1}(\fc(X))}f(X)\,d |\omega^{\geom}_{\fc(X)}| 
\,d |\omega_{\A_G}|.
\end{equation}

Our immediate goal is to
derive the relationship between these two measures on the orbit: $\mu_{T\backslash G}$ and $|d \omega_a^{\geom}|$, where $a=\fc(X)$. 
First we observe that since both measures are quotient measures of a chosen Haar measure on $G$, their ratio does not depend on the choice of the measure on $G$, as long as it is compatible with the choice of the measure on the Lie algebra; thus at this point, the choice of the measure on $G$ is determined by our choice of the form 
$\omega_{\fg}$. (Conversely, one often chooses a measure on $G$ first, and this determines $\omega_\fg$.)  
At the same time, the choice of the measures on the representatives of conjugacy classes of tori affects the measure 
$\mu_{T\backslash G}$ but not the measure $|\omega_a^{\geom}|$. 
Here we address two natural choices of such measures: 
\begin{enumerate}[(i)]
\item Let $\omega_G$ be a volume form on $G$, and on each algebraic torus $T$, define the form $\omega_T$ using the characters of the torus as in (\ref{eq:omegaT}). 
Then we get the quotient measure $\omega_{T\backslash G}$ on each orbit. 
This is the measure discussed in \cite{langlands-frenkel-ngo}.
We discuss this measure in this section.   

\item Use the measure denoted above by $|\omega^\can|$, associated with the N\'eron model, on each torus. This measure on the orbits is discussed in the next section.  

\end{enumerate} 

Thus our  first goal is to determine the ratio of $\omega_{T\backslash G}$ to $\omega_a^\geom$, for each $X\in \fg^\rss$ (which determines $T$ and $a$). It turns out that the  conversion between these two measures is based on exactly the same calculation as the Weyl integration formula, which we now review.
\subsubsection{Weyl integration formula, revisited}
We follow \cite[{\S 7, \S14.1.1}]{kottwitz:clay}, and use the same notation (except we continue to use boldface letters to denote varieties). 
For a torus $T\subset G$, let $W_T={\mathbf N_G(T)}(F)/\bT(F)$ be the relative Weyl group of $T$ (cf. \cite[{\S 7.1}]{kottwitz:clay}). 
Weyl integration formula (which we quote in this form from \cite[\S 7.7]{kottwitz:clay}), 
for an arbitrary Schwartz-Bruhat function $f\in C_c^\infty(\fg)$, asserts:
\begin{equation}\label{eq:Weyl}
\int_\fg f(Y) d|\omega_\fg(Y)| =\sum_T \frac 1{|W_T|}\int_{\ft}|D(X)|
\int_{T\backslash G} f(\Ad(g^{-1}) X ) d|\omega_{T\backslash G}| \, d|\omega_\ft(X)|,
\end{equation}
where the sum on the right-hand side  is over the representatives of the conjugacy classes of tori in $G$. 

The proof of this formula relies on a computation of the Jacobian of the map 
\begin{equation}\label{eq:themap}
 \begin{aligned}(T\backslash G)\times \fg &\to \fg \\
 (g, X) &\mapsto \Ad(g^{-1})X.
 \end{aligned}
 \end{equation} 
This map is $|W_T|:1$, and its  Jacobian  at $X$ is precisely $|D(X)|$ (see \cite[\S 7.2]{kottwitz:clay} for a beautiful exposition).  

\subsubsection{The relation between geometric and canonical orbital integrals for the Lie algebra}
Let us just na\"ively compare the right-hand side of the Weyl integration formula with the right-hand side of 
(\ref{eq:geom_lie2}) above (since the left-hand sides are the same). 
First, note (as already discussed above) that our space $\A_G$ is,  up to a set of measure zero, a disjoint union of images of the representatives of the conjugacy classes of tori, and for each torus, Chevalley map $\fc: T\to \A_G$ is $|W_T|:1$. 
Thus, the right-hand side of (\ref{eq:geom_lie1}) would look exactly like the right-hand side of the Weyl integration formula (\ref{eq:Weyl}) if we could replace integration over $\A_G$ with the sum of integrals over the representatives of the conjugacy classes of 
Cartan subalgebras $\ft=\Lie(T)$ (as $T$ ranges over the conjugacy classes of maximal tori).

The situation is summarized by  the commutative diagram:
\begin{equation}\label{eq:diagram}
\xymatrix{
&(T\backslash G)\times \ft \ar[d] \ar[r] & \fg \ar[d] \\
&\ft \ar[r] & \A_G=\ft^\spl/W
}
\end{equation}
Here the horizontal map on the top is the map (\ref{eq:themap}); this map is $|W_T|:1$ and its Jacobian at $(g, X)$  
is $|D(X)|$ (see \cite[{\S 7.2}]{kottwitz:clay}). 
The vertical arrow on the left is projection onto $\ft$; the vertical arrow on the right is Chevalley map $\fc$; 
and the horizontal arrow at the bottom is $\fc \vert_{\ft}$, which is also $|W_T|:1$ . 

We have the forms $\omega_\fg$ on $\fg$ and $\omega_{\A_G}$ on $\A_G$;
let us choose the invariant differential form
$\omega_T$ on $T$ defined by the characters of $T$ as in (\ref{eq:omegaT}); we also need an invariant top degree form  $\omega_G$ on $G$, which is required to be 
compatible with $\omega_\fg$ under the exponential map, which determines it uniquely. 
As discussed above, given $\omega_G$ and $\omega_T$, we get the quotient measure $|\omega_{T\backslash G}|$ that corresponds to a differential form 
$\omega_{T\backslash G}$ satisfying 
$\omega_T\wedge \omega_{T\backslash G} =\omega_G$, and a differential form $\omega^{\geom}_{\fc^{-1}(a)}$ on each fibre of the map $\fc:\fg \to \A_G$. 
Both $\omega_{T\backslash G}$ and $\omega_{\fc^{-1}(a)}^{\geom}$ are generators of the top exterior power of the cotangent bundle of $\bT\backslash \bG$, hence they differ by a constant (which can depend on $a$). 

Looking at the top, right and bottom maps in the diagram (\ref{eq:diagram}), respectively, we see that these differential forms are related as follows (the first and third lines follow from the Jacobian formula and the fact that the horizontal maps are $|W_T|:1$; 
the second line is the definition of $\omega^{\geom}_{\fc^{-1}(a)}$ with $a=\fc(X)$): 
\begin{equation}\label{eq:measures}
\begin{aligned}
 & |W_T|^{-1} |D(X)| \omega_{T\backslash G}\wedge \omega_{\ft} = \omega_{\fg}\\
 & \omega_{\fg}(X)=\omega^{\geom}_{\fc^{-1}(\fc(X))}\wedge \omega_{\A_{G}}(\fc(X))
 \\
 & |W_T|^{-1} |\Jac(\fc\vert_{\ft})| \omega_{\ft}=\omega_{\A_G}.
\end{aligned}
\end{equation} 
We conclude that 
\begin{equation}\label{eq:jac_chev}
|D(X)|\omega_{T\backslash G}(X)= |\Jac(\fc\vert_{\ft}) (X)| \omega^{\geom}_{\fc^{-1}(\fc(X))}, \quad  X\in \ft.
\end{equation}
The Jacobian of the restriction of Chevalley map $\fc \vert_{\ft}$ at $X$ is $\prod_{\alpha\in \Phi^+}\alpha(X)$, \emph{up to a constant in $F^\times$} (see \cite[\S 14.1]{kottwitz:clay}).  This constant depends on the choice of coordinates on $\ft$. 
We use the basis of the character lattice  $\{\chi_i\}_{i=1}^r$, as in (\ref{eq:omegaT}), to define the coordinates on $\ft$. 
With this choice of coordinates, the constant turns out to be $\pm 1$; 
the sign depends on the ordering of the characters $\chi_i$ and does not affect the resulting measure.  
The reason  for this is that the constant is $1$ for the split torus (this is not trivial; it follows from the argument in 
\cite[ch.5,\S5]{Bourbaki:Lie}, and the group version of this statement is also proved in \cite{langlands-frenkel-ngo} over $\C$ (see proof of Proposition 3.29, especially (3.33) and (3.34)); the argument holds for any split torus). 
If $T$ is not split,  
we can work over an extension $E$ where $T$ splits, and since our coordinate system is precisely the one 
used for the split torus over $E$, 
the equality continues to hold.  
We observe that on the Lie algebra, we have 
\begin{equation}\label{eq:root_discr}
|\prod_{\alpha\in \Phi^+}\alpha(X)| = |D(X)|^{1/2}.
\end{equation}

Putting the relations (\ref{eq:measures}), (\ref{eq:jac_chev}), and (\ref{eq:root_discr}) together, we obtain the following Proposition.
\begin{proposition}\label{prop:cT}(cf. \cite[Proposition 3.29]{langlands-frenkel-ngo}.)
 Let $\fc:\fg\to \A_G$ be Chevalley map as above; let $X\in \fg$ be a regular semisimple element, let 
the algebraic torus $T$ be its centralizer, with the Lie algebra $\ft$. 
Then with the measures defined as above, we have:  
\begin{equation*}\label{eq: final_liealg}
|\omega^{\geom}_{\fc^{-1}(\fc(X))}|=|D(X)|^{1/2} |\omega_{T\backslash G}|. 
\end{equation*}
\end{proposition}

We conclude this section with an  example illustrating the proposition. 
\begin{example} Let $\fg=\fsl_2$ and let $\ft=\ft^{\spl}$ be the subalgebra of diagonal matrices.
Then we have $\fc:X(t):=\left[\begin{smallmatrix} t & 0\\ 0& -t \end{smallmatrix}\right] \mapsto -t^2$; here the Jacobian is just the derivative (since we are dealing with a function of one variable), so 
$\Jac(\fc\vert_{\ft})=-2t = -\alpha(X(t))$. 

Now consider $\ft_E$ -- a non-split Cartan subalgebra corresponding to a quadratic extension 
$E=F[\sqrt{\epsilon}]$:
$\ft_E= \left\{X(t):=\left[\begin{smallmatrix} 0& t\\ \epsilon t & 0  \end{smallmatrix}\right], \ t\in F  \right\}$. 
We have $\fc\vert_{\ft_E} = -\epsilon t^2$, and its Jacobian is 
$-2\epsilon t =-\sqrt{\epsilon} \alpha(X(t))$ (note that the eigenvalues of our element are $\pm\sqrt{\epsilon} t$). 
At the same time, on $\ft_E$, the measure $\omega_{T}$ is $\sqrt{\epsilon}dt$ (recall that $\omega_T$ is defined by means of characters of $T$ over the algebraic closure). Hence, with this choice of the differential form,  
we obtain, again, with $a=\fc(X(t))$: 
$$da = -\sqrt{\epsilon} \alpha(X(t)) dt = -\alpha(X(t)) \omega_T(X(t)).$$
\end{example}

\subsection{The simplest group case}\label{sub:group}
 Let us assume that $\bG$ is semi-simple, split, and  simply connected. 
We are now almost ready to explain the relation (3.31) of \cite{langlands-frenkel-ngo} (see equation (\ref{eq:FLN}) above). 
The definitions are essentially the same as in the Lie algebra case: 
\begin{itemize}
\item an orbit of a regular semisimple element $\gamma\in G:=\bG(F)$, as a manifold over $F$, can be identified with $T\backslash G$, where $T$ is the centralizer of $\gamma$. As above, if $\omega_G$ is a volume form on $G$ and $\omega_T$ - a volume form on $T$, we get the measure $|\omega_{T\backslash G}|$ on the orbit of $\gamma$. 
\item The regular fibres of the map $\fc: G\to \A_G$ are stable orbits; each stable orbit of a regular semisimple element is a finite disjoint union of $F$-rational orbits, and thus we get the geometric measure $|\omega_a^{\geom}|$ on each such orbit, by considering the quotient of the measures on $G$ and $\A_G$.
\end{itemize}

For $\gamma\in G$, let 
$$\Delta(\gamma):= \rho^{-1}(\gamma)\prod_{\alpha>0}(1-\alpha(\gamma)); \quad \text{thus } |\Delta(\gamma)|^2=|D(\gamma)|.$$ 
\begin{theorem}(\cite [Relation (3.31)]{langlands-frenkel-ngo}.) Let $\bG$ be a connected semi-simple simply connected group over a local field $F$, and let $\gamma\in G$  be a regular semisimple element.  Then for any $f\in C_c^\infty(G)$, the orbital integrals 
with respect to the geometric measure on the orbit of $\gamma$, and the measure $\omega_{T\backslash G}$ (which, by definition, is the quotient of the measures $|\omega_G|$ on $G$ and $|\omega_T|$ on $T$, with $\omega_T$ defined by (\ref{eq:omegaT}))  are related via: 
$$ \int_{\ri(\gamma)}f(g)d|\omega_{\fc(\gamma)}^{\geom}(g)|= |\Delta(\gamma)|
\int_{T\backslash G}f(\Ad(g^{-1})\gamma)\omega_{T\backslash G}(g),$$
where on the left, the orbit $\ri(\gamma)$ is thought of as an open subset of the stable orbit $\fc^{-1}(\fc(\gamma))$ and endowed with the geometric measure $|\omega_{\fc(\gamma)}^{\geom}(g)|$ as above. 
\end{theorem}

We first explain two differences with the statement in \cite{langlands-frenkel-ngo}. 
\begin{remark} Our expression does not (yet) include the factor 
$L(1, \sigma_{T\backslash G})$ that appears in (3.31) of \cite{langlands-frenkel-ngo}.  
This factor appears simply by their definition of the measure $d\bar{g}_v:=L(1, \sigma_{T\backslash G})\omega_{T\backslash G}$ which appears on the right-hand side of (3.31).  
As we shall see in the next section, using the measure $d\bar{g}_v$ ensures that the local orbital integral on the right-hand side is $1$ for almost all places of a given number field (which is desirable for defining the orbital integral globally), 
and for almost all places this coincides with the orbital integral with respect to the canonical measure. 
%
\end{remark}

\begin{remark} Note that we stated the theorem as a relation between orbital integrals, whereas in \cite{langlands-frenkel-ngo} it is stated as a relation between \emph{stable} orbital integrals. Since the measure is a local notion, this is an equivalent statement: in fact, the assertion of the theorem is just that 
the two measures on the stable orbit (and hence, by restriction, on every rational orbit) are related via 
$$|\omega_{\fc(\gamma)}^{\geom}(g)| = |\Delta(\gamma)|
|\omega_{T\backslash G}(g)|.$$
\end{remark}

\subsubsection{Sketch of the proof.} As the measures are defined by differential forms, the calculation is carried out in the exterior power of the cotangent space, and hence it is essentially the same calculation as for the Lie algebra above. The only ingredients that needs to be treated slightly differently are the discriminant and the Jacobian of the map from $T$ to 
$T/W$. 
Indeed, for $\gamma\in G^\rss(F^\sep)$,  we still have the exact sequence of tangent spaces (see \cite[Lemma 26]{langlands-frenkel-ngo})
$$0\to Tan_\gamma(\fc^{-1}(a)) \to  Tan_\gamma \bG \to Tan_a(\A_G) \to 0,$$
and by definition, 
$\omega_{\fc(\gamma)}^\geom\wedge \omega_{\A_G} =\omega_G$; $\omega_G=\omega_T\wedge \omega_{T\backslash G}$. 
The proof proceeds exactly as for Lie algebras, except the map (\ref{eq:themap}) needs to be replaced with the map 
\begin{equation}\label{eq:themapgroup}
 \begin{aligned}(T\backslash G)\times G &\to G \\
 (g, \gamma) &\mapsto g^{-1}\gamma g, 
 \end{aligned}
 \end{equation} 
and the map $\ft\to \A_G=T/W$ is replaced with the map $T\to T/W$. 
The Jacobian of the first map is the group version of the Weyl discriminant (and fits into the group version of Weyl integration formula in the exact same way as it did for the Lie algebra): 
$$|W_T|^{-1} |D(\gamma)| \omega_{T\backslash G}\wedge \omega_{T} = \omega_{G}.$$
Next, we need to relate $\omega_T$ with $\omega_{\A_G}$. 
\begin{lemma}\label{lem:JacT}
(\cite[Proposition 3.29]{langlands-frenkel-ngo}.) Let $\bG$ be a split, semi-simple, simply connected group, and let 
$T\subset G$ be a maximal torus.
Let 
$\omega_T$ be defined by (\ref{eq:omegaT}). 
Then 
$$|W_T|^{-1}\omega_T(\gamma) = |\Delta(\gamma)|\omega_{\A_G}(\fc(\gamma)).$$
\end{lemma} 
\begin{proof} For $T= T^\spl$, this is proved in \cite{langlands-frenkel-ngo}, as well as in \cite{Bourbaki:Lie} (where the field is assumed algebraically closed, but the proof works verbatim for the split torus). 
Now it remains to consider the restriction of $\fc$ to an arbitrary (not necessarily split) maximal torus. 
The map $\fc$ on $T$ can be defined as a composition 
$$\bT\to \bT^\spl\to \A_G=\bT^\spl/W,$$ where the first map is an isomorphism over the algebraic closure of $F$. 
The pullback of the form $\omega_{T^\spl}$ on $T^\spl$ is precisely the form $\omega_T$ on $T$, and thus the equality remains true.
\end{proof}
 The theorem follows, precisely as in the Lie algebra case. 
To conclude this section, we compute some examples illustrating the above Lemma (which show that it is substantially non-trivial even for the split torus).  

\subsubsection{Examples of Jacobians and discriminants on the group}

\begin{example} We again start with $\bG=\Sl_2$. 
Let $\gamma_t= \left[\begin{smallmatrix} t & 0\\ 0& t^{-1} \end{smallmatrix}\right]$. 
The map $\fc$ on the diagonal torus is given by: 
$\fc:\left[\begin{smallmatrix} t & 0\\ 0& t^{-1} \end{smallmatrix}\right] \mapsto t+t^{-1}$. 
Its Jacobian (i.e., the derivative) is $1-t^{-2}$; so we get: 
$$
\Jac(\fc\vert_{T^{\spl}})(\gamma_t)=
1-t^{-2}=(1-(-\alpha)(\gamma_t))=t^{-2}(1-\alpha(\gamma_t)).
$$
We observe that for $\SL_2$,  the half-sum of positive roots is $\rho=\frac12\alpha$, so 
$\rho(\gamma_t)= t$.

However, this is not yet the whole story. We are interested in the ratio between the measure $\omega_T=\frac{dt}t$ on $T$ and the measure $da$ on $\A^1\simeq T/W$, and our map, as above, is given by $a=t+t^{-1}$. 
We just computed: $da=(1-t^{-2})dt$. Then we have: 
$$\frac{dt}{t}=\frac{(1-t^{-2})^{-1}}{t} da = \rho(\gamma_t) \prod_{\alpha>0} (1-\alpha(\gamma_t))^{-1} da.$$ 
\end{example} 

It is instructive to do one more, higher rank, example. 
\begin{example} Let $\bG=\Sp_4$ (defined explicitly as in Example \ref{ex:Weyl_disc} above), and consider the split torus $T=\{\diag(t_1, t_2, t_1^{-1}, t_2^{-1})\mid t_i\in F^\times \}$.

Let $\gamma_{t_1, t_2}= \diag(t_1, t_2, t_1^{-1}, t_2^{-1})$. 
In these coordinates, Steinberg map is given explicitly by  the elementary symmetric polynomials: 
$$
\begin{aligned}
&\fc: \gamma_{t_1, t_2} \mapsto (a,b), \\
& a=t_1+t_2+t_1^{-1}+t_2^{-1}, \quad b=t_1t_2+t_2t_1^{-1}+t_1t_2^{-1}+t_1^{-1}t_2^{-1}+2. 
\end{aligned}
$$
\end{example}
The Jacobian of this map is (we are skipping the details of a painful calculation)
\begin{equation*}
\left|\begin{smallmatrix} \frac{\partial a}{\partial t_1} & \frac{\partial b}{\partial t_1}\\
\frac{\partial a}{\partial t_2} & \frac{\partial b}{\partial t_2} 
\end{smallmatrix} \right| = 
(1-t_1^{-2})(1-t_2^{-2})(1-t_1^{-1}t_2^{-1})(t_1-t_2),
\end{equation*} 
which we recognize as:
$$\Jac (\fc\vert_T)(\gamma_{t_1, t_2})= t_1\prod_{\alpha<0}(1-\alpha(\gamma_{t_1, t_2})),$$
Note the factor of $t_1$ in front (which is not a root value). 
Thus we obtained: 
$$\begin{aligned} 
&dt_1\wedge dt_2 = \pm \prod_{\alpha<0} (1-\alpha(\gamma_{t_1,t_2}))^{-1} t_1^{-1} da\wedge db \\
&= \pm\left( \prod_{\alpha>0} (1-\alpha(\gamma_{t_1, t_2}))^{-1} \right)\rho^2(\gamma_{t_1, t_2}) t_1^{-1} da\wedge db.
\end{aligned}$$
The plus-minus sign in front is not important and depends on the ordering of the coordinates. 

Now, we are interested in the ratio between the invariant measure $\omega_T= \frac{dt_1\wedge dt_2}{t_1 t_2}$ on $T$ and the measure $d \omega_{\A_G}=da\wedge db$ on $T/W$. 
We note that in this case $\rho$, the half-sum of positive roots, evaluated at $\gamma_{t_1, t_2}$  is 
$\rho(\gamma_{t_1, t_2})= \left(t_1^2 t_2^2 \frac{t_1}{t_2} (t_1 t_2)\right)^{1/2} = t_1^2 t_2$, and compute further (here we write $\gamma:=\gamma_{t_1, t_2}$ to avoid notational clutter): 
\begin{equation*}
\begin{aligned}
\omega_T & = \frac{dt_1\wedge dt_2}{t_1 t_2} = \frac{\pm\left( \prod_{\alpha>0} (1-\alpha(\gamma))^{-1} \right)\rho^2(\gamma) t_1^{-1}}{t_1 t_2} da\wedge db \\
 & = \pm\left( \prod_{\alpha>0} (1-\alpha(\gamma)^{-1} \right)\rho(\gamma)  da\wedge db =\Delta(\gamma)  da\wedge db. 
\end{aligned}
\end{equation*}

\subsection{The general case}
First, suppose that $\bG$ is a connected  split reductive group over $F$, with $\bG^\der$ simply connected. 
Then if one uses the correct general notion of Steinberg-Hitchin base as defined in \cite{langlands-frenkel-ngo}, all measures are invariant under the action of the centre, and hence relation (3.31) of \cite{langlands-frenkel-ngo} holds in this case as well, with no further proof needed. 

If $\bG^\der$ is not simply connected, the space we denoted by $\A_{G^\der}$ is no longer an affine space, and one needs to use $z$-extensions. 
If $\bG$ is not split, we need to consider Galois action on Steinberg-Hitchin base. Both topics are discussed in \cite{langlands-frenkel-ngo} but are beyond the scope of these notes.

\subsection{Aside: na\"ive approach for classical groups -- what works and what doesn't}\label{sub:naive_meas}
 Suppose for a moment that $G\hookrightarrow \GL_n(F)$ 
is a split classical reductive group. It is tempting (and often done in Number theory\footnote{For example, \cite{gekeler03}, \cite{achterwilliams15}, \cite{davidetal15}, \cite{achter-altug-garcia-gordon}, etc. 
A reader not interested in this type of a calculation can safely skip this section.})
to  
still try to  use the coefficients of the characteristic polynomial to define the  maps from $\fg$ and $G$ to $\A_G$.
This works (with further caution discussed below) for the groups of type  $A_n$, $C_n$ and $B_n$, but does not quite work for type $D_n$.

First, consider Chevalley map on the Lie algebra. 
 
If $\fg =\mathfrak{sp}_{2n}$, then the characteristic polynomial of any element 
$X\in \fg$ has the form:
$f_X(t)=t^{2n}+a_1 t^{2n-2}+\dots +a_{n}$. 
 We can define $\fc'_{\fsp}(X)$ to be the tuple of coefficients 
$(a_1, \dots, a_{n})\in \A^{n}$.  
The relationship between this map and Chevalley map is determined by the relation between the fundamental representation 
$\rho_i$ and the $i$-th exterior power of the standard representation $\wedge^i \rho_1$, for $i=1,\dots, n$. 
For the symplectic Lie algebra, it turns out that $\wedge^i \rho_1$ is a direct sum of $\rho_i$ and some representations of lower 
highest weights (see e.g. \cite[Theorem 17.5]{fulton-harris:RepresentationTheory}). Hence, the transition matrix between the characters of 
$\rho_i$ and the characters of $\wedge^i \rho_1$ (i.e., the coefficients of the characteristic polynomial up to sign) is upper-triangular with $1$s on the  
diagonal.  Therefore, we get a measure-preserving isomorphism between the affine space $\A_G$ with coordinates $\tr(\rho_i)$ and the affine space $\A_G'$ with the coordinates given by the coefficients of the characteristic polynomial. 
This implies that the map $\fc'_{\fsp}$ could be used instead of $\fc_{\fsp}$ in all the calculations, without affecting the results. 

For the odd orthogonal Lie algebras $\fso_{2n+1}$, the exterior powers $\wedge^i\rho_1$ are irreducible for $1=1, \dots, n$, and for $i=1, \dots, n-1$, coincide with the first $n-1$ fundamental representations; however, the last fundamental representation, the \emph{spin representation} is not obtained this way (see \cite[Theorem 19.14]{fulton-harris:RepresentationTheory}). 
For the even orthogonal Lie algebra $\fso_{2n}$, the representations $\wedge^i\rho_1$ are irreducible for $1=1, \dots, n-1$, and for $i=1, \dots, n-2$, coincide with the first $n-2$ fundamental representations, and $\wedge^n\rho_1$ decomposes as a direct sum of two irreducible representations whose weights are \emph{double} the fundamental weights (see \cite[Theorem 19.2]{fulton-harris:RepresentationTheory}).
Nevertheless,  for the odd orthogonal Lie algebras, the coefficients of the characteristic polynomial still distinguish the stable conjugacy classes of regular semisimple elements; for the even orthogonal Lie algebras, one needs to add the \emph{pfaffian}.

Passing to Steinberg map and algebraic groups: for the symplectic group, the coefficients of the characteristic polynomial can still be used without affecting any of the measure conversions, since this group is simply connected,  and an argument similar to the above argument on the Lie algebra applies. For special orthogonal groups, $\A_G$ is not an affine space since $\bG$ is not simply connected;  the coefficients of the characteristic polynomial give a map to an affine space. It seems to be a worthwhile exercise to work out the relationship between these two spaces and their measures; I have not done this calculation. 

Finally, we discuss the cases $\bG=\GL_n$ and  $\bG=\GSp_{2n}$ in some more detail since the latter calculation is needed in \cite{achter-altug-garcia-gordon}. 
For $\GL_n$, we just map $g$ to the coefficients of its characteristic polynomial. 
For $\GSp_{2n}$, we can define $\fc^\charp(g)=(a_1, \dots, a_n, \nu(g))$, where $a_1, \dots a_n$ are the first $n$ non-trivial coefficients of the characteristic polynomial, and $\nu(g)$ is the multiplier (this is ad hoc; one could have used the determinant instead to be consistent with $\GL_n$); the superscript `$\charp$' is to remind us that we are using the characteristic polynomial and distinguish this map from the standard map $\fc$.  
The codomain  of the map $\fc^\charp$ is the space we call $\A_{G}^\charp$ which is $\A^{n-1}\times\G_m$ if $\bG=\GL_n$, and $\A^n\times \G_m$ if 
$\G=\GSp_{2n}$. The restriction of $\fc^\charp$ to 
$\bG^\der$ (if we forget the $\G_m$-component) coincides with the map $\fc'$ constructed above for $\bG^\der$ (which coincides with $\fc$ if $\bG^\der =\SL_n$).

The measure on the base  $\A_G^\charp$ in this  case should be defined as 
 the product of the measures associated with the form 
$da$ on the affine space, and $ds/s$ on $\G_m$, where we  denote the coordinates on $\A^{k}\times \G_m$ by $(a,s)$ (with $k=n-1$ or $k=n$). 
With this definition, the resulting measure is, essentially, independent of the specific map used for the last coordinate 
(for example, in the case of $\GSp_{2n}$, if the  determinant instead of the multiplier were mapped to $s$, the measure would just change by the factor $|n|_F$, which is $1$ unless the residue characteristic of $F$ divides $n$; but this caveat is the reason  we prefer to work with the multiplier).

Let $\omega_a^\charp$ be the form on the fibre $(\fc^\charp)^{-1}(a)$ defined the same way as the form $\omega_a^{\geom}$ in (\ref{eq:geom_lie1}) and \S\ref{sub:group}, but using the map $\fc^\charp$ instead: 
\begin{equation}\label{eq:charp}
\omega_a^\charp \wedge (da\wedge \frac{ds}s) = \omega_G.
\end{equation}
As before, the forms $\omega_G$,  $\omega_T$ and $\omega_{G/T}$ are, of course, invariant under the action of the centre of $G$,
 but the centre does not even act on $\A_G^\charp$ as a group. Nevertheless,  multiplication by scalars still makes sense on this space.
 
 To find the relation between the differential forms $\omega_a^\charp$ and $\omega_{T\backslash G}$ on a given orbit, let us work over the algebraic closure of $F$ for a moment.  
Over $\bar F$, every element $\gamma\in \bG(\bar F)$ can be written as $\eta_z\gamma'$ with $\eta_z\in \bZ(\bar F)$ and $\gamma'\in \bG^\der(\bar F)$ (defined up to an element of ${\mathbf A}(\bar F)$; we just pick one such pair). 
For $\eta_z:=z{Id}\in Z$ with $z\in \bar F$,  
and $g\in \bG^\der(\bar F)$, each coefficient $a_i(\eta_z g)$ of the characteristic polynomial of $\eta_zg$ differs from $a_i(g)$ by a power of $z$. Then the form  $\omega_a^\charp$  would have  to scale by the power of $z$ as well, to preserve (\ref{eq:charp}). 
We denote by $O_\gamma^\charp$ the orbital integral of $\gamma$ with respect to the form $\omega_{\fc^\charp(\gamma)}^\charp$, as a distribution on $C_c^\infty(G)$. We compute explicitly the relation between this integral and the integral with respect to 
$\omega_{T\backslash G}$ for $\GSp_{2n}$. 

\begin{example} 
$\bG=\GSp_{2n}$. 
In this case the scaling factor is $|z|^{S}$, where $S$ is the sum of the degrees (as homogeneous polynomials in the roots) of the first $n$ coefficients of the 
characteristic polynomial, i.e., $1+2+\dots + n =\frac{n(n+1)}2$.  If $\gamma=\eta_z\gamma'$ with $\det(\gamma')=1$ 
(and $\eta_z\in \bZ(\bar F)$), then 
$|z|=|\det(\gamma)|^{1/2n}$.
We obtain, for $\gamma\in \GSp_{2n}(F)^\rss$:
$$
O^\charp_{\gamma}(f)= 
|D(\gamma)|^{1/2} |\det(\gamma)|^{-\frac{(n+1)}{4}}\int_{G/T}f(g^{-1}\gamma g) \omega_{G/T}.$$
\end{example}

\subsection{Summary} To summarize, so far the following choices have been made 
(we use the same notation as in \cite{langlands-frenkel-ngo} whenever possible): 
\begin{itemize} 
\item The measure $|dx|$  on $F$, such that the volume of $\ri_F$ is $1$. 
If we are working over a global field $K$, 
and $F=K_v$ is its completion at a finite place $v$, 
this measure differs from \cite{langlands-frenkel-ngo} for a finite number of places $v$.
For orbital integrals, this discrepancy gives rise to the factor $|\Delta_{K/\Q}|_v^{\dim(G)-\rank(G)}$ (independent of the element) at each place $v$. 
\item An invariant differential form $\omega_G$ on $\bG$ -- appears on the both sides and does not affect the ratio between measures. 
\item For an algebraic torus $\bT$, a choice of $\{\chi_1, \dots, \chi_n\}$ - a $\Z$-basis of $X^\ast(\bT)$. 
This choice does not affect anything.
\end{itemize}

Given $\gamma\in G$ - a regular semisimple element, with $T=C_G(\gamma)$, 
the following differential forms and measures have been constructed from these choices: 
\begin{itemize}
\item $\omega_T=d\chi_1\wedge\dots \wedge d\chi_r$; 
\item $\omega_{T\backslash G}$  (which should be thought of as a measure on the orbit of $\gamma$, with $T=C_G(\gamma)$) satisfying $\omega_G=\omega_T\wedge \omega_{T\backslash G}$. (Note that the centralizers of stably conjugate elements are isomorphic as algebraic tori over $F$, so one can also think of $\omega_{T\backslash G}$ as a form on the \emph{stable} orbit of $\gamma$.) 
\item $\omega_a$, also on the stable orbit of $\gamma$, with $a=\fc(\gamma)$, satisfying $\omega_{\A_G}\wedge \omega_a=\omega_G$. 
\item In \cite{langlands-frenkel-ngo}, there is a renormalized measure 
$d\bar g_v:= \frac{L(1, \sigma_G)}{L(1, \sigma_T)} \omega_{T\backslash G}$. 
\end{itemize}

Recall the notation $D(\gamma)$ for the Weyl discriminant of $\gamma\in G^\rss$, $|\Delta(\gamma)|=|D(\gamma)|^{1/2}$, as well as the definition of Artin $L$-factor, (\ref{eq:L}).  
The following relations between these measures have been established: 
\begin{itemize}
\item For $\gamma\in G^\rss$, and $a=\fc(\gamma)$, 
$\omega_a=|\Delta(\gamma)|\omega_{T\backslash G}$. 
\item Consequently, for the measure $d\bar g_v$ defined in \cite[\S3.4 below (3.17)]{langlands-frenkel-ngo} we have: 
$\omega_a = |\Delta(\gamma)| {L(1, \sigma_{G/T})}d\bar g_v$, %
where the representation $\sigma_{T\backslash G}$ of the Galois group is, by definition, the quotient of the representation 
$\sigma_T$ on $X^\ast(\bT)$ by the subrepresentation $\sigma_G$ on the characters of $G$ \footnote{
the rank of this subrepresentation is the same as the rank of the centre of $G$; if  $G$ is a semisimple group, we have 
$\sigma_{T\backslash G} =\sigma_T$.}, and hence 
$\left(\frac{L(1, \sigma_G)}{L(1, \sigma_T)}\right)^{-1}$ is precisely $L(1,  \sigma_{T\backslash G})$.

\end{itemize} 
This establishes the identity (3.31) in \cite{langlands-frenkel-ngo} (we note that $\mathrm{Orb}(f)$ is defined in \emph{loc.cit.} as the integral with respect to the measure $d\bar g_v$ on the orbit). Now we move on to the discussion of the factor $L(1, \sigma_{T\backslash G})$ and the relationship with the \emph{canonical measures} in the sense of Gross.

\section{Orbital integrals: from differential forms to `canonical measures'}\label{sec:canon}

So far, we have worked with measures coming from differential forms, as summarized above.
However, in many parts of the literature the so-called \emph{canonical} measures are used. They are defined by means of defining a \emph{canonical} subgroup, and then normalizing the measure so that the volume of this subgroup is $1$. 
This introduces the following factors: 
\begin{itemize} 
\item By definition of the canonical measure, for a torus $\bT$,
$$\mu_T^{\can}=\frac1{\vol_{\omega^{\can}}(\cT^0)}|\omega^{\can}|, $$
where $\omega^{\can}$ is the so-called \emph{canonical} invariant volume form (discussed briefly in \S\ref{subsub:Neron} above; the details of the definition are not important here).

By Theorem \ref{thm:torus} above, 
$$\vol_{\omega^{\can}}(\cT^0)= L(1, \sigma_T)^{-1}.
$$ 
   
Hence, on $\bT$,  we have 
\begin{equation}\label{eq:can}
\omega^{\can}=\frac{\vol_{\omega^{\can}}(\cT^0)}{\vol_{\omega_T}(\cT^0)}\omega_T.
\end{equation} 

\item Recall that $\omega_G=\omega_G^{\can}$ since we are assuming $\bG$ is split; this is also true more generally for $\bG$ unramified,
(and in any case, the choice of the form on $\bG$ matters much less than the choice of the form on $\bT$, as discussed above). 
Therefore, on the orbit of $\gamma$, we have: 
$$
\omega^{\can}_{T\backslash G} = 
\frac{\vol_{\omega_T}(\cT^0)}{\vol_{\omega^{\can}}(\cT^0)}\omega_{T\backslash G},$$
\text{  and   }
\begin{equation}\label{eq:quot}
\begin{aligned}
&\mu^{\can}_{T\backslash G} = \frac{\mu_G^\can}{\mu_T^\can} =
 \frac{{\frac1{\vol_{\omega_G}(G^0)}}{|\omega_G|}}{{\frac1{\vol_{\omega^\can}(\cT^0)}}{|\omega^\can|}}
=\frac{\vol_{\omega^\can}(\cT^0)}{\vol_{\omega_G}(G^0)} \frac{|\omega_G|}{|\omega^\can|}\\
&=\frac{\vol_{\omega^\can}(\cT^0)}{\vol_{\omega_G}(G^0)}
\frac{|\omega_G|}{\frac{\vol_{\omega^\can}(\cT^0)}{\vol_{\omega_T}(\cT^0)}|\omega_T|}
  =\frac{\vol_{\omega_T}(\cT^0)}{\vol_{\omega_G}(G^0)}\frac{|\omega_G|}{|\omega_T|} =  
\frac{\vol_{\omega_T}(\cT^0)}{\vol_{\omega_G}(G^0)}|\omega_{T\backslash G}|.
\end{aligned}
\end{equation}

When $\bT$ splits over an unramified extension, by Theorem \ref{thm:torus} above, $\vol_{\omega_T}(\cT^0)=L(1, \sigma_T)^{-1}$. 
Thus at almost all places $v$, the measure
$d\bar g_v$ is related to the canonical measure on the orbit by: 
\begin{equation}\label{eq:dgbar}
d\bar g_v= L(1, \sigma_G)\vol_{\omega_G}(G^0)\mu^{\can}_{T\backslash G}.
\end{equation}

Combining (\ref{eq:quot})  with the relation $\omega_a=\pm \Delta(\gamma) \omega_{T\backslash G}$, we also obtain the relation between the geometric measure and canonical measure (for all places): 
\begin{equation}\label{eq:canon_to_geom}
|\omega_a| = |\Delta(\gamma)|  \frac{1}{\vol_{\omega_T}(\cT^0)} \mu^{\can}_{T\backslash G},
\end{equation}
where the factor $\vol_{\omega_G}(G^0)$ disappears since the same measure on $\bG$ needs to be used on both sides when defining $\omega_a$. 
We recall that the factor  ${\vol_{\omega_T}(\cT^0)}$ is discussed above in \S \ref{subsub:general}.

\end{itemize}

\subsection{Example: $\GL_2$}\label{sub:GL2}
For $\bG=\GL_2$, we can make everything completely explicit. 
The orbital integrals  of spherical functions with respect to \emph{canonical} measure are computed, for example, 
in \cite[\S 5]{kottwitz:clay}.  We combine this computation with our conversion factors to obtain the integrals with respect to the geometric measure. We observe that the result agrees with the formula (3.6) of  \cite{Langlands:2013aa}. 

In \cite[\S 5]{kottwitz:clay} the orbital integrals are computed using the reduced building (i.e. the tree) for $\GL_2$, and expressed in terms of the integer $d_\gamma$ (for $\gamma\in G(F)^\rss$). 
The number $d_\gamma$ is defined in terms of the valuations of the eigenvalues of $\gamma$, 
see the top of p.415 for the split case, p.417 for the unramified case,  and (5.9.9) for the ramified case.  

In fact, we have 
$$
|D(\gamma)|=\begin{cases}q^{-2d_\gamma}, \quad \gamma \text{ split or unramified},\\
q^{-2d_\gamma-1}, \quad \gamma \text{ ramified}.
\end{cases} 
$$
(This is the definition in the split and unramified cases and an easy exercise in the ramified case.)

Here we only look at the simplest orbital integral of $f_0$,  the characteristic function of the maximal compact subgroup $G_0=\GL_2(\ri_F)$.  
\begin{itemize}
\item If $\gamma$ is split over $F$, then, from formula (5.8.4) \emph{loc.cit.}: 
\begin{equation}\label{eq:split}
O_{\gamma}(f_0)= q^{d_\gamma} = |D(\gamma)|^{-1/2}
\end{equation}

\item If $\gamma$ is elliptic (which in $\GL_2^\rss$, is the same as not split), then 
$O_\gamma(f_0)=|V^\gamma|$, the cardinality of the set of fixed points of the action of $\gamma$ on the building; see formula  (5.9.3).
Note that here the right-hand side does not depend on the choice of the measures on $G$ and  on the centralizer of $\gamma$ (which  we denote by $T=G_\gamma=C_G(\gamma)$ to consolidate notation with this part of \cite{kottwitz:clay}). Thus, there is a unique choice of measures for which this equality is true. This equality is explained in \S 3.4 of \emph{loc.cit.}; see also the explanation just above (5.9.1). 
In fact, for elliptic $\gamma$, one has 
$$\vol(Z\backslash G_\gamma)O_\gamma(f_0)=|V^\gamma|,$$
where on the left the volume and the orbital integral are taken with respect to the same choice of the measure on $G_\gamma$, and the measure on $G$ that gives $G_0$ volume $1$ (note that in this formula both sides are independent of the choice of the measure on $G_\gamma$). 
Thus the measure on $G_\gamma$ that makes (5.9.3) work is precisely the measure such that $\vol(Z\backslash G_\gamma)=1$. 

Plugging in the calculations of $|V^\gamma|$ from \emph{loc.cit.}, in the two remaining cases we obtain: 

\item If $\gamma$ is unramified (5.9.7):
\begin{equation}\label{eq:unram}
O_{\gamma}(f_0)= 1+(q+1)\frac{q^{d_\gamma}-1}{q-1}.
\end{equation}

\item If $\gamma$ is ramified (5.9.10):
\begin{equation}\label{eq:ram}
O_{\gamma}(f_0)= 2\frac{q^{d_\gamma+1}-1}{q-1}.
\end{equation}
\end{itemize}

Assume as usual that $p\neq 2$. 
Suppose we started with the measure on $G_\gamma$ that gave volume $1$ to its maximal compact subgroup, and the measure on $Z$ such that the volume of $Z\cap G_0$ is $1$.  
In the unramified case, the map from $T^c$ to $Z\backslash T$ is surjective, and this choice of measures gives the quotient $Z\backslash T$ volume $1$. 
In the ramified case, the image of  $T^c$ in  $Z\backslash T$ has index $2$, and thus the volume of $Z\backslash T$ we get from this natural measure on $T$ is not $1$ but $2$. 
The upshot is that in the ramified case, the measure giving the volume $1$ to $Z\backslash G_\gamma$ does \emph{not} come from a natural measure on $G_\gamma$.

Thus, combining the relation (\ref{eq:canon_to_geom}) with (\ref{eq:split}), (\ref{eq:unram}), and (\ref{eq:ram}), respectively, and using (\ref{eq:res_quadr}), which applies since for all tori in $\GL_2$, $T^c=\cT^0$,  we obtain the integrals with respect to the geometric measure:
\begin{equation}\label{eq:oigl2}
\begin{aligned}
& O_\gamma^\geom(f_0)  =  \frac{|D(\gamma)|^{1/2}}{\vol_{\omega_T}(T^c)}  O_{\gamma}(f_0)\\
& =\begin{cases} &(1-\frac{1}{q})^{-2} \quad \gamma \text{ is hyperbolic} \\
&\frac{q^2}{(q^2-1)}\left(1+(q+1)\frac{q^{d_\gamma}-1}{q-1}\right) q^{-d_\gamma} \quad \gamma  \text{ is unramified elliptic}\\
&\frac{q\sqrt{q}}{q-1} \left(\frac{q^{d_\gamma+1}-1}{q-1}\right) q^{-d_\gamma-1/2} \quad \gamma \text{ is ramified elliptic}
\end{cases}\\
& = \left(1-\frac1q\right)^{-2} \begin{cases} & 1 \quad \gamma \text{ is hyperbolic} \\
& 1 - \frac{2}{q+1}q^{-d_\gamma} \quad \gamma  \text{ is unramified elliptic}\\
&1 -  q^{-(d_\gamma+1)} \quad \gamma \text{ is ramified elliptic}
\end{cases}
\end{aligned}
\end{equation}

These formulas agree with \cite[(2.2.10)]{Langlands:2013aa}, if one multiplies our results by 
$\vol_{\omega_G}(G_0)= (1-\frac1q)^2(1+\frac 1q)$, as required by the relation (\ref{eq:dgbar}).

\subsection{The next step} In \cite{Langlands:2013aa}, Langlands works out Poisson summation on the geometric side of the stable Trace Formula for $\SL_2$. 
Roughly speaking, Poisson summation formula is a relation between the sum of the values of a smooth function over a lattice in a vector space, and the sum of the values of its Fourier transform over a dual lattice. 
Here the space is the set of ad\`elic points of the Steinberg-Hitchin base for $\SL_2$, which is just the affine line. 
The lattice in it is the image of the diagonal embedding of the base field (we can take $\Q$ for simplicity). 
The geometric side of the Trace Formula contains a sum over the conjugacy classes of elliptic elements $\gamma\in \SL_2(\Q)$, which corresponds to a sum over $\Q$ in the Steinberg-Hitchin base. Thus at least for the elliptic  part, it is tempting to take the function to be a stable orbital integral (i.e., the integral of some fixed test function over a fibre of Steinberg map $\fc^{-1}(a)$, as a function of $a$), and apply Poisson summation. 
However, for that the function needs to satisfy some  smoothness assumption. 
Now we can at least make some preliminary remarks about how far our function is from being smooth, at least at every finite place. 
 
If we take $\gamma\in \SL_2(\Q_p)$, the $\GL_2(\Q_p)$-orbital integral computed above is the \emph{stable} orbital integral of $\gamma$.  
All along we have been assuming that $\gamma$ is a regular semisimple element. It is  well-known that the singularities of orbital integrals occur only at non-regular elements (and we will see this explicitly in a moment, in this example). 
More precisely, it is a result of Harish-Chandra that for a given test function $f$, when a measure of the form 
$|\omega_{G/T}|$ is used on each regular semisimple orbit, the orbital integral $\gamma\mapsto O_\gamma(f)$ is
a smooth (i.e., locally constant) function on the open set $G^\rss$ of regular semisimple elements. 
This function is not bounded as $\gamma$ approaches a non-regular element; however, its growth is controlled by $|D(\gamma)|^{-1/2}$.   
Specifically, Harish-Chandra proved that (still with $f$ fixed), the so-called \emph{normalized orbital integral}, namely, the function $
\gamma\mapsto |D(\gamma)|^{1/2}O_\gamma(f)$ is bounded on compact subsets of $G$, and 
locally integrable on $G$. We note that since $D(\gamma)$ vanishes at non-regular elements, this normalized orbital integral is also zero off the regular set. Thus, the normalized orbital integral, as a function of $\gamma$ (for a given test function $f$), is locally constant on $G^\rss$, zero on non-regular semisimple elements, and locally bounded on $G$. However, this does not imply that 
it is continuous on $G$. Indeed, while it is locally constant on the set of regular semisimple elements, as $\gamma$ approaches a non-regular element, the neighbourhoods of constancy get smaller; at a non-regular element $\gamma_0$ itself this function is zero since $D(\gamma_0)=0$; by Harish-Chandra's theorem this function is bounded on any compact neighbourhood of $\gamma_0$; however, nothing says that it is continuous at $\gamma_0$: without a careful choice of measures, it will have ``jumps''. 
As we shall see in our $\SL_2$-example, the extension of the normalized orbital integral by zero to non-regular elements does not actually give a continuous function on $G$; however, when the geometric measure is chosen, one gets a function on the Steinberg-Hitchin base with just a removable discontinuity.

In $\SL_2$, we have just two non-regular semisimple 
elements, namely, $\pm {\mathrm{Id}}$. Their images under Steinberg map $\fc_{SL_2}:\SL_2 \to \A^1$ are $\pm 2$. 
Fix $p$ (for now, $p\neq 2$) and consider, for example, a neighbourhood of the point $2\in \A_{\SL_2}(\Q_p)=\A^1(\Q_p)$. 
It consists of the images of split, ramified, and unramified elements with sufficiently large $d_\gamma$ (the split/ramified/unramifed is determined by the discriminant of the characteristic polynomial of a given element, as discussed above in Example \ref{ex:gl2}).    
The formula (\ref{eq:oigl2}) shows that as $d_\gamma\to \infty$ (i.e, as $\gamma$ approaches $\pm \mathrm{Id}$), the stable orbital integral of $f_0$ on the orbit of  
$\gamma$ with respect to the geometric measure gives a   \emph{continuous} function on $\A^1$, with value $(1-q^{-1})^{-2}$ at $a=2\in \A_{\SL_2}(\Q_p)$. (This, of course, cannot be said about the orbital integrals with respect to the canonical measure, as they get large - of the size $q^{d_\gamma}$; as $|D(\gamma)|^{1/2} \asymp q^{-d_\gamma}$, we see the confirmation of Harish-Chandra's boundedness result; but still the function $\gamma\mapsto 
|D(\gamma)|^{1/2}O^\can(\gamma)$ has ``jumps'' at $\pm Id$; once we make the adjustments by the volumes of the maximal compact subgroups of the corresponding tori, it becomes continuous). This continuity result is one of the insights of \cite{Langlands:2013aa}. 
However, as we see explicitly from (\ref{eq:oigl2}), this function is continuous but not smooth (i.e. not constant on any neighbourhood of $a=\pm 2$); and so far this is just the description of the situation at a single place, whereas ultimately we will need a global Poisson summation formula. This causes some of the technical difficulties discussed in Altug's lectures.

\section{Global volumes}\label{sec:global}

\subsection{The analytic class number formula for an imaginary quadratic field}\label{sub:an_cnf}
Here we recast the analytic class number formula for an imaginary quadratic field $K$ as a volume computation, using the above ideas. It was observed by Ono, \cite{ono_tamagawa_tori} (see also \cite{shyr77}), that the analytic class number formula in this case amounts to the fact that the Tamagawa number $\tau(\bT)$ of the torus $\bT=\res_{K/\Q}\G_m$ equals $1$. 
We will assume that $\tau(\bT)=1$ and derive the analytic class number formula for $K$ from this fact. This also serves as preparation for \S\ref{sub:orb_glob} where the same volume term combines with an orbital integral for an interesting result.

The analytic class number formula for a general number field is the relation 
\begin{equation}\label{eq:cnf}
\lim_{s\to 1^+}(s-1)\zeta_K(s) = \frac{2^{r+t}\pi^t R_K h_K}{w_K |\Delta_K|^{1/2}}, 
\end{equation}
where: $\zeta_K$ is the Dedekind zeta-function of $K$, $R_K$ is the regulator (we will not need it in this note so we skip the definition), $\Delta_K$ is the discriminant of $K$, $h_K$ is the class number, $w_K$ is the number of roots of $1$ in $K$, $r$ is 
the number of real embeddings, and $2t$ is the number of complex embeddings of $K$. 

Let us consider an imaginary quadratic field $K=\Q(\sqrt{-D})$  (with $D>0$); denote its ring of integers by $\ri$.  
In this case, we have $r=0$, $t=1$, the regulator $R_K$ is automatically equal to $1$, and 
the left-hand side equals the value at $s=1$ of $\frac{\zeta_K(s)}{\zeta(s)}=L(1, \chi_K)$ -- the value (in the sense of a conditionally convergent product) of the Dirichlet $L$-function $L(1, \chi_K)$ \footnote{This equality is the simplest case of the correspondence between Artin and Hecke $L$-functions.}. Here $\zeta(s)$ is the Riemann zeta-function, and  $\chi_K$ is the Dirichlet character associated with $K$
: 
\begin{equation}\label{eq:dirichlet_char}
\chi_K(p)=\begin{cases}  1 &\quad  p \text{ splits in } K\\
-1 &\quad  p \text{ is inert in } K\\
 0 &\quad  p \text{ ramifies in } K
\end{cases}. 
\end{equation}

Thus, for an imaginary quadratic field $K$ the analytic class number formula reduces to: 
\begin{equation}\label{eq:quadr_class}
L(1, \chi_K)=\frac{2\pi h_K}{w_K\sqrt{\Delta_K}} .
\end{equation}
Our goal is to prove this relation by using only the known facts about algebraic groups and the measure conversions discussed above. 
The algebraic group in question here is just the torus $\bT=\res_{K/\Q}\G_m$. 

Let $\A_K$ be the ad\`eles\footnote{There is an unfortunate clash of standard notation: we used $\A_G$ to denote Steinberg quotient of $\bG$; hopefully this causes no confusion. Another notation clash is $\mu_K$ for the group of roots of $1$ in $K$, as we have been using the letter $\mu$ (with subscripts and superscripts) to denote various measures.} of $K$ and let $\A_K^\fin$ be the finite ad\`eles. 
In general $K^\times$ embeds (diagonally) in $\A_K^\times$ with discrete image; for $K$ imaginary quadratic, the image of the embedding
$K^\times \hookrightarrow (\A_K^\fin)^\times$ is still discrete (in fact, this is true only when $K=\Q$ or is imaginary quadratic,  see e.g.,\cite{milne:ClassFieldTheory}).   
We know (weak approximation, see e.g.,  \cite{platonov-rapinchuk:AlgGroupsAndNT}) that for $\Q$, 
$$\Q^\times \big\backslash (\A_{\Q}^\fin)^\times \big/ \prod_p \Z_p^\times  =\{1\}.$$
Since the image of $K^\times$ in $(\A_K^\fin)^\times$ is discrete, we can define a similar double quotient for $\A_K^\fin$: 
$K^\times \big\backslash (\A_{K}^\fin)^\times \big/ \prod_{v\nmid\infty} (\ri_v)^\times$, which, roughly speaking, should measure the size of the class group of $K$. The reason this is not exactly the class group is the intersection of the image of $K^\times$ with the  compact subgroup $\prod_{v\nmid\infty} (\ri_v)^\times$. 
More precisely, we have the exact sequence: 
\begin{equation}\label{eq:exseq_ideles}
\xymatrix{
1 \ar[r]  & (K^\times \cap  \prod_{v\nmid\infty} (\ri_v)^\times )\backslash \prod_{v\nmid\infty} (\ri_v)^\times \ar[r] & K^\times \big\backslash (\A_{K}^\fin)^\times 
 \ar[r]&  \mathrm{Cl}(K) \ar[r] & 1, 
}
\end{equation}
where $\mathrm{Cl}(K)$ is the class group of $K$. 

The group $K^\times \cap  \prod_{v\nmid\infty} (\ri_v)^\times$ is precisely the group $\mu_K$ of roots of $1$ in $K$ (the elements of $K^\times$ that are units at every finite place). 

The key point is that \emph{if we normalize the volume of the group of units $\ri_v^\times$ to be $1$ at every place, and call this measure $\nu_T$}
\footnote{An important coincidence that happens for our torus $T$, because it is obtained from $\G_m$ by restriction of scalars, is that the measure $\nu_T$ coincides with the canonical measure $\mu_T^\can$ at every finite place, as discussed above in 
the point (1) of \S \ref{subsub:general}. See \cite{shyr77} for the general situation.},  then we get from the above exact sequence: 

\begin{equation}\label{eq:Tama1}
\vol_{\nu_T}\left(K^\times \big\backslash (\A_{K}^\fin)^\times \big/ \prod_{v\nmid\infty} (\ri_v)^\times\right) = \frac{h_K}{w_K}.
\end{equation}

We will assume as fact that the Tamagawa number of $\bT$ is $1$ (this is so because  $\bT$ is obtained from $\G_m$ by Weil restriction of scalars, as briefly discussed below). The analytic class number formula will follow as soon as we relate the volume on the left-hand side of (\ref{eq:Tama1}) to the Tamagawa number $\tau(\bT)$.

\subsubsection{Tamagawa measure}\label{subsub:Tama}
We briefly recall the definition of the Tamagawa measure, just for the special case of our torus  $\bT$. We follow the definition of Ono, \cite{ono:boulder}, \cite{ono:arithmetic_tori}, which has become standard. 
\footnote{We note that superficially, it differs from the definition that A. Weil uses in \cite{weil:adeles}, in the sense that Ono uses a specific set of convergence factors, and incorporates a global factor that makes his definition independent of the choice of the convergence factors at finitely many places. The resulting global measure, of course, is the same in the both sources.}   

{\bf 1.} Let $(\A_\Q^\times)^1$ be the set of norm-1 ad\`eles (also referred to as \emph{special ideles}): 
$$(\A_\Q^\times)^1:=\{(x_v)\in \A_\Q: \prod_{v} |(x_v)|_v =1\}, $$
where the product is over all places of $\Q$. 

We  have the exact sequence  
\begin{equation}\label{eq:infty1}
1\to (\A_\Q^\times)^1 \to \A_\Q^\times \to \R_{>0}^\times  \to 1,
\end{equation} 
where the first map is the inclusion and the second map is the product of absolute values over all places, 
$x=(x_v)\mapsto \prod_v |x_v|_v$.
Moreover, the exact sequence splits and we have a canonical decomposition 
\begin{equation}\label{eq:adeles_product}
\A_\Q^\times \simeq (\A_\Q^\times)^1\times (\R^\times)^0,
\end{equation} 
as a direct product of topological groups, where $(\R^\times)^0$ stands for the connected component $R_{>0}^\times$ 
(in the sense of the metric topology) of the group $\R^\times$. 
We note that the image of the diagonal embedding of $\Q^\times$ into $\A_\Q^\times$ is contained in $(\A_\Q^\times)^1$, and 
it follows from (\ref{eq:adeles_product}) that the quotient $\Q^\times \backslash (\A_\Q^\times)^1$ is compact. 

{\bf 2.} To define the Tamagawa measure on $\bT(\Q)\backslash\bT(\A)$, one needs to start with a volume form $\omega$ on $\bT$ defined over $\Q$. 
We note that even writing down such a form concretely is not trivial: the natural form $\omega_T$ defined in \S \ref{sub:tori} is not defined over $\Q$. 
Fortunately, in our special case, the differential form $\omega:=\frac{1}{\sqrt{\Delta_K}}\omega_T$ is defined over $\Q$, see Example \ref{ex:restr} (This easy case can also be verified directly by a calculation similar to that of Example \ref{ex:quadr}).  \footnote{See \cite{gan-gross:haar}, Corollary 3.7, for a way to define such a form in general. In our special case of the quadratic field, it  is an  easy case of the discriminant-conductor formula that the Artin conductor of the motive constructed in \cite{gan-gross:haar} coincides with the discriminant $\Delta_K$, so our definition is a special case of the construction in \cite{gan-gross:haar}.}

Recall the local Artin $L$-factors attached to the representation $\sigma_T$ of $\Gal(K/\Q)$  on $X^\ast(\bT)$, see (\ref{eq:L}), and let 
$L(s, \sigma_T):=\prod_p L_p(s, \sigma_T)$.
Let $r_T$ be the multiplicity of the trivial representation as a sub-representation of $\sigma_T$. In our case, $\sigma_T$ is $2$-dimensional, and $r_T=1$; a copy of the trivial representation in $\sigma_T$ is generated by the norm character, which is stable under the action of the Galois group $\Z/2\Z$.  
Let $$\rho_T:= \lim_{s\to 1+} (s-1)^{r_T}L(s, \sigma_T).$$ 
We see  that in our case, $\rho_T$ coincides with the left-hand-side of (\ref{eq:cnf}).  
The measure $\mu^{Tama}$ on $\bT(\A)$ is defined as: 
\begin{equation}\label{eq:def_tama}
\mu^{Tama}:=\rho_T^{-1}|\omega_\infty|\prod_p L_p(1, \sigma_T) |\omega|_p,
\end{equation}
where $\omega_\infty$  is the form induced by $\omega$ on $\bT(\R)$, in our case. 

We make a note of some subtle features of this definition:
\begin{enumerate} 
\item The definition does not depend on the choice of a volume form (as long as $\omega$ is defined over $\Q$), since any two choices differ by a constant in $\Q$, which does not matter globally thanks to the product formula. 
\item Without the \emph{convergence factors} $L_p(1, \sigma_T)$, the product $\prod_p|\omega|_p$ does not define a measure on $\bT(\A)$, since (as one can easily see in our example) the maximal compact subgroup $\prod_{v\nmid \infty} (\ri_v)^\times$ of $\bT(\A^\fin)$ would have infinite volume with respect to such a `measure', since by (\ref{eq:res_quadr}), it contains the Euler product for the Riemann zeta function at $1$.
There is some choice involved in the definition of the convergence factors (for example, in Weil's definition in \cite{weil:adeles} the convergence factors in this case would be simply $(1-1/p)$, which would be sufficient to achieve convergence of the product measure). 
As Ono explains in \S 3.5 of \cite{ono:arithmetic_tori}, if one modifies the individual convergence factors by any multipliers whose product converges, it does not affect the final result thanks to the global factor $\rho_T^{-1}$. 
\end{enumerate} 
For future use, we define by $\mu_p$ the measure $L_p(1, \sigma_T) |\omega|_p$ on $\bT(\Q_p)$.

{\bf 3.} The Tamagawa number of $\bT$ is, by definition, the volume  (with respect to the Tamagawa measure on the quotient, discussed below), 
of $\bT(\Q)\backslash \bT(\A_\Q)^1$, where 
$$\bT(\A_\Q)^1=\{(x_v)\in \bT(\A_\Q): \prod_v|\chi(x_v)|_v = 1 \text{ for all }\chi\in X^\ast(\bT) \text { that are defined over }\Q \}.$$
The group of $\Q$-characters of $\bT$ has rank 1, and is generated by the norm map.
Thus,
$$\bT(\A_\Q)^1=\{(x_v)\in \bT(\A_\Q): \prod_{v}|N_{K_v/\Q_v}(x_v)|_v = 1\},$$
where the product is over the places of $\Q$. 
We note that $\bT(\A_\Q)=\A_K^\times$, and we have the exact sequence 
\begin{equation}\label{eq:T1}
1\to \bT(\A_\Q)^1 \to \A_K^\times \to \R_{>0}^\times \to 1,
\end{equation}
where the map to $\R_{>0}^\times$  is the product of the local norm maps. 

Let $dm$ (using the notation and terminology of \cite{shyr77}) be the measure on 
$\bT(\A_\Q)^1$ that `matches topologically' in this exact sequence with the measure $\mu^{Tama}$ on $\bT(\A_\Q)=\A_K^\times$ defined above, and the measure $\frac{dt}{t}$ on $\R_{>0}^\times$.
That is, $dm$ is the measure on $\bT(\A_\Q)^1$ such that
$$\mu^{Tama}= dm\wedge \frac{dt}{t}.$$ 
Since $\bT(\Q)=K^\times$ is a discrete subgroup of $\bT(\A_\Q)^1$, the measure $dm$ descends to the quotient by this subgroup, and the volume of $\bT(\Q)\backslash \bT(\A_\Q)^1$ with  respect to this measure is, by definition, the Tamagawa  number $\tau(\bT)$.

We note that the exact sequence (\ref{eq:T1}) splits, and we have a group isomorphism 
$$ K^\times \backslash \A_K^\times \simeq \bT(\Q)\backslash \bT(\A_\Q)^1 \times \R_{>0}^\times.$$ 

\subsubsection{The proof of the analytic class number formula}
Now we are ready to go from (\ref{eq:Tama1}) to the analytic class number formula for $K$; we do it in two steps. 

{\bf Step1. The finite places.} 
Rewriting the relation (\ref{eq:res_quadr})  of Example \ref{ex:quadr} using the notation of this section (and noting that 
for $p\neq 2$, $|\Delta_K|_p=p^{-1}$ if  $p$ ramifies in $K$ and $|\Delta_K|_p=1$ otherwise), we obtain, for $v\vert p$: 
\begin{equation}\label{eq:classnum1}
\vol_{|\omega_T|}(\ri_v^\times)=\left(1-\frac1p\right)L_v(1, \chi_K)^{-1}|\Delta_K|_p^{1/2}.
\end{equation}
We will see below in \S\ref{sub:two} that this relation holds at $p=2$ as well. 
Thus, $\vol_{\mu_p}(\ri_v^\times) =1$, and we see explicitly in this example, that our measure $\mu_p$ coincides with the 
$p$-component of the measure $\nu_T$ used in (\ref{eq:Tama1}) (this is a very special case of \cite{gan-gross:haar}, Corollary 7.3).

{\bf Step 2. The infinite places and putting it together.}
We also have the exact sequence  (where $S^1\subset \C^\times$ is the unit circle)
\begin{equation}\label{eq:infty}
1\to S^1 \to \bT(\A_\Q)^1\to \A_K^\fin  \to 1,
\end{equation} 
where the first map is the inclusion into $\A_K$ that maps $z\in S^1$ to the ad\`ele $(1,z)$ trivial at all the finite places, and the second map is the projection onto the finite ad\`eles. 
Since the image of the diagonal embedding of $K^\times$ intersects the image of $S^1$ in $\bT(\A_\Q)^1$ trivially, 
(\ref{eq:infty}) yields the exact sequence for the quotients by $K^\times$: 
\begin{equation}\label{eq:infty1}
1\to S^1 \to K^\times\backslash \bT(\A_\Q)^1\to K^\times\backslash \A_K^\fin  \to 1.
\end{equation}  

Now we need to carefully find the component at infinity $dm_\infty$  (which is a measure on $S^1$) of the measure $dm$ defined via the exact sequence  (\ref{eq:T1}). First, we choose a convenient basis for the character lattice  of $\bT$, in view of this exact sequence: we use the characters 
$\chi_1(z, \bar z)=z $ and $\eta(z):=z\bar z =|z|^2$. These two characters, which can be thought of as the vectors $(1,0)$ and $(1,1)$ with respect to the `standard' basis of $X^\ast(\bT)$ in Example \ref{ex:quadr}, still form a $\Z$-basis of $X^\ast(\bT)$, and hence we can write $\omega_T=\frac{dz}{z}\wedge \frac{d\eta}{\eta}$. We will also use the coordinates $(z, \eta)$ on $\bT$ for the rest of this calculation.   
We write every element of $\bT(\A)$ as $a=a_f a_\infty$, where $a_f$ has the infinity component $1$ and $a_\infty=(1,(z, \eta))$
has all the components at the finite places equal to $1$.
In this notation, $\bT(\A)^1$ is defined by the condition $|z|^2=|\eta| = \|a_f\|^{-1}$. 
We write 
$$\mu^{Tama}=\mu_{\fin}^{Tama}\mu_{\infty}^{Tama},$$
where $\mu_{\fin}^{Tama} = \rho_T^{-1}\prod_p \mu_p$, and $\mu_\infty^{Tama}=|\omega|$. 
 We recall that by definition, $\omega=\frac{1}{\sqrt{|\Delta_K|}}\omega_T$. Then by the definition of $dm$, we have 
 $$  dm_\infty \wedge \frac{d\eta}{\eta} = \mu^{Tama}_\infty = \frac{1}{\sqrt{|\Delta_K|}}\frac{da_\infty}{a_\infty}\wedge \frac{d{\eta}}{\eta}, 
 $$
 and thus 
 $$dm_\infty = \frac{1}{\sqrt{|\Delta_K|}}\frac{da_{\infty}}{a_{\infty}} = \frac{1}{\sqrt{|\Delta_K|}}\frac{dz}{z}.$$
  We have computed above  in (\ref{eq:S1}) that the form $dz/z$ gives precisely the arc length $d\theta$  on the unit circle. 

Thus, the volume of $S^1$ with respect to $dm_\infty$ is  $\frac{2\pi}{\sqrt{|\Delta_K|}}$. 

Finally, we get from (\ref{eq:classnum1}) and (\ref{eq:Tama1}): 
\begin{equation}\label{eq:cnf_proof}
\begin{aligned}
& 1=\vol^{Tama}(\bT(\Q)\backslash \bT(\A_\Q)^1) = \frac{2\pi}{\sqrt{|\Delta_K|}} \vol^{Tama}(\bT(\Q)\backslash \bT(\A_\Q^\fin)) \\ 
& = \frac{2\pi}{\sqrt{|\Delta_K|}}\rho_T^{-1} \vol_{\nu_T}(K^\times \big\backslash \left(\A_{K}^\fin)^\times\right) 
= \frac{2\pi}{\sqrt{|\Delta_K|}}   L(1, \chi_K)^{-1} \frac{h_K}{w_K}, 
\end{aligned}
\end{equation}
recovering the analytic class number formula for $K$. 
We note that the product $L(1, \chi_K)$  converges only conditionally.


\subsection{What happens at $p=2$}\label{sub:two} 
We do the detailed (and elementary but tedious) analysis of the changes one needs to make at $p=2$ to all of the above calculations 
as they apply to an imaginary  quadratic extension $K$  of $\Q$ as above in order to  completely justify the global calculation above, and also point out interesting geometric differences relevant for a norm-one torus of a quadratic extension\footnote{This section can be skipped if the reader is willing to believe that (\ref{eq:classnum1}) holds at $p=2$ as well, and is not interested in the norm-1 torus (which is not used in the sequel).}.  

We write $K=\Q(\sqrt{d})$, with $d$ a square-free (and in our case, negative) integer $d=-D$. 
Our main reference for the number theory information is  e.g., \cite[\S VI.3]{frohlich-taylor}. 
\subsubsection{Quadratic extension at $p=2$}\label{subsub:2start}
First of all, recall that the ring of integers of $K$ is 
$$\ri_K=\begin{cases}
\Z\left[\frac{\sqrt{d}+1}2 \right], &\quad  d\equiv 1 \mod 4,\\
\Z[\sqrt{d}], &\quad d\equiv 2, 3 \mod 4. 
\end{cases} 
$$
We will need the fact that the ring  $\ri_K$ is generated over $\Z$ by a root of the monic polynomial $X^2-X+\frac{1-d}2$ in the first case, and of $X^2-d$ in the second case. 
Therefore the behaviour of the prime $2$ in $K$ depends on the residue of $d$ modulo $8$. Indeed,  
if $d\equiv 1 \mod 4$, then we need to look at the reduction of the polynomial $X^2-X+\frac{1-d}2$ modulo $2$; we see that it is irreducible 
over $\F_2$ if $\frac{1-d}2$ is odd, and it factors as $x(x-1)$ if $\frac{1-d}2$ is even.
Thus if $d\equiv 5\mod 8$, $K_v$ (where $v\vert 2$) is the unramified quadratic extension of $\Q_2$, while if $d\equiv 1 \mod 8$, the prime $2$ splits in $K$. In the remaining cases, $2$ ramifies: 
if $d\equiv 3, 7 \mod 8$, the relevant polynomial is $X^2-d$, and its reduction $\mod 2$ factors as  $(x-1)^2$; 
if $d\equiv 2, 6 \mod 8$, then the reduction of our polynomial $\mod 2$ is just $x^2$.

\subsubsection{From Dedekind zeta-factor to Dirichlet $L$-factor} While the local $L$-factor at $p=2$ itself looks a bit different from the other primes, the argument relating the local factor of $L(1, \chi)$ to the local factor of the Dedekind zeta-function is the same for all primes (including $2$), when the Dirichlet character associated with $K$ is defined by (\ref{eq:dirichlet_char}). We observe that above, we have just computed the Dirichlet character of $K$ at $2$: 
\begin{equation}\label{eq:char2}
\chi_K(2)=\begin{cases}  1 &\quad  d\equiv 1 \mod 8\\
-1 &\quad  d\equiv 5 \mod 8 \\
 0 &\quad  \text{ otherwise } 
\end{cases}. 
\end{equation}
By definition, the local factor at $p=2$ (as at any other prime) of the Dedekind zeta-function is 
$$\zeta_2(s)=\prod_{\mathfrak p\supset (2)}\frac1{1-{\mathrm N}_{K/\Q}(\mathfrak p)^{-s}},$$ 
where the product is over the prime ideals of $\ri_K$ lying over $2$.
It remains to recall that in all cases, for $\mathfrak{p}$ lying over $(p)$, $N_{K/\Q}(\mathfrak p)=p^f$, where $f$ is the residue degree, 
see e.g. \cite[II.4]{frohlich-taylor}.
 
\subsubsection{Quadratic extensions and norm-$1$ tori} If $F$ is a local field with residue characteristic $2$, then 
$|F^\times/(F^\times)^2|=8$, and $F$ has one unramified and 6 ramified quadratic extensions. 
(Indeed, we recall that $F^\times \simeq \Z\times \ri_F^\times$, and for a $2$-adic field, $|\ri_F^\times/(\ri_F^\times)^2|=4$).   
In the case $F=\Q_2$ everything can again be computed in an elementary way, and this is what we do in this section. 
We can list the extensions explicitly using the discussion from \S\ref{subsub:2start}: 
$\Q_2(\sqrt{5})$ is the unramified extension; $\Q_2(\sqrt{3})$ and $\Q_2(\sqrt{7})$ are the ramified extensions coming from the non-square units, and 
$\Q_2(\sqrt{2})$, $\Q_2(\sqrt{10})$, and $\Q_2(\sqrt{6})$ and $\Q_2(\sqrt{14})$ are the ramified extensions corresponding to the elements $\varpi\epsilon$ where $\varpi=2$ and $\epsilon$ is a non-square unit. 
We compute the volumes of $\bT(\Q_2)^c$ with respect to the form $\omega_T$ as in \S\ref{sub:tori}, for the tori  $\bT=\res_{K/\Q}\G_m$ as well as for the norm-1 tori of the corresponding extensions, for $K= \Q(\sqrt{d})$, with $d\equiv 5,3,2 \mod 8$, to compare the calculation  
with Examples \ref{ex:quadr} and \ref{ex:norm1}.  

We recall that the identification $E^\times = (\res_{E/F}\G_m)(F)$ is completely general; and for a quadratic extension $E$, when we think of $(\res_{E/F}\G_m)(F)$ as a torus in $\GL_2(F)$, the determinant in $\GL_2$ corresponds to the norm map $N_{E/F}$ on $E^\times$ under this identification. 
Till the end of the section, we keep the notation $\bT=\res_{E/\Q_2}\G_m$ (we are now looking locally at $p=2$, so to match the notation of 
\S\ref{sub:tori}, we have $F=\Q_2$ and $E$ is the completion of $K$ at $p=2$). 
 
{\bf 1. The unramified case, $E=\Q_2(\sqrt{5})$.} We write the elements of $E$ as $x+\varphi y$ 
where 
$\varphi=\frac{1+\sqrt{5}}2\in E$ is a root of $X^2-X-1$ and $x, y \in \Q_2$; then $\ri_E=\{x+\varphi y\mid x,y \in \Z_2\}$ (of course, for elements of $E$, the first representation is equivalent to simply writing $x'+y'\sqrt{5}$, but for the ring of integers this would not give the whole ring).  The norm map in these coordinates is 
$$N_{E/\Q_2}(x+\varphi y)=\left(x+\frac{y}2+\frac{\sqrt{5}}2y\right)\left(x+\frac{y}2-\frac{\sqrt{5}}2y\right)= x^2+xy-y^2.$$
This causes small changes to our na\"ive calculations of Example  \ref{ex:quadr}. 
In particular, the pullback to $\bT$ of the differential form $\frac{dx}x$ on $\G_m$ is now (exactly as in (\ref{eq:pullback_res}), with 
$\varphi'=\frac{1-\sqrt{5}}2$ denoting the Galois conjugate of $\varphi$)
\begin{equation}\label{eq:pullback2unram}
\begin{aligned}
&
\omega_T=\frac{d(x+\varphi y)}{x+\varphi y}\wedge\frac{d(x+\varphi' y)}{x + \varphi' y}
=\frac{(dx+\varphi dy)\wedge(dx +\varphi' dy)}{x^2+xy- y^2}\\
& = \frac{\varphi'-\varphi}{N_{E/F}(x+ \varphi y)} dx\wedge dy
=\frac{-\sqrt{5}}{N_{E/F}(x+ \varphi y)} dx\wedge dy.
\end{aligned}
\end{equation} 
Note the absence of the factor  $2$, compared with (\ref{eq:pullback_res}), which would have caused trouble here. 

From the volume to the point-count: since the extension is unramified, this part does not change. 
In fancy terms, we can say that the so-called standard model over $\Z_2$ for $\bT$,  defined by the coordinates $x$ and $y$ that we chose, is smooth. The set of $\F_2$-points of 
its special fibre (in simple terms, the reduction mod $2$, which is the uniformizer of our unramified extension) is still 
$(\F_4)^\times$. 

Thus, the volume formula (\ref{eq:res_quadr}) still holds in this case. 
If this extension was obtained as the completion at $2$ of a quadratic extension $K=\Q(\sqrt{d})$ of $\Q$, then $d\equiv 5 \mod 8$, and therefore in this case the discriminant of $K$ is just $\sqrt{d}$. 
Hence, the relation (\ref{eq:classnum1}) holds without any modification. 

The norm-1 torus of this extension is treated very similarly to the unramified case with $p\neq 2$; and the final answer for its volume is given by the same formula as in the case $p\neq 2$. 

{\bf 2. Ramified case 1,  $E=\Q_2(\sqrt{2})$.} We have $\ri_E=\{x+y\sqrt{2}\mid x,y\in \Z_2\}$, and the norm map is 
$N_{E/\Q_2}(x+\sqrt{2}y)= x^2-2y^2$ similarly to the $p\neq 2$ case.
In this case the calculation of the form $\omega_T$ applies verbatim, so we get the extra factor $\frac12$ in (\ref{eq:ex1}). Specifically, 
(\ref{eq:ex1}) now becomes:
\begin{equation*}
\vol_{|\omega_T|}(T^c)=\frac1{2\sqrt{2}}\int_{\{(x, y) \in \Z_2^2:|x^2-2y^2|=1\}} dx dy.
\end{equation*} 
The condition $x^2-2y^2\in \Z_2^\times$ is still equivalent to $x \in \Z_2^\times$, and we obtain  
that $\vol_{|\omega_T|}(T^c)=\frac{1}{2\sqrt{q}}\frac{q-1}q$; here $q=2$, we are just writing it this way for the ease of comparison with 
(\ref{eq:res_quadr}). 
However, the relation (\ref{eq:classnum1}) again holds without modification, since  if the completion of $K=\Q(\sqrt d)$ at $2$ is ramified, then $\Delta_K=4\sqrt{d}$, and $|\Delta_K|_2^{1/2}$ has the extra factor $2$ as well. 

\begin{example}\label{ex:q22}
\emph{The norm-1 torus of $E=\Q_2(\sqrt{2})$:}
Unlike the full torus obtained by the restriction of scalars,  for the norm-1 subtorus the reduction of the volume computation to counting residue-field points looks very different from Example \ref{ex:norm1}.
We include this point-count exercise without a full discussion of its implications for the computation of the volume with respect to $\omega_T$ or $\omega^\can$, to illustrate the difficulties that arise when the reduction $\mod \varpi$ is not smooth.

We write the $2$-adic expansions $x=x_0+2x_1+4x_2+ \ldots,$  $y=y_0+2y_1+4y_2+ \ldots,$ with $x_i, y_i \in \{0,1\}$.
Then 
\begin{equation}\label{eq:2squaring}
\begin{aligned}
x^2&=  (x_0+2x_1+4x_2+ 8x_3+ 16x_4 + \ldots)^2 \\
{}&= x_0^2+4(x_0x_1+x_1^2)+8(x_0x_2)+16(x_2^2+x_1x_2+x_0x_3)\\
{}&+2^5(x_0x_4+x_1x_3)+2^6(x_3^2+x_0x_5+x_1x_4+x_2x_3)+ \ldots.
\end{aligned}
\end{equation}

Thus the condition $x^2-2y^2=1$ becomes (each line comes from the congruence modulo the next power of $2$ (indicated in the left column), 
and each congruence is congruence $\mod 2$): 
\begin{equation*}
\begin{aligned}
&\mod 2: &\quad x_0=1;\\
&\mod 2^2: &\quad  y_0=0; \\
&\mod 2^3: &\quad x_0x_1+x_1^2 \equiv 0; \\
&\mod 2^4: &\quad x_0x_2 -(y_0y_1 +y_1^2) \equiv 0; \\
&\mod 2^5: &\quad x_2^2+x_1x_2+x_0x_3 - (y_0y_2) \equiv 0; \\
&\mod 2^6: &\quad x_0x_4+x_1x_3 - (y_2^2+y_1y_2+y_0y_3) \equiv 0; \\
&\mod 2^7: &\quad x_3^2+x_0x_5+x_1x_4+x_2x_3 -(y_0y_4+y_1y_3) \equiv 0; \\
&\ldots & \quad \ldots.
\end{aligned}
\end{equation*}
These equations yield: 
\begin{equation}\label{eq:2digits}
\begin{aligned}
&x_0=1, &\quad y_0=0,\\
&x_1 \text{ is arbitrary},  &\quad y_1 \text{ is arbitrary},\\
&x_2 = y_1, &\quad y_2 \text{ is arbitrary},\\
&x_3 = x_2^2+x_1x_2, &\quad  y_3 \text{ is arbitrary},\\
&\ldots 
\end{aligned}
\end{equation}
This illustrates that Hensel's Lemma, as expected, starts working once we have a solution $\mod 8$, but 
\emph{not for solutions $\mod 2$}. 
In fancier terms, we have the reduction $\mod 8$ map defined on the set of $(\Z_2\times \Z_2)$-solutions of the norm equation $x^2-2y^2=1$.
The image of this map is the set $\{(x_0, y_0, x_1, y_1, x_2, y_2)\}$ defined by the first three lines of (\ref{eq:2digits}) inside 
$\A^6(\F_2)$; we see that it is a $3$-dimensional hyperplane in $\A^6(\F_2)$.  
The fibre of the reduction $\mod 8$  map over each point in its image  is a translate of $(2^3)\subset Z_2$,  so the volume of each fibre is $\frac{1}{2^3}$. 
We obtain that the volume of our set of solutions is $\frac{\#\A^3(\F_2)}{2^3}$. 
Note that this is geometrically quite different from the answer in the ramified case in Example \ref{ex:norm1}, where the image of the reduction $\mod p$ map was two copies of an affine line, and therefore the image of the reduction $\mod p^2$ map then was \emph{two} copies of an affine plane over $\F_p$; and the image of reduction $\mod p^3$ map, inside $\A^6(\F_p)$, was \emph{two} copies of $\A^3(\F_p)$.
\end{example} 

{\bf 3. Ramified case 2,  $E=\Q_2(\sqrt{3})$.}  
Exactly as in the case $E=\Q_2(\sqrt{2})$ considered above, for $\bT=\res_{E/\Q_2}\G_m$, 
the volume $\vol_{|\omega_T|}(T^c)$ has the extra factor of $\frac12$ compared with the $p\neq 2$, and so does the discriminant (basically, the calculation of the volume of $T^c$ is not sensitive to \emph{which} ramified extension of $\Q_2$ we are considering). 

\begin{example} \emph{The norm-1 torus of $E=\Q_2(\sqrt{3})$.} 
The calculation is also similar to \ref{ex:q22} above: we use (\ref{eq:2squaring}) to count solutions of the equation $x^2-3y^2=1$, but the geometry looks slightly different.  As above, from  (\ref{eq:2squaring}) we get: 
\begin{equation*}
\begin{aligned}
&\mod 2: &\quad x_0^2+y_0^2=1;\\
&\mod 2^2: &\quad  x_0^2-3y_0^2\equiv 1 \mod 4; \\
&\mod 2^3: &\quad x_0x_1+x_1^2 - 3(y_0y_1+y_1^2)\equiv 0 \mod 2; \\
\end{aligned}
\end{equation*}

From the first two equations, we get: $x_0=1$, $y_0=0$ (note that the second congruence $\mod 4$  does not allow for the option $x_0=0$, $y_0=1$, in contrast to the ramified case with $p\neq 2$; this is an example of a solution $\mod 2$ that \emph{does not lift} to a solution $\mod 4$). 
Then the third equation becomes $x_1+x_1^2-3y_1^2\equiv 0 \mod 2$, which allows for arbitrary $x_1$ and makes $y_1=0$.
Next, from the truncation $\mod 2^4$, we get (plugging all this information in): 
$$
\quad 4(x_1+x_1^2) -8 (x_2-3\cdot 0\cdot y_2)  \equiv 0 \mod 16.
$$
If $x_1=0$, we get $x_2=0$; if $x_1=1$, then $x_2=1$.   
Summing it up, so far we have:
\begin{equation}\label{eq:truncation_q2rt3}
\begin{aligned}
&x_0=1, &\quad y_0=0,\\
&x_1 \text{ is arbitrary},  &\quad y_1 =0,\\
&x_2 = x_1, &\quad y_2 \text{ is arbitrary},\\ 
\end{aligned}
\end{equation}
Note that the pattern so far has been something we have not seen before: 
the congruences modulo an odd power of $2$ might have some `carry-over' to the next power; then the congruence modulo the next even power
forces the truncated expression to become literally zero with no carry-over. 

Continuing with one more step:
\begin{equation*}
\begin{aligned}
&\mod 2^5: & {}\\
&{}  \quad x_2^2+x_1x_2+x_0x_3 - 3(y_2^2+y_1y_2+y_0y_3) \equiv 0 \mod 2; \\
&\mod 2^6: &{}\\
&{}\quad  (x_2^2+x_1x_2+x_0x_3 - 3(y_2^2+y_1y_2+y_0y_3)) + (x_0x_4+x_1x_3 - 3(y_0y_4+y_1y_3)) \equiv 0,
\end{aligned}
\end{equation*}
where the last congruence is $\mod 4$. 
As we plug in what we already know about the first terms, the first equation becomes $x_3+y_2^2\equiv 0 \mod 2$, so $x_3=y_2$. 
Plugging this into the second equation (and ignoring the squares), we get 
$(2x_2-2x_3)+(x_4+x_3)\equiv 0 \mod 4$, which determines $x_4$ uniquely. 
Again we notice that by now Hensel's Lemma started working as expected: at every step we get one linear relation in two unknown parameters, 
so the fibre over each truncated solution is an affine line over $\Q_2$. 

To summarize, (\ref{eq:truncation_q2rt3}) says that the image of the reduction $\mod 8$ map is 
a plane in $\A^6(\F_2)$, and thus for the volume of $T^c$, we get 
$\frac{\#\A^2(\F_2)}{2^3}$.
\end{example}

Now we are ready to return to the global calculations. 

\subsection{Global orbital integrals in $\GL_2$}\label{sub:orb_glob} 
Let us put together the information we have so far about the orbital integrals with respect to the canonical vs. geometric measures, and the information about the volume of $K^\times \big\backslash \left(\A_{K}^\fin\right)^\times$ that we just obtained. 

Let $\gamma\in \GL_2(\Q)$ be a regular semisimple element, such that its centralizer $T$ is non-split over $\R$.
Let $K$ be the quadratic extension of $\Q$ generated by the eigenvalues of $\gamma$. Then 
$T=\bT(\Q)$ with
$\bT=\res_{K/\Q} \G_m$, as in \S\ref{sub:an_cnf}. 
Let $f^\fin =\otimes_p f_p$, with $f_p$ equal to ${\bf 1}_{\GL_2(\Z_p)}$, 
the characteristic function of $\GL_2(\Z_p)$, for almost all $p$, 
be a test function on $G(\A^\fin)$.  
Taking the product of the relations (\ref{eq:canon_to_geom}) at every prime $p$, and applying the product formula to 
the  absolute values $\prod_p|D(\gamma)|_p$ and $\prod_p|\Delta_K|_p$, we obtain: 
\begin{equation}\label{eq:ofin}
O_\gamma^{\geom}(f^\fin)= |D(\gamma)|^{-1/2}|\Delta_K|^{1/2} L(1, \sigma_{G/T}) O_\gamma^{\can}(f^\fin).
\end{equation} 

We observe that the canonical measure on the centralizer of $\gamma$ (or any  measure coinciding with it at almost all places) is convenient for defining a global orbital integral because with such a measure, all but finitely many factors $O_{\gamma_p}(f_p)$ are equal to  $1$, and thus the global orbital integral is a finite product, namely, the product over the primes that divide $D(\gamma)$ and the primes where $f_p\neq {\bf1}_{G(\Z_p)}$, and there is no question of convergence.  

In the Trace Formula,  the  orbital integrals are weighted by volumes; and the 
volume has to be taken with respect to the same measure on the centralizer that was used to define the orbital integral. 
Comparing (\ref{eq:ofin}) with  (\ref{eq:cnf}),  we  see explicitly that for $\GL_2$, the volume term contains some of the factors that also appear when we pass from the canonical measure  to the geometric measure on the orbit. 
Now we are ready to explicate an observation (that is implicit in the work of Langlands) that switching to the geometric measure makes the volume term disappear in the case 
$\bG=\GL_2$, in addition to making the orbital integrals at finite places better-behaved. This comes at the cost of now having the orbital integral expressed as only a \emph{conditionally convergent} infinite product. 
This theorem can be thought of as the main point of this note.

\begin{theorem}\label{thm:main}
 Let $\gamma$ be a regular elliptic element of $\GL_2$ as above, that splits over a quadratic extension $K$. 
Let $\bT=\res_{K/\Q}\G_m$.  Then
$$\vol^\can(\bT(\Q)\backslash \bT(\A^\fin)) O_\gamma^\can(f^\fin) =\frac{|D(\gamma)|^{1/2}}{2\pi}
O_\gamma^\geom (f^\fin).$$
\end{theorem} 

\begin{proof} 
Combining the relation (\ref{eq:ofin}) with the calculation in (\ref{eq:cnf_proof}), we see that the factors $L(1,\chi_K)$
and $|\Delta_K|^{1/2}$ cancel out, and 
we obtain: 
$$\begin{aligned}
&\vol^\can(\bT(\Q)\backslash \bT(\A^\fin))O_\gamma^\can(f^\fin) 
= \frac{\sqrt{|\Delta_K|}}{2\pi}L(1, \chi_T)\vol^{Tama}(\bT(\Q)\backslash \bT(\A)^1) O_\gamma^\can(f^\fin)\\
& = 
\frac{|D(\gamma)|^{1/2}}{2\pi} \tau(\bT)
O_\gamma^\geom (f^\fin).
\end{aligned}$$
Since $\bT$ is obtained from $\G_m$ by restriction of scalars, we have $\tau(\bT)=1$, as discussed above in \S\ref{subsub:Tama}. 
\end{proof} 

We make a few remarks: 
\begin{remark}
{\bf 1.} In our statement, the left-hand side of the equation is actually independent of the choice of the measure on $\bT(\Q_p)$ at every finite place $p$, as long as the volume and the orbital integral are taken with respect to the same measure; this is consistent with the right-hand side, which does not involve any measure on $\bT$ at all. 

{\bf 2.} The volume that appears in the Trace Formula is $\vol(\bT(\Q)\backslash \bT(\A)^1)$; when passing from the volume appearing on the left-hand side of Theorem \ref{thm:main} to this volume, the ratio between them will depend on the precise choice of the normalization of the component of the measure at infinity. Calculations  of this sort (for a general reductive group, with various specific choices of the measure at infinity) appear in \cite{gan-gross:haar} and \cite{gross:motive}.

{\bf 3. }The right-hand side of the relation in Theorem  \ref{thm:main} might be preferable in two ways: 
there is no complicated  volume term,   and the orbital integral 
has the local components that are continuous as functions on the
Steinberg-Hitchin base (however, now the orbital integral on the right is an infinite product that converges conditionally). 

{\bf 4. } The proof of the theorem does not use the analytic class number formula. Moreover, 
the proof is general, except for three pieces: 
\begin{enumerate}  
\item The specific knowledge that at every finite place $v$, 
$\vol_{\omega_T}(\cT^0)$ is $|\Delta_K|_v^{1/2}L_v(1, \chi_T)^{-1}$. 
See \cite{gan-gross:haar} for a generalization of such a relation. 

\item For general tori Tamagawa numbers can be difficult to compute explicitly 
(see \cite{rud:tori} for some partial results),  but for the maximal tori in $\GL_n$ they are $1$. 

\item The factor $2\pi$ in the denominator is specific to $\GL_2$; in general, it needs to be replaced with the factor 
determined by the component of the chosen measure at infinity.
\end{enumerate} 
\end{remark}

We conclude with a brief discussion of Eichler-Selberg Trace Formula, since it is the starting point for Altug's lectures. This discussion is entirely based on \cite{knightly-li}. 
\subsection{Eichler-Selberg  Trace Formula for $\GL_2$}
The Eichler-Selberg Trace Formula expresses the trace of a Hecke operator $T_n$ on the space $S_k(N)$ of cusp forms of weight $k$  and level $N$. 
For simplicity of exposition, we set $N=1$ in this note; in this case the central character is also trivial. 
In this setting, the Eichler-Selberg Trace Formula states: 
\begin{equation}\label{EichlerTF} 
\begin{aligned}
& n^{1-k/2} \tr T_n & = \frac{k-1}{12}\begin{cases} &1, n \text{ is a perfect square} \\ & 0 \text{ otherwise } \end{cases} \\
&{}& -\frac12\sum_{t^2<4n} \frac{\rho^{k-1}-\bar\rho^{k-1}}{\rho-\bar\rho} \sum_m h_w\left(\frac{t^2-4n}{m^2}\right) \\
&{}& -\frac12\sum_{d\mid n} \min(d, \frac{n}{d})^{k-1},   
\end{aligned} 
\end{equation}
where  in the middle line  $\rho$ and $\bar{\rho}$ are roots of the polynomial $X^2-tX+n$, and 
$h_w\left(\frac{t^2-4n}{m^2}\right)$ is the weighted class number of the order in $\Q[\rho]$ that has discriminant $\frac{t^2-4n}{m^2}$.
The sum over $m$ runs over the integers $m\ge 1$ such that $m^2$ divides $t^2-4n$, and $\frac{t^2-4n}{m^2}$ is $0$ or $1$ $\mod 4$. 

The goal of this section is to sketch, without any detail, a connection between this formula and the (geometric side of)  Arthur-Selberg Trace formula. 

\subsubsection{The test function} 
We start with a very brief recall of the connection between modular forms and automorphic forms on $\GL_2$. We refer to e.g., Knightly and Li \cite{knightly-li} for all the details. 

Let $\bG=\GL_2$ for the rest of this section. 
A cusp form of weight $k$ generates (as a representation of $\bG(\A_\Q)$ under the  action by right translations) a closed subspace $(\pi, V)$ of $L^2_0(\bG(\Q)\backslash \bG(\A))$; for a level $1$ cusp form (which is our assumption here, so in particular we should assume $k>2$), the central character of $\pi$ is trivial.  
By Flath's theorem, the representation $\pi$ factors as a restricted  tensor product 
$$\pi= \pi_\fin\otimes \pi_{\infty} = \otimes'_{p} \pi_p\otimes \pi_{\infty}.$$
Since $(\pi, V)$ came from a cusp form of weight $k$, we have $\pi_{\infty}=\pi_k$ -- the discrete series representation of highest weight $k$. 

Every function $f\in C_c(\bG(\A)) $ gives rise to a linear operator $R(f)$ on $L^2(\bZ(\A)\backslash \bG(\A))$  defined by 
$$(R(f)\phi)(g):= \int_{(\bZ\backslash \bG)(\A)}f(x)\phi(gx) dx.$$
In this language, the Hecke operator $T_n$ on $S_k$ is precisely  $n^{k/2-1}R(f_{n,k})$, where $f_{n,k}$ is a specific test function in the space $L^2(\bG(\Q)\backslash \bG(\A)^1)$. We quote the definition of this test function from \cite{knightly-li}.  
\begin{itemize} 
\item Let $f_\infty$ be a matrix coefficient of the representation $\pi_k$. 
By orthogonality of matrix coefficients, this ensures that  the image of $R(f)$ is contained in $(\pi, V)$. 
(Since $(\pi, V)$ is irreducible, this means $R(f)$ projects onto $V$). 

\item For $p\nmid n$, let $f_p$ be the characteristic function of $\bZ(\Q_p) \bG(\Z_p)$, 
\item For  $p\mid n$, let $f_p$ be the characteristic function of $\bZ(\Q_p) M_{n,p}$, where $M_{n,p}$ is the set of matrices of determinant $n$ in $M_2(\Q_p)$ (we prefer to think of it as the characteristic function of the union of the double cosets of the Cartan decomposition for $\GL_2(\Q_p)$ of determinant $n$).

\end{itemize}

\subsubsection{The transition to Arthur Trace Formula}
We plug $f:=f_{n,k}$ into Arthur's Trace Formula for $\GL_2$, and examine the geometric side. 
Since for our test function $f$, the continuous and residual parts of the spectral side vanish, the geometric side in fact equals
$\tr R(f)$ (see \cite[\S 22]{knightly-li}). 
Finally, Knightly and Li show that:  
\begin{itemize}
\item The first line of (\ref{EichlerTF}) matches   
the contribution of the trivial conjugacy class;
\item The last line matches the contribution of the unipotent and hyperbolic conjugacy classes.
\item The middle line matches the contribution of the elliptic conjugacy classes.
\end{itemize}  
We discuss why the last claim is plausible.
By definition of the test function $f$, its orbital integrals vanish on all elements $\gamma$ such that $\det(\gamma)\neq n$; hence, in the geometric side of Arthur's Trace Formula,  we are left with the sum over $\gamma\in \bG(\Q)$ satisfying $\det(\gamma)=n$. 
The conjugacy classes in $\GL_2$ are parametrized by characteristic polynomials, and the elliptic ones correspond to the polynomials with negative discriminants, so at least superficially, we recognize the sum  over the integers $t$ such that $t^2<4n$ as a sum over the rational elliptic conjugacy classes. 

Next, note that the expression $\frac{\rho^{k-1}-\bar\rho^{k-1}}{\rho-\bar\rho} $ is the value of the character of $\pi_k$ on the corresponding conjugacy class; thus, we recognize it as the orbital integral of $f_\infty$ (see e.g. \cite[\S 1.11]{kottwitz:clay} for the discussion of characters as orbital integrals of matrix coefficients; see also \cite[Ch.I, \S5.2]{gelfand-graev}). 
 
Knightly and Li show (in our notation):
\begin{equation} 
O_\gamma^\can(f^\fin) = \sum_m h_w\left(\frac{t^2-4n}{m^2}\right),
\end{equation} 
where the sum over $m$ is as in (\ref{EichlerTF}).
While the appearance of class numbers in our earlier calculations is suggestive,  and proves this relation in the trivial case when $t^2-4n$ is square-free, it appears that our  arguments  are insufficient for getting  a simpler proof of this claim in general (other than by essentially direct computation of the both sides, or matching the computation of the right-hand side in \cite{knightly-li} with the calculation on the building in \cite{kottwitz:clay}, the results of which we already quoted above). 
A similar statement for $\GL_n$, relating orbital integrals to sums of class numbers of orders, 
is proved by Zhiwei Yun, \cite{yun:dedekind}.


\section{Appendix A. Kirillov's form on co-adjoint orbits: two examples}
\begin{center} {by Matthew Koster } \end{center} 

In this appendix we illustrate Kirillov's construction of a volume form on co-adjoint orbits in a Lie algebra. 
Here we work over $\R$ in order to be able to use the intuition from calculus. In these examples, we also relate this form to the geometric measure discussed in the article. 
\footnote{This work was part of an NSERC summer USRA project in the summer of 2019; we acknowledge the support of NSERC.}

\subsection{The coadjoint orbits}\label{sub:coorb}
A large part of this section is quoted from \cite[\S 17.3]{kottwitz:clay} for the reader's convenience and to set up notation.
Let $G$ denote a semisimple Lie group, $\mathfrak{g}$ its Lie algebra,  and $\mathfrak{g}^*$ the linear dual space of $\mathfrak{g}$.
We denote elements of $\fg$ by capital letters, e.g. $X$, and use starts to denote elements of $\fg^\ast$; 
unless explicitly stated there is no a priori relationship between $X$ and $X^*$.

$G$ acts on $\mathfrak{g}$ by Ad and acts on $\mathfrak{g}^*$ by Ad$^*$, where
$$\langle \Ad^*(g)(X^*), X \rangle = \langle X^*, \Ad(g^{-1})(X) \rangle.$$ 
Let $\ri(X^*) \subset \mathfrak{g}^* $ denote the  orbit of $X^*$ under this action (called a co-adjoint orbit). \\
We recall that the differential of the adjoint action  $\Ad$ of $G$ is the action of $\mathfrak{g}$ 
on itself by $\ad$ where $\ad_X(Z) = [X,Z]$.
The co-adjoint action of $\fg$ 
on $\mathfrak{g}^*$  is 
the differential of $\Ad^*$; 
we denote it 
by $\ad^*$; explicitly, this action is defined by
$\langle \ad_X^*(Y^\ast), Z \rangle = \langle Y^\ast, [Z,X] \rangle$. 

A choice of an element $X^\ast\in \fg^\ast$  defines a map 
$\varphi_{X^\ast}:G \to \mathfrak{g}^*$ 
by $\varphi_{X^\ast}(g) = \Ad^*_g(X^*)$. 
The differential of this map at the identity $e \in G$ gives an identification of  $\mathfrak{g}/\mathfrak{c}(X^*)$ with the tangent space $T_{X^*}\ri (X^*)$ at $X^\ast$, 
defined by $X \mapsto \ad^*_X(X^*)$. Here 
$\fc(X^\ast)$ is the stabilizer of $X^\ast$ under $\ad^\ast$, and we are viewing $T_{X^*} \ri(X^*)$ as a subspace of $T_{X^*} \mathfrak{g}^* \cong \mathfrak{g}^*$. We denote this identification by $\Phi_{X^\ast}:\mathfrak{g}/\fc(X^\ast) \to T_{X^\ast}\ri(X^\ast)\hookrightarrow \mathfrak{g}^*$. 
The element $X^\ast$ gives an alternating form $\omega_{X^\ast}$ on $\fg$, defined by 
\begin{equation}
\omega'_{X^\ast}(X, Y):= \langle X^\ast, [X, Y] \rangle = -\langle \ad^\ast(X) X^\ast, Y \rangle.
\end{equation}
This form clearly vanishes on $\fc(X^\ast)$, and gives a non-degenerate bilinear form on $\fg/\fc(X^\ast)$, which we have just identified with $T_{X^\ast}\ri(X^\ast)$. 
Thus, given a co-adjoint orbit $\mathcal O$, we get a symplectic $2$-form $\omega'$ on it by letting the value of $\omega'$ at $X^\ast\in \mathcal O$ equal $\omega'_{X^\ast}$. In particular, as a manifold, $\mathcal O$ has to have even dimension; if its dimension is $2k$, then 
the $k$-fold wedge product of the form $\omega'$  gives a volume form on $\mathcal O$. 

Over a field of characteristic zero, we can identify a semisimple Lie algebra with its dual; we will use the Killing form for this. 
Then the adjoint orbits in $\fg$ 
get identified with the co-adjoint orbits in $\fg^\ast$, and thus we get a very natural algebraic volume form on each adjoint orbit in $\fg$. 
Here our goal is to compute this form explicitly in two examples: the regular nilpotent orbit in $\mathfrak{sl_2}(\R)$ and a semisimple 
$\SO_3(\R)$-orbit in $\mathfrak{so}_3(\R)$ (we do not use the accidental isomorphism in this calculation). In both cases the orbit will be two-dimensional, so we are just computing the form denoted by $\omega'$ above. 

\subsection{Rewriting the form as a form on an orbit in $\fg$}\label{sub:form}
Given $X_0^\ast\in \fg^\ast$, we compute the form $\omega$ on the orbit of an element $X_0\in \fg$ that corresponds to 
$X_0^\ast\in \fg^\ast$ under the isomorphism defined by Killing form,  in three steps:
\begin{enumerate}
\item
Compute the map $\Phi_{X_0^\ast}$.
\item For $\widetilde{v_1}, \widetilde{v_2} \in T_{X_0^*}\ri(X_0^*)$ find $v_1, v_2 \in \mathfrak{g}$ with $\Phi_{X_0^\ast}(v_i) = \widetilde{v_i}$  
for $i=1,2$, and then evaluate $\omega_{X_0^*}(\widetilde{v_1}, \widetilde{v_2} ) = \langle X_0^*, [v_1,v_2]\rangle$.  
\item Using Killing form, identify $\fg$ with $\fg^\ast$, which identifies a co-adjoint orbit of $X_0^\ast$ in $\fg^\ast$ with an adjoint
 orbit  of an element $X_0\in \fg$.
Then use the adjoint action to explicitly define the volume form  $\omega_X$ on $T_X \ri$ at a point $X$ in this orbit by pulling back the form $\omega_{X_0}$ . 
\end{enumerate}

\subsection{Example I: a regular nilpotent orbit in $\fsl_2(\R)$}\label{sub:sl2R} 
Let $G= \SL(2; \mathbb{R})$, $\mathfrak{g} = \mathfrak{sl}(2; \mathbb{R})$, and consider the standard basis 
$\{\be, \bff, \bh\}$ 
for  $\mathfrak{g}$ given by: 
\begin{align*}
{\bf e} = \begin{bmatrix}
0 & 1 \\
0 & 0
\end{bmatrix}, \ \ \ \ 
{\bf f} = \begin{bmatrix}
0 & 0 \\
1 & 0
\end{bmatrix}, \ \ \ \
{\bf h} = \begin{bmatrix}
1 & 0 \\
0 & -1
\end{bmatrix}. 
\end{align*}

Let $\{\be^*, \bff^*, \bh^*\}$  be the basis for $\mathfrak{g}^*$ dual to $\{\be, \bff, \bh\}$ under the Killing form. Explicitly this means that $\be^*(\bff)=4$, $\bff^*(\be)=4$, and $\bh^*(\bh)=8$. 
We will compute Kirillov form on the co-adjoint orbit $\ri_{\bff^\ast}$ 
of $\bff^*$, which we identify with the adjoint orbit of $\be$ in $\fg$. 

When we refer to coordinates $x,y,z$ on $\fg$, it is with respect to our chosen basis $\{\be, \bff, \bh\}$. Given this choice of coordinates, we 
have the basis of the space of $1$-forms on $\fg$ given by $dx$, $dy$, $dz$. 

Under the isomorphism $\fg^\ast \simeq \fg$ defined by Killing form, 
a point $(x,y,z)\in \fg^\ast$ is mapped to 
$\left[\begin{smallmatrix} z/2 & x \\ y & -z/2\end{smallmatrix}\right]\in \fg$. 
We can describe the orbit of $\be$ very explicitly in these coordinates.

\subsubsection{The nilpotent cone} 
If we are working over $\R$, then the set of nilpotent elements in $\fg$ forms a cone: indeed, for a nilpotent matrix we have 
$\det\left[\begin{smallmatrix} z/2 & x \\ y & -z/2\end{smallmatrix}\right]=0$, i.e., $z^2+4xy=0$. 
(One can easily see that it is, indeed, a cone by the change of coordinates  $u=x+y$, $v=x-y$: in these coordinates, the equation of the orbit becomes $z^2+u^2=v^2$). See e.g. \cite[\S 2.3]{debacker:clay}  for more detail of this picture.  
 
The nilpotent cone consists of 3 orbits of $\SL_2(\R)$: $\{0\}$, the half-cone with $v>0$ (which is the orbit of $\be$), and the half-cone with $v<0$ (the orbit of $\bff$).  
One can explicitly compute that given a matrix $X_0:=\left[\begin{smallmatrix} z_0/2 & x_0 \\ y_0 & -z_0/2\end{smallmatrix}\right]\in \fg$ satisfying 
$z_0^2=-4x_0y_0$ (which forces $x_0y_0<0$ if we are working over $\R$), the element $g_0$ below provides the conjugation so that
$X_0 = Ad^*_{g_0} (\be)$ (it is convenient for us to write $g_0$ as a product of a diagonal and a unipotent matrix with a view toward further calculations): 
\begin{equation}\label{eq:g}
g_0 = \begin{bmatrix}
\sqrt{x_0} & 0 \\ 
 \sqrt{-y_0} & \frac{1}{\sqrt{x_0}} 
\end{bmatrix}
= \begin{bmatrix}
\sqrt{x_0} & 0 \\ 
0 & 
\frac{1}{\sqrt{x_0}} \\
\end{bmatrix} \cdot  \begin{bmatrix}
1 & 0 \\  \sqrt{-{x_0}{y_0}} & 1 \\
\end{bmatrix}. 
\end{equation}
Note that since $v=x_0-y_0>0$,  and $x_0y_0<0$, we have $x_0>0, y_0<0$, which explains our choice of signs inside the square roots.

\subsubsection{A measure from calculus} Given that our orbit is an open half-cone, we can write down a natural measure on it as a parametrized surface. 
In fact, as we think of a parametrization for this cone, we can be guided by the fact that we are looking for an $\SL_2(\R)$-invariant measure. We recall Cartan decompositon: $\SL_2(\R)=UK$, where $U$ is the group of lower-triangular unipotent matrices and $K=\SO_2(\R)\simeq S^1$. 
The adjoint action of $\SO_2(\R)$ is given by a fairly complicated formula (see \cite[\S 2.3]{debacker:clay}), but at the same time one has the obvious action of $S^1$ on the cone by rotations; thus it is reasonable to make a rotation-invariant measure on our cone. 
Therefore, we use cylindrical coordinates to parametrize the cone $z^2+u^2=v^2$ and arrive at $z=t\cos(\theta)$, $u=t\sin(\theta)$, $v=t$, which translates to the parametrization 
$\rho:(0, \infty) \times [0, 2\pi) \to \mathfrak{g}$  given by: 
\begin{equation}
\rho(t, \theta) = (t (\cos \theta +1), t (\cos \theta -1), t \sin \theta).
\end{equation}
The natural volume form on the cone is then $dt\wedge d\theta$;  below we see how it compares to Kirillov's volume form. 
Note: now that we made this guess at a form, we could just express the actions of $U$ and $K=\SO_2(\R)$ on the cone in the $(t, \theta)$ coordinates, and check if this form is invariant. However, we prefer to compute the Kirillov's form directly and derive the comparison this way.

\subsubsection{Computing Kirillov's form}

\emph{Step 1. The calculation at $\bff^\ast$.}
We compute the map $\Phi_{\bff^\ast}$ defined in \S \ref{sub:coorb}. 
It is a map from $\fg$ to $\fg^\ast$, so given $(x,y,z)=x\be+y\bff+z\bh \in \fg$, its image under $\Phi_{\bff^\ast}$ is a linear functional on $\fg$. Thus it makes sense to write $\langle \Phi_{\bff^\ast}(x,y,z), (x',y',z')\rangle$, where $(x',y',z')\in \fg$. We evaluate:  
\begin{align*}
\langle \Phi_{\bff^\ast}(x\be+y\bff+z\bh),(x', y',z') \rangle &= \langle \bff^*, [x'\be+y'\bff+z'\bh, x\be+y\bff+z\bh] \rangle  \\ &= \langle \bff^*, 
2(z'x-x'z)\be + 2(y'z-z'y)\bff+(x'y-y'x)\bh \rangle \\
 &
= 8(z'x-x'z) \\ &= \langle -(2z\bff^*-x\bh^*), (x', y',z')\rangle.
\end{align*}
Thus 
\begin{equation}\label{eq:Phi-e}
\Phi_{\bff^\ast}(x\be + y\bff + z\bh ) = - (2z \bff^* - x\bh^*).
\end{equation}

\emph{Step 2.} 
We need to find a preimage under $\Phi_{\bff^\ast}$ for a vector $\tilde v \in T_{\bff^\ast} \ri_{\bff^\ast}$. 
We recall that this tangent space is identified with a subspace of $\fg^\ast$, which we later plan to identify with a subspace of $\fg$. Because of this latter anticipated identification, we write $\tilde v = x \bff^\ast + y \be^\ast + z \bh^\ast $.  We see directly from (\ref{eq:Phi-e}) 
 that for $\tilde v$ to be in the image of $\Phi_{\bff^\ast}$ it has to satisfy $y =0$,  and then $v =  z \be - \frac{x}2 \bh$ satisfies 
 $\Phi_{\bff^\ast}(v)=\tilde v$. 
 
 We are now ready to compute the form $\omega_{\bff^\ast}$. Let 
 $\widetilde{v_i} = x_i\bff^*+  y_i\be^* + z_i\bh^*$  for $i=1,2$; then as discussed above, we can take $v_i= z_i \be  - \frac{x_i}2 \bh$.
 Then $[v_1, v_2]  
   = (z_1x_2-x_1z_2)\be$, and  
finally we have that: 
\begin{equation}\label{eq:omega-e}
\omega_{\bff^*}  ( \tilde v_1, \tilde v_2 ) =\langle \bff^*, [v_1,v_2]\rangle  = 4(z_1x_2-x_1z_2).
\end{equation} 
Under the identification of the differential $2$-forms on a vector space with alternating $2$-tensors, we recognize this form as 
$-4 dx^*\wedge dz^*$, which we identify with the form $\omega_{\be}:= -4dx\wedge dz$ on the adjoint orbit of $\be\in \fg$.

\emph{Step 3. Pullback of $\omega_{\be}$ under the adjoint action.} We compute the operator $\Ad_{g_0}$ for the element $g_0$ from (\ref{eq:g}) in our coordinates, in order to use it to pull back the form $\omega_{\be}$. 
By the right-hand side of (\ref{eq:g}), the matrix of $\Ad_{g_0}$ in the basis $\{\be, \bff, \bh\}$ is 
\begin{equation}
\Ad_{g_0}=
\begin{bmatrix}
x_0 & 0 & 0\\ 
0 & \frac1{x_0} & 0 \\
0 & 0 & 1
\end{bmatrix}
\cdot
\begin{bmatrix} 
1 & 0 &  0\\
-{x_0}{y_0}  & 1   & 2 \sqrt{-{x_0}{y_0}} \\
\sqrt{-{x_0}{y_0}} & 0 & 1
\end{bmatrix}
= \begin{bmatrix} 
x_0 & 0 &  0 \\
-{y_0} & \frac{1}{x_0}   & 2 \sqrt{-\frac{y_0}{x_0}} \\
\sqrt{-{x_0}{y_0}} & 0  & 1
\end{bmatrix}.
\end{equation} 
Thus, 
\begin{equation}
\begin{aligned}
&(\Ad_{g_0})^\ast (dx\wedge dz)= 
\left|
\begin{matrix} 
x_0  & 0   \\
\sqrt{-x_0y_0} & 0
\end{matrix}
\right| dx\wedge dy 
- 
\left|
\begin{matrix} 
x_0   & 0  \\
\sqrt{-x_0y_0}  & 1
\end{matrix}
\right| dx\wedge dz 
+\left|
\begin{matrix} 
 0 & 0  \\
 0  & 1
\end{matrix}
\right| dy\wedge dz\\
&= x_0 dx\wedge dz.
\end{aligned}
\end{equation}

By definition, the volume form at $X_0\in \fg$ is $(\Ad_{g_0^{-1}})^\ast(\omega_{\be})$, and thus we get 
$$\omega_{X_0}=\frac{1}{x_0}dx\wedge dz.$$

Coverting to $(t, \theta)$-coordinates, we get: 

\begin{align*}
\rho^*(dx \wedge dz) &= ( (\cos \theta+1)dt - t \sin \theta  \ d \theta) \wedge (\sin \theta \ dt + t \cos \theta \ d \theta) \\ &= t \cos \theta (\cos \theta+1 ) dt \wedge d \theta - t \sin^2 \theta d \theta \wedge dt \\ &=(t \cos^2 \theta + t \cos \theta + t \sin^2 \theta ) dt \wedge d \theta \\ &= t (1+ \cos \theta) dt \wedge d \theta
\end{align*}

and therefore
\begin{align*}
\rho^* \omega &= \rho^* \left(\frac{4 dx \wedge dz }{x }\right) =   
\frac{ 4t (1+ \cos \theta) }{t (1+ \cos \theta)} dt \wedge d \theta\\  &= 4 \ dt \wedge d \theta.
\end{align*}

\subsubsection{Semisimple orbits in $\fsl_2(\R)$} The orbits of split semisimple elements in $\fsl_2(\R)$ are hyperboloids of one sheet asymptotically approaching the nilpotent cone on the outside; the orbits of elliptic elements are the individual sheets of hyperboloids of two sheets that lie inside the same asymptotic cone (see e.g., \cite[\S 2.3.3]{debacker:clay} for detail). 
Measures on them can be computed in a similar way (we return to this calculation below).
For now we compute another example, 
a semisimple orbit in $\mathfrak{so}_3(\R)$.

\subsection{Another example: a semisimple (elliptic) element in $\mathfrak{so}_3(\R)$}\label{sub:so3R}
Let $G=\SO(3)$, $\mathfrak{g}=\mathfrak{so}(3)$, and let $\{X, H, Y \}$ be the basis for $\mathfrak{g}$ given by:

$$X = \begin{pmatrix} 0 & 0 & 0 \\ 0 & 0 & -1 \\ 0 & 1 & 0 \end{pmatrix} \ \ H = \begin{pmatrix} 0 & 0 & 1 \\ 0 & 0 & 0 \\ -1 & 0 & 0 \end{pmatrix} \ \ Y  = \begin{pmatrix} 0 & -1 & 0 \\ 1 & 0 & 0 \\ 0 &0 & 0 \end{pmatrix} $$
Let $\{X^\ast, H^\ast, Y^\ast \}$ be dual basis for $\mathfrak{g}^*$  under the Killing form. 
Explicitly this means $X^*(X)=-2$, $Y^*(Y)=-2$, and $H^\ast(H)=-2$. Denote by $\ri_{H^\ast}$ the co-adoint orbit of $H^\ast$. Then a brief calculation shows that $\ri_{H^\ast}$ is the unit sphere: 
$$ \ri_{H^\ast} = \{ (x,y,z) \in \mathfrak{g}^* \ \vert \ x^2+y^2+z^2=1 \} $$
We compute Kirillov's form on this sphere, using a slightly different method from the above (to illustrate various approaches to such computations). Namely, rather than computing the form at one fixed base point on the orbit and then using the group action to compute it at all points, we do the computation directly for each point $X_0$ of our orbit $\ri_{H^\ast}$. 

As above, given  
$X_0\in \mathcal{O}_{H^\ast}$,
 we have  $\varphi_{X_0}:G \to \mathfrak{g}^* $ given by $\varphi_{X_0}(g) = Ad^*_{g}(X_0)$. 
 We write $X_0=(x_0,y_0,z_0)$ in $\{X^\ast, H^\ast, Y^\ast\}$-coordinates. 
As above, the differential of $\varphi_{X_0}$ at $e \in G$ gives an identification $\Phi_{X_0}:\mathfrak{g} / \mathfrak{c}(X)  \to T_{X_0}\ri_{H^\ast} \hookrightarrow{} T_{X_0}\mathfrak{g}^* \cong \mathfrak{g}^* $. This can be computed either by recalling that $\langle \Phi_{X_0}(X), Y\rangle  = \langle X_0, [Y,X] \rangle$ or via the exponential map:
\begin{align*} \Phi_{X_0}(X) = \frac{d}{dt} \big\vert_{t=0} (\varphi_{X_0} \circ \exp)(tX). 
\end{align*}
The result of this computation is that with respect to  
our coordinates, the matrix representation for $\Phi_{X_0}$ is given by:
$$\begin{pmatrix}0 & z_0 & -y_0 \\ -z_0 & 0 & x_0  \\ y_0 & -x_0 & 0  \end{pmatrix}.$$
Write  $\omega=f_1 dX^* \wedge dH^* + f_2 dX^* \wedge dY^* + f_3 dH^* \wedge dY^*$.  
We have:
\begin{align*}
&f_1(X_0) = \omega_{X_0}(\frac{\partial}{\partial X^*}, \frac{\partial}{\partial H^*})  = \langle X_0, [ (0,z_0,-y_0), (-z_0,0,x_0) ] \rangle = -2(x_0^2z_0 + y_0^2z_0 + z_0^3) = -2z_0\\
&f_2(X_0) = \omega_{X_0}(\frac{\partial}{\partial X^*}, \frac{\partial}{\partial Y^*})  = \langle X_0, [ (0,z_0,-y_0), (y_0,-x_0,0) ] \rangle = 2(x_0^2y_0 + y_0^3 + y_0z_0^2) = 2y_0\\
&f_3(X_0) = \omega_{X_0}(\frac{\partial}{\partial H^*}, \frac{\partial}{\partial Y^*})  = \langle X_0, [ (-z_0,0,x_0), (y_0,-x_0,0) ] \rangle = -2(x_0^3 +x_0 y_0^2 + x_0z_0^2) = -2x_0.
\end{align*}
Therefore,
\begin{equation}\label{eq:omega1}
 \begin{aligned}
 \omega(x,y,z) &=  -2 x \ dH^* \wedge dY^* +2y \  dX^*\wedge dY^* - 2z \  dX^* \wedge dH^* \\ &= -2 ( x \ dH^* \wedge dY^* -y \  dX^* \wedge dY^* + z\  dX^* \wedge dH^* ). \end{aligned}
 \end{equation} 
It is easy to check that if we  parametrize the sphere using the spherical coordinates, this form is rewritten as twice the usual surface area element: 
$\omega(\varphi, \theta)=2\sin\varphi$. We leave this check as an exercise. 

\subsubsection{General semi-simple orbits in ${\mathfrak{so}}_3(\R)$ and $\fsl_2(\R)$} 
Since the group $\SO_3(\R)$ is  compact, all its maximal tori are conjugate; consequently, every semisimple element in ${\mathfrak {so}}_3(\R)$ is conjugate to $rH$ for some $r\in \R$ (and all semisimple orbits are spheres). It is clear that if we replace $H^\ast$ with $rH^\ast$, the form in (\ref{eq:omega1}) gets scaled by $r$: 
$\omega_{rH^\ast} =2r dx\wedge dz$, so it is again the natural area element on a sphere of radius $r$. 

We also note that all our calculations for these algebraic volume forms are valid over any field of characteristic different from $2$ (the only reason we were working over the reals is the nice geometric picture and the intuitive parametric equations for the orbits as surfaces; note that despite our use of these 
transcendental parametrizations, in the end all the differential forms are algebraic). 

Returning to semi-simple orbits in $\fsl_2(\R)$, we can use the accidental isomorphism: over $\C$,  $\fsl_2$ and $\mathfrak{s0}_3$ are isomorphic. Thus, the same calculation as above shows also that the value at the element $t\bh$ of the  Kirillov form on the orbit 
 of $t{\bh}\in \SL_2(\R)$ is $2t dx\wedge dy$ (note that the $y$- and $z$-coordinates are swapped in \S \ref{sub:sl2R} and 
 \S \ref {sub:so3R}); this uniquely determines the invariant form on the orbit. 
 
How does this form relate to the volume form $\omega_c^\geom$ defined by (\ref{eq:geom_lie1}) in \S \ref{sub:lie}? 
Using the coordinates of \ref{sub:sl2R}, Chevalley map is given by 
$$\left[\begin{smallmatrix} z/2 & x \\ y & -z/2\end{smallmatrix}\right] \mapsto -\frac{z^2}4-xy. $$
The geometric measure is defined as a quotient: 
$dx\wedge dy \wedge {\frac12 dz} = \omega_c^\geom \wedge dc$, where $c= -\frac{z^2}4-xy$. 
Evaluating all the forms at the point $t\bh$ (which corresponds to $x=y=0$, $z=2t$, and thus $c=-t^2$), we see that 
$\omega_c^\geom$ must satisfy 
$$(\omega_c^\geom)_{t\bh}\wedge (-2t dt) = dx\wedge dy\wedge dt,$$ 
and therefore,   
$(\omega_c^\geom)_{t\bh}=-\frac1{2t}dx\wedge dy$. 
We obtain the conversion coefficient between  Kirillov's form and the geometric form:  on the orbit of a split semi-simple element $t\bh$ it is $-\frac1{4t^2}=-D(t\bh)^{-1}$. 
It would be interesting to find this coefficient for a general reductive Lie algebra.

\bibliographystyle{amsalpha}
\bibliography{latest_biblio}

\newcommand{\etalchar}[1]{$^{#1}$}
\def\polhk#1{\setbox0=\hbox{#1}{\ooalign{\hidewidth
  \lower1.5ex\hbox{`}\hidewidth\crcr\unhbox0}}}
\providecommand{\bysame}{\leavevmode\hbox to3em{\hrulefill}\thinspace}
\providecommand{\MR}{\relax\ifhmode\unskip\space\fi MR }
\providecommand{\MRhref}[2]{%
  \href{http://www.ams.org/mathscinet-getitem?mr=#1}{#2}
}
\providecommand{\href}[2]{#2}
\begin{thebibliography}{AAG{\etalchar{+}}19}

\bibitem[AAG{\etalchar{+}}19]{achter-altug-garcia-gordon}
Jeffrey~D. Achter, Salim~Ali Altug, Luis Garcia, Julia Gordon, and {with
  Appendix by Thomas R\"ud and Wen-Wei Li}, \emph{Counting {A}belian varieties
  over finite fields via {F}robenius densities},
  https://arxiv.org/abs/1905.11603 (2019).

\bibitem[AW15]{achterwilliams15}
Jeffrey Achter and Cassandra Williams, \emph{Local heuristics and an exact
  formula for abelian surfaces over finite fields}, Canad. Math. Bull.
  \textbf{58} (2015), no.~4, 673--691. \MR{3415659}

\bibitem[Bat99]{batyrev:calabi-yau}
Victor~V. Batyrev, \emph{Birational {C}alabi-{Y}au {$n$}-folds have equal
  {B}etti numbers}, New trends in algebraic geometry ({W}arwick, 1996), London
  Math. Soc. Lecture Note Ser., vol. 264, Cambridge Univ. Press, Cambridge,
  1999, pp.~1--11. \MR{1714818 (2000i:14059)}

\bibitem[Bit11]{bitan}
Rony~A. Bitan, \emph{The discriminant of an algebraic torus}, J. Number Theory
  \textbf{131} (2011), no.~9, 1657--1671.

\bibitem[BLR80]{bosch-lutkebohmert-raynaud:NeronModels}
Siegfried Bosch, Werner L\"utkebohmert, and Michel Raynaud, \emph{{N\'eron
  models}}, Springer-Verlag, Berlin, 1980.

\bibitem[Bou85]{Bourbaki:commalg}
N.~Bourbaki, \emph{El\'em\'ents de math\'ematique. alg\`ebre commutative},
  Masson, 1985.

\bibitem[Bou02]{Bourbaki:Lie}
\bysame, \emph{Lie groups and lie algebras. chapters 4-6.}, Springer-Verlag,
  2002.

\bibitem[CL08]{cluckers-loeser}
Raf Cluckers and Fran{\c {c}}ois Loeser, \emph{Constructible motivic functions
  and motivic integration}, Invent. Math. \textbf{173} (2008), no.~1, 23--121.

\bibitem[De{B}05]{debacker:clay}
Stephen De{B}acker, \emph{Homogeneity for reductive {$p$}-adic groups: an
  introduction}, Harmonic analysis, the trace formula, and {S}himura varieties,
  Clay Math. Proc., vol.~4, Amer. Math. Soc., Providence, RI, 2005,
  pp.~393--522. \MR{2192014 (2006m:22016)}

\bibitem[DKS]{davidetal15}
Chantal David, Dimitris Koukoulopoulos, and Ethan Smith, \emph{Sums of {E}uler
  products and statistics of elliptic curves}.

\bibitem[DL01]{DL.arithm}
Jan Denef and Fran\c{c}ois Loeser, \emph{Definable sets, motives, and $p$-adic
  integrals}, J. Amer. Math. Soc. \textbf{14} (2001), no.~2, 429--469.

\bibitem[FH91]{fulton-harris:RepresentationTheory}
William Fulton and Joe Harris, \emph{Representation theory: A first course},
  Graduate Texts in Mathematics, vol. 129, Springer, New York, 1991.

\bibitem[FLN10]{langlands-frenkel-ngo}
Edward Frenkel, Robert Langlands, and B{\'a}o~Ch{\^a}u Ng{\^o}, \emph{Formule
  des traces et fonctorialit\'e: le d\'ebut d'un programme}, Ann. Sci. Math.
  Qu\'ebec \textbf{34} (2010), no.~2, 199--243. \MR{2779866 (2012c:11240)}

\bibitem[FT91]{frohlich-taylor}
A.~Fr\"ohlich and M.J. Taylor, \emph{Algebraic number theory}, Cambridge
  studies in advanced mathematics, vol.~27, Cambridge University Press, 1991.

\bibitem[Gek03]{gekeler03}
Ernst-Ulrich Gekeler, \emph{Frobenius distributions of elliptic curves over
  finite prime fields}, Int. Math. Res. Not. (2003), no.~37, 1999--2018.

\bibitem[GG99]{gan-gross:haar}
Wee~Teck Gan and Benedict~H. Gross, \emph{Haar measure and the {A}rtin
  conductor}, Transactions of the AMS \textbf{351} (1999), no.~4, 1691--1704.

\bibitem[GGPS16]{gelfand-graev}
I.M. Gelfand, M.I. Graev, and I.~Pyatetskii-{S}hapiro, \emph{Generalized
  functions: volume 6. representation theory and automorphic functions}, AMS
  Chelsea publishing, 2016, Originaly published in Russian, 1958.

\bibitem[Gro97]{gross:motive}
Benedict~H. Gross, \emph{On the motive of a reductive group}, Invent. Math.
  \textbf{130} (1997), {287--313}.

\bibitem[Hum72]{humphreys}
James~E. Humphreys, \emph{Introduction to {L}ie algebras and representation
  theory}, Graduate Texts in Mathematics, Springer-Verlag, 1972.

\bibitem[KL06]{knightly-li}
Andrew Knightly and Charles Li, \emph{Traces of {H}ecke operators}, Mathematica
  surveys and monographs, vol. 133, American Math. Society, 2006.

\bibitem[Kot82]{kottwitz82}
Robert~E. Kottwitz, \emph{Rational conjugacy classes in reductive groups}, Duke
  Math. J. \textbf{49} (1982), no.~4, 785--806. \MR{683003 (84k:20020)}

\bibitem[Kot97]{kottwitz:isocrystals-2}
Robert Kottwitz, \emph{{Isocrystals with additional structure. II}}, Compositio
  Math. \textbf{109} (1997), no.~3, {255--339}.

\bibitem[Kot05]{kottwitz:clay}
Robert~E. Kottwitz, \emph{Harmonic analysis on reductive {$p$}-adic groups and
  {L}ie algebras}, Harmonic analysis, the trace formula, and {S}himura
  varieties, Clay Math. Proc., vol.~4, Amer. Math. Soc., Providence, RI, 2005,
  pp.~393--522. \MR{2192014 (2006m:22016)}

\bibitem[Lan13]{Langlands:2013aa}
Robert~P. Langlands, \emph{Singularit\'es et transfert}, Ann. Math. Qu\'e.
  \textbf{37} (2013), no.~2, 173--253. \MR{3117742}

\bibitem[Mil08]{milne:ClassFieldTheory}
James~S. Milne, \emph{Class field theory}, Available at
  \url{http://jmilne.org}, 2008.

\bibitem[Oes82]{Oesterle}
Joseph Oesterl\'e, \emph{R\'eduction modulo $p\sp{n}$ des sous-ensembles
  analytiques ferm\'es de ${{Z}}\sp{N}\sb{p}$}, Invent. Math. \textbf{66}
  (1982), no.~2, 325--341.

\bibitem[Ono61]{ono:arithmetic_tori}
Takashi Ono, \emph{Arithmetic of algebraic tori}, Ann. of Math. (2) \textbf{74}
  (1961), 101--139.

\bibitem[Ono63]{ono_tamagawa_tori}
\bysame, \emph{On the {T}amagawa number of algebraic tori}, Ann. of Math. (2)
  \textbf{78} (1963), 47--73. \MR{0156851}

\bibitem[Ono66]{ono:boulder}
\bysame, \emph{On {T}amagawa numbers}, Arithmetic properties of algebraic
  groups. Ad\`ele groups., Proc. Symp. Pure Math., vol. 9, II, Amer. Math.
  Soc., Providence, RI, 1966, pp.~122--132.

\bibitem[PR91]{platonov-rapinchuk:AlgGroupsAndNT}
Vladimir Platonov and Andrei Rapinchuk, \emph{Algebraic groups and number
  theory}, Academic Press, Inc., San Diego, 1991.

\bibitem[R{\"u}d20]{rud:tori}
Thomas R{\"u}d, \emph{Explicit {T}amagawa numbers for certain tori over number
  fields}, https://arxiv.org/abs/2009.04431 (2020).

\bibitem[Ser81]{serre:chebotarev}
Jean-Pierre Serre, \emph{Quelques applications du th\'eor\`eme de densit\'e de
  {C}hebotarev}, Inst. Hautes \'Etudes Sci. Publ. Math. (1981), no.~54,
  323--401.

\bibitem[Shy77]{shyr77}
Jih~Min Shyr, \emph{On some class number relations of algebraic tori}, Michigan
  Math. J. \textbf{24} (1977), no.~3, 365--377. \MR{0491596}

\bibitem[Vey92]{Veys:measure}
W.~Veys, \emph{Reduction modulo $p^n$ of $p$-adic subanalytic sets}, Math.Proc.
  Cambridge Philos. Soc. \textbf{112} (1992), 483--486.

\bibitem[Wei82]{weil:adeles}
Andr{\'e} Weil, \emph{Adeles and algebraic groups}, Progress in mathematics,
  vol.~23, Birkh\"auser, 1982.

\bibitem[Yun13]{yun:dedekind}
Zhiwei Yun, \emph{Orbital integrals and {D}edekind zeta funcions}, The Legacy
  of {S}rinivasa {R}amanujan, RMS-Lecture notes series (2-13), no.~20,
  399--420.

\end{thebibliography}

\end{document}